\documentclass{amsart}
\usepackage[a4paper,width={16cm},left=2.5cm,bottom=3cm, top=3cm]{geometry}

\usepackage{hyperref}
\hypersetup{
    colorlinks=true,
    linkcolor=blue,
    citecolor=blue,
    urlcolor=blue,
}

\makeatletter
\newcommand\org@maketitle{}
\newcommand\@authors{}
\let\org@maketitle\maketitle
\def\maketitle{%
	\let\@authors\authors
	\nxandlist{; }{ and }{; }\@authors
	\hypersetup{
		linktocpage=true,
		pdftitle={\@title},
                pdfauthor={\@authors},
                pdfsubject={\subjclassname. \@subjclass},
		pdfkeywords={\@keywords}
	}%
	\org@maketitle
}
\makeatother

\usepackage{amssymb}
\usepackage{rsfso}
\usepackage{mathtools}
\usepackage{microtype}
\usepackage[alphabetic,msc-links]{amsrefs}
\usepackage{enumitem}
\usepackage{amsfonts}
\usepackage{subfig}
\usepackage{tikz}

\usepackage{doi}
\renewcommand{\PrintDOI}[1]{\doi{#1}}

\numberwithin{equation}{section}

\newtheorem{theorem}{Theorem}[section]
\newtheorem{lemma}[theorem]{Lemma}
\newtheorem{corollary}[theorem]{Corollary}
\newtheorem{proposition}[theorem]{Proposition}
\theoremstyle{definition}
\newtheorem{definition}[theorem]{Definition}

\theoremstyle{remark}
\newtheorem{remark}[theorem]{Remark}

\newcommand{\e}{\varepsilon}

\newcommand{\R}{\mathbb{R}}

\def\XXint#1#2#3{{\setbox0=\hbox{$#1{#2#3}{\int}$}
    \vcenter{\hbox{$#2#3$}}\kern-.5\wd0}}

\mathchardef\ordinarycolon\mathcode`\:
\mathcode`\:=\string"8000
\begingroup \catcode`\:=\active
  \gdef:{\mathrel{\mathop\ordinarycolon}}
\endgroup

\author{Alessandra De Luca}
\address{Alessandra De Luca\newline\indent
Dipartimento di Matematica ``Giuseppe Peano''
\newline\indent
Universit\`a degli Studi di Torino
\newline\indent
Via Carlo Alberto 10, 10124, Torino, Italy   }
\email{a.deluca@unito.it}
\author{Veronica Felli}
\address{Veronica Felli\newline\indent
Dipartimento di Matematica e Applicazioni
\newline\indent
Università degli Studi di Milano-Bicocca
\newline\indent
Via Cozzi 55, 20125, Milano, Italy
}
\email{veronica.felli@unimib.it}
\author{Stefano Vita}
\address{Stefano Vita\newline\indent
Dipartimento di Matematica ``Giuseppe Peano''
\newline\indent
Universit\`a degli Studi di Torino
\newline\indent
Via Carlo Alberto 10, 10124, Torino, Italy   }
\email{stefano.vita@unito.it}

\date{May 21, 2024}

\title[Unique continuation from conical boundary points for fractional equations]{Unique continuation from conical boundary points for fractional equations}


\allowdisplaybreaks

\begin{document}

\begin{abstract}
We provide fine asymptotics of solutions of fractional elliptic equations at boundary points where the domain is locally conical; that is, corner type singularities appear. Our method relies on a suitable smoothing of the corner singularity and an approximation scheme, which allow us to provide a Pohozaev type inequality. Then, the asymptotics of solutions at the conical point follow by an Almgren type monotonicity formula, blow-up analysis and Fourier decomposition on eigenspaces of a spherical eigenvalue problem. A strong unique continuation principle follows as a corollary.

\end{abstract}

\maketitle

\section{Introduction and main results}
The present paper continues the program started in \cites{DelFel21,DelFelVit22} on strong unique continuation properties at the boundary for fractional elliptic equations, see also \cite{DelFelSic23} for the case of spectral fractional Laplacians. In particular, here we are interested in precise local asymptotics of solutions at boundary points where the domain presents corner type singularities, resembling those considered in \cite{FelFer13} for the local case; the effect of a Hardy-type potential, with a singularity at the vertex of the cone,  is also considered.

The subject of unique continuation for elliptic operators is strongly related to the problem of asymptotic behavior of solutions and classification of blow-up profiles; this connection becomes particularly relevant when dealing with  fractional problems, as emerged from \cite{FalFel14}. 
Unique continuation principles for second order elliptic equations 
have been addressed in the literature through two different types of approaches: a first one dates back to Carleman \cite{Ca1939} and relies on weighted a priori inequalities, whereas a second one, due to Garofalo and Lin \cite{GarLin86}, makes use of doubling properties derived from the Almgren monotonicity formula, see \cite{almgren}. 
The Almgren frequency function, defined as the ratio of local energy over mass near a fixed point, possesses an intrinsically local nature. Therefore, the development of an Almgren-type approach for nonlocal operators, associated with fractional elliptic equations, requires a localization of the problem; this is achievable, for example, through the  Caffarelli-Silvestre \cite{CafSil07} extension, characterizing the fractional Laplacian as the Dirichlet-to-Neumann map in one extra spatial dimension. Nevertheless, while doubling conditions suffice for establishing unique continuation in the local case, they only ensure unique continuation for the extended local problem and not for the fractional one. 
Such difficulty was overcome in \cite{FalFel14}  through an accurate  blow-up analysis and a precise classification of the possible blow-up limit profiles, in terms of a Neumann eigenvalue problem on the half-sphere. 

Deriving  monotonicity formulas around boundary points, and consequently proving  unique continuation from the boundary, poses greater difficulties, because of possible loss of regularity and inevitable interference with the geometry of the domain.
This issue arises in the study of boundary unique continuation also in the local case, which has been investigated in \cites{AEK95,AE97,DelFel21,FelSic22,KN98,FelFer13} by monotonicity methods; see also \cite{DFV2000}, where unique continuation for elliptic equations with non-homogeneous Neumann boundary conditions in cones was studied. 
We mention  \cite{KriLi23} for  the study of the asymptotic behavior of solutions to homogeneous Dirichlet problems  at boundary points with tangent cones, without monotonicity formulas.

For fractional equations and the corresponding extension problem, unique continuation from the boundary has been addressed in \cite{DelFelVit22} for the restricted fractional Laplacian, and in \cite{DelFelSic23} for the spectral fractional Laplacian.

Let $\Omega\subset \R^N$ be an  open set, $N\geq 2$. Let us consider weak solutions to the following class of fractional elliptic equations with a singular homogeneous potential
\begin{equation}\label{eq1}
\begin{cases}
(-\Delta)^su-\frac{\lambda}{|x|^{2s}}u=hu &\mathrm{in} \ \Omega\\
u=0 &\mathrm{in} \ \R^N\setminus \Omega,
\end{cases}
\end{equation}
where $(-\Delta)^s$ is the fractional Laplacian of order $s\in(0,1)$, see \eqref{eq:def-fract-lapl1}--\eqref{eq:def-fract-lapl2}, and 
\begin{equation}\label{eq:ipo-h}
h\in W^{1,p}(\Omega) \quad \text{for some }p>\frac{N}{2s}.
\end{equation}
We assume that $0\in\partial\Omega$ is a conical boundary point in the sense that, up to rotations and dilations, $\{x\in\Omega:|x|<1\}=\{x\in \mathcal C:|x|<1\}$, where $\mathcal C$ is a cone; that is, denoting as $x=(x',x_N)\in\R^{N-1}\times\R$ the variable in $\R^N$, 
\begin{equation}\label{ipocone}
\mathcal C=\mathcal C^\varphi:=\{(x',x_N)\in\R^{N-1}\times\R:x_N<\varphi(x')\},
\end{equation}
with $\varphi:\R^{N-1}\to\R$ satisfying
\begin{equation}\label{cone_property}
\begin{cases}
\varphi(x')=|x'| \,\varphi\Big(\frac{x'}{|x'|}\Big), \quad \varphi(0)=0,\\
g:=\varphi_{|_{\mathbb S^{N-2}}}\in C^{1,1}.
\end{cases}
\end{equation}
In particular, $\varphi$ is a $1$-homogeneous function; that is, equivalently,
\begin{equation*}\label{Euler}
\nabla_{x'}\varphi(x')\cdot x'=\varphi(x'),
\end{equation*}
by Euler's theorem. The cone $\mathcal C$ is an hypographical
cone in the sense that it lives below the graph of $\varphi$. Of course, its complement $\R^N\setminus \mathcal C$ is an \emph{epigraphical} cone. Moreover,  $\partial \mathcal C^\varphi\setminus \{x'\in\R^N:|x'|<r\}$ 
has the same $C^{1,1}$-regularity of $g$ for any $0<r<1$.

The fact that the potential $\lambda |x|^{-2s}$  has the same homogeneity order as the operator $(-\Delta)^s$  makes the value of $\lambda$ affect the order of vanishing of solutions at $0$.
With the aim of investigating the combined effect of the cone's amplitude and the Hardy-type potential on the asymptotic behavior of solutions at the cone's vertex, we assume the following bound on the coefficient $\lambda$:
\begin{equation}\label{assumptonlambda}
\lambda< \Lambda_{N,s}(\mathcal C),
\end{equation} 
where  $\Lambda_{N,s}(\mathcal C)$ is defined in \eqref{lambdaC} as the best constant in a fractional Hardy-type inequality on cones. 
In particular, condition \eqref{assumptonlambda} ensures semi-boundedness from below of the operator at the left hand side of \eqref{eq1}.
In  Subsection \ref{subHardy} we  discuss some properties of this optimal constant, in particular its monotonicity  with respect to inclusion of cones, and, consequently, its relation to the well known optimal constant $\Lambda_{N,s}(\R^N)$ on the whole of $\R^N$ provided in \cite{Her77}, see \eqref{costantesututtoRn}.

The precise definition of weak solutions to \eqref{eq1} is given in Section \ref{sec:weak}. To state our main results, we anticipate here that by weak solution we mean a function $u$, belonging to   the 
homogeneous fractional space
 $\mathcal{D}^{s,2}(\mathbb{R}^N)$ defined in Subsection \ref{sec:weak}, such that $u=0$ in $\R^N\setminus \Omega$ and
 \begin{equation}\label{weakformula}
    \int_{\mathbb{R}^N} |\xi|^{2s} \mathcal{F}u(\xi) \overline{\mathcal{F}v(\xi)} = \int_{\Omega} \left(\frac{\lambda}{|x|^{2s}}+ h\right) u v \, dx\quad \text{for every $v\in C^\infty_c(\Omega)$},
\end{equation}
with $\mathcal F$ denoting the unitary Fourier transform.
As is often done in fractional problems, see \cite{CafSil07}, one can reformulate \eqref{eq1} via Dirichlet-to-Neumann maps, by adding a spatial variable $t\in\R$ and convoluting $u$ with the Poisson kernel of the upper half space $\mathbb{R}^{N+1}_+:= \mathbb{R}^N \times (0,+\infty)$. The resulting function $U$ solves
\begin{equation}\label{eqEXT}
\begin{cases}
\mathrm{div}(t^{1-2s}\nabla U)=0 &\mathrm{in} \ \R^{N+1}_+\\
-\lim_{t\to0^+}t^{1-2s}\partial_tU=\kappa_s\left( h + \frac{\lambda}{|x|^{2s}}\right) \mathop{\rm Tr}U &\mathrm{on} \ \Omega\times\{0\}\\
\mathop{\rm Tr}U =0 &\mathrm{on} \ (\R^N\setminus \Omega)\times\{0\},
\end{cases}
\end{equation}
with $\mathop{\rm Tr}U$ denoting the trace of $U$ on $\partial\R^{N+1}_+$, see \eqref{eq:trace}.
Here the full variable in $\R^{N+1}_+$ is denoted by $z=(x,t)\in \mathbb{R}^{N+1}_+$
and $\kappa_s$ is the explicit positive constant
\begin{equation*}\label{eq:kappa-s}
    \kappa_s=\frac{\Gamma(1-s)}{2^{2s-1}\Gamma(s)}.
\end{equation*}
A weak solution to \eqref{eqEXT} must be understood as a function $U$ belonging to the weighted Beppo Levi space $\mathcal{D}^{1,2}(\mathbb{R}^{N+1}_+, t^{1-2s})$ defined in \eqref{eq:D12}, with $\mathop{\rm Tr}U$ supported in $\Omega$ and satisfying
\begin{equation}\label{extensionweak-intro}
\int_{\R^{N+1}_+} t^{1-2s} \nabla U\cdot\nabla V\, dz= \kappa_s\int_{\Omega} \left(h+\frac{\lambda}{|x|^{2s}}\right) \mathop{\rm Tr}U\,\mathop{\rm Tr}V \, dx \quad \text{for all $V\in C^\infty_c(\R^{N+1}_+\cup \Omega )$}.
\end{equation}
Since we are interested in local properties of solutions at $0$, we localize the problem as
\begin{equation}\label{eq1EXT}
\begin{cases}
\mathrm{div}(t^{1-2s}\nabla U)=0 &\mathrm{in} \ B_1^+\\
-\lim_{t\to0^+}t^{1-2s}\partial_tU=\kappa_s\left( h + \frac{\lambda}{|x|^{2s}}\right) \mathop{\rm Tr}U &\mathrm{on} \ \Omega\cap B'_1\\
\mathop{\rm Tr}U=0 &\mathrm{on} \ B'_1\setminus \Omega,
\end{cases}
\end{equation}
where $B^+_1=B_1\cap \R^{N+1}_+$ is the unit half ball in $\mathbb{R}^{N+1}_+$, with $B_1=\{z\in\R^{N+1}:|z|<1\}$ being the unit $(N+1)$-dimensional ball, and  $B'_1= B_1 \cap \{(x,t)\in\R^{N+1}:t=0\}$ is the $N$-dimensional thin ball. A weak formulation of problem \eqref{eq1EXT} can be naturally given  in the weighted Sobolev space $H^1(B_1^+,t^{1-2s})$, made of functions $U\in
L^2(B_1^+,t^{1-2s})$ such that $\nabla U\in
L^2(B_1^+,t^{1-2s})$, and endowed with the norm
\begin{equation*}
\|U\|_{H^1(B_1^+,t^{1-2s})}:=\left(\int_{B_1^+} t^{1-2s} |\nabla
U(x,t)|^2 dx\,dt+\int_{B_1^+} t^{1-2s} U^2(x,t) dx\,dt
\right)^{1/2}.
\end{equation*}
A function $U\in H^{1}(B_1^+,t^{1-2s})$ is a weak solution to \eqref{eq1EXT} if  
\begin{equation}\label{extensionweak}
\int_{B^+_1} t^{1-2s} \nabla U\cdot\nabla V\, dz= \kappa_s\int_{\Omega\cap B'_1} \left(h+\frac{\lambda}{|x|^{2s}}\right) \mathop{\rm Tr}U\mathop{\rm Tr}V\, dx \quad \text{for all $V\in C^\infty_c(B^+_1\cup (\Omega\cap B'_1) )$},
\end{equation}
where 
\begin{equation}\label{eq:traccia-H1}
    \mathop{\rm Tr}:H^{1}(B_1^+,t^{1-2s})\to H^s(B_1')
\end{equation}
is the continuous trace map provided in  \cite{Nek93} and \cite{JinLiXio14}*{Proposition 2.1}.  We observe that such a trace operator is indeed the restriction of the map in \eqref{eq:trace}; for simplicity, we denote both by $\mathop{\rm Tr}$.

In Section \ref{spherical} we establish a close connection between the limiting behaviours at $0$ of solutions to \eqref{eq1EXT} and \eqref{eq1}, and the weighted spherical eigenvalue problem with mixed boundary conditions
\begin{equation}\label{probautov}
\begin{cases}
-\mathrm{div}_{\mathbb S^{N}}\left(\theta_{N+1}^{1-2s}\nabla_{\mathbb S^{N}}v\right)=\mu\theta_{N+1}^{1-2s} v &\mathrm{in \ } \mathbb S^{N}_+\\
-\lim_{\theta_{N+1}\to0^+}\theta_{N+1}^{1-2s}\nabla_{\mathbb S^{N}}v\cdot\mathbf{e}=\kappa_s\lambda v  &\mathrm{on \ } \omega\subseteq \mathbb S^{N-1}\\
v=0  &\mathrm{on \ } \mathbb S^{N-1}\setminus\omega,
\end{cases}
\end{equation}
where $\mathbf{e}=(0,\dots,0,1)$, $\mathbb{S}^N_+:=\partial B_1 \cap \R^{N+1}_+$, $\mathbb{S}^{N-1}:= \partial \mathbb{S}^N_+=\partial B_1 \cap 
\{(x,t)\in\R^{N+1}:t=0\}$, and $\omega:= \mathcal C\cap \mathbb{S}^{N-1}$ is the spherical cap spanning the cone. We will denote the variable on $\mathbb S^N$ as $\theta=(\theta',\theta_{N+1})$, with  $\theta'=(\theta_1,...,\theta_{N})$.

In order to write the weak formulation of \eqref{probautov}, we introduce the weighted Sobolev space 
\begin{equation*}
H^{1}(\mathbb S^N_+,\theta_{N+1}^{1-2s})=\left\{\psi\in L^2(\mathbb S^N_+,\theta_{N+1}^{1-2s}): \nabla_{\mathbb S^{N}}\psi\in  L^2(\mathbb S^N_+,\theta_{N+1}^{1-2s})\right\}
\end{equation*}
where $L^2(\mathbb S^N_+,\theta_{N+1}^{1-2s})=\{\psi:\mathbb S^N_+\to\R\text{ measurable: }\int_{\mathbb S^N_+}\theta_{N+1}^{1-2s}\psi^2(\theta)\,dS<\infty\}$, endowed with the norm 
\begin{equation*}
    \|\psi\|_{H^{1}(\mathbb S^N_+,\theta_{N+1}^{1-2s})}=\left(\int_{\mathbb S^N_+}\theta_{N+1}^{1-2s}\left(
    |\nabla_{\mathbb S^{N}}\psi|^2+\psi^2\right)\,dS    
    \right)^{\!\!1/2}.
\end{equation*}
Here $dS$ denotes the volume element on $N$-dimensional spheres.
We also consider the closed subspace of $H^{1}(\mathbb S^N_+,\theta_{N+1}^{1-2s})$
\begin{equation}\label{Vomega}
    V_\omega:=\{\psi\in H^{1}(\mathbb S^N_+,\theta_{N+1}^{1-2s}) \, : \, \mathop{\rm Tr}\psi\equiv 0 \ \mathrm{in} \ \mathbb S^{N-1}\setminus \omega\},
\end{equation}
where, with a slight abuse of notation, we denote again by $\mathop{\rm Tr}$ the trace operator 
\begin{equation}\label{eq:traccia-sfera}
\mathop{\rm Tr}:H^{1}(\mathbb S^N_+,\theta_{N+1}^{1-2s})\to L^2(\mathbb S^{N-1}),
\end{equation}
see \cite{FalFel14}*{Lemma 2.2}.
A real number $\mu$ is said to be an eigenvalue of problem \eqref{probautov} if there exists $v\in V_\omega\setminus\{0\}$ such that 
\begin{equation*}
\int_{\mathbb{S}^N_+} \theta_{N+1}^{1-2s}\nabla_{\mathbb S^{N}} v \cdot \nabla_{\mathbb S^{N}}\varphi\, dS-\lambda \kappa_s \int_{\omega}\mathop{\rm Tr}v  \mathop{\rm Tr}\varphi \, dS'=\mu\int_{\mathbb{S}^N_+} \theta_{N+1}^{1-2s} v \varphi\, dS\quad \text{for every $\varphi\in V_\omega$},
\end{equation*}
$dS'$ denoting the volume element on $(N-1)$-dimensional spheres.
Proposition \ref{propautovalori} guarantees the existence of an increasing diverging sequence of real eigenvalues 
$\{\mu_{j}\}_{j\geq1}$ of problem \eqref{probautov} (depending on $\lambda$ and $\mathcal C$, and counted with their multiplicities) and an orthonormal basis $\{\psi_{j}\}_{j\geq1}$ of $L^2(\mathbb{S}^N_+, \theta_{N+1}^{1-2s})$ composed by eigenfunctions. To be more precise, 
if
\begin{equation*}
     \mu_{\bar{j}-1}<\mu_{\bar{j}}= \mu_{\bar{j}+1}= \cdots= \mu_{\bar{j}+\bar{m} -1} < \mu_{\bar{j}+\bar{m}},
\end{equation*}
that is $\bar{m}$ is the multiplicity of the eigenvalue $\mu_{\bar{j}}$ with $\bar{j}\geq 1$, the finite family $\{\psi_j\}_{j= \bar{j},\dots, \bar{j} + \bar{m}-1 }$ is chosen in such a way that it is an orthonormal basis of the eigenspace associated with the eigenvalue $\mu_{\bar{j}}$.

Following the approach developed by Garofalo and Lin in \cite{GarLin86} to prove unique continuation for second order elliptic equations, in Section \ref{sezionemonotonia} we consider a non-trivial solution $U$ to \eqref{eq1EXT} and derive a monotonicity formula for the associate Almgren frequency function 
\begin{equation}\label{Nintro}
 \mathcal N(r)= \frac{r^{2s-N}\left(\int_{B^+_r}t^{1-2s}|\nabla U|^2\,dz-\kappa_s\int_{\Omega\cap B'_r} \left(h+\frac{\lambda}{|x|^{2s}}\right)|\mathop{\rm Tr} U|
 ^2\,dx\right)}{r^{2s-N-1}\int_{\R^{N+1}_+\cap \partial B_r } t^{1-2s} U^2\, dS},
\end{equation}
which allows us to deduce the existence and finiteness of the limit $\gamma= \lim_{r\to 0^+} \mathcal N(r)$. 
We start  from this result to perform a  blow-up analysis, proving the convergence  of  the family of rescaled and renormalized solutions
\begin{equation*}
w^\tau(z):= \frac{U(\tau z)}{\sqrt{H(\tau)}},\qquad\tau>0,
\end{equation*}
to a non-trivial homogeneous limit profile,
where the height function $r\mapsto H(r)$ is the denominator of~\eqref{Nintro}.
Furthermore, in Lemma \ref{lemmaconvblow} we prove that there exists an eigenvalue $\mu_{j_0}$ of problem \eqref{probautov} such that 
\begin{equation}\label{gammauguale}
    \gamma=\gamma_{j_0},
\end{equation}
where, for every $j\in\mathbb N\setminus\{0\}$ and $\lambda<\Lambda_{N,s}(\mathcal C)$, we define
\begin{equation}\label{eq:gamma-h0}
   \gamma_{j}:= \sqrt{\left(\frac{N-2s}{2}\right)^2+\mu_{j}}-\frac{N-2s}{2}.
\end{equation}
Letting $m\in\mathbb N\setminus\{0\}$ be the multiplicity of the eigenvalue $\mu_{{j}_0} $, $j_0\geq1$ can be chosen  in such a way that 
\begin{equation}\label{eq:multiplicity}
    \mu_{j_0-1}   <  \mu_{j_0} =\mu_{j_0+1} =\cdots=\mu_{j_0+m-1} <    \mu_{j_0+m} 
\end{equation}
Letting $j_0$ and $m$ be as in \eqref{gammauguale}--\eqref{eq:gamma-h0} and \eqref{eq:multiplicity}, for every $j=j_0,\dots, j_0+m-1$ and  $\tau\in (0,1)$, we define
\begin{equation}\label{phi}
\varphi_j(\tau):= \int_{\mathbb{S}^N_+}\theta_{N+1}^{1-2s} U(\tau \theta) \psi_j(\theta)\, dS
\end{equation}
and 
\begin{equation}\label{U}
   \Upsilon_j(\tau):= \kappa_s \int_{\mathcal C\cap B'_\tau} h(x)\mathop{\rm Tr} U(x) \mathop{\rm Tr}\psi_j\left(\frac{x}{|x|}\right) \, dx. 
\end{equation}
The main results of the present paper are contained in the following two theorems, which provide  the asymptotics at $0$ of solutions to \eqref{eq1EXT} and \eqref{eq1} respectively.
\begin{theorem}\label{asimestens}
Let $U\in H^1(B^+_1,t^{1-2s})$ be a non-trivial weak solution to \eqref{eq1EXT}.
Then there exists an eigenvalue $\mu_{j_0} $ of problem \eqref{probautov} such that
\eqref{gammauguale}--\eqref{eq:gamma-h0} is satisfied, with $\gamma= \lim_{r\to 0^+} \mathcal N(r)$ and $\mathcal N$ being as in \eqref{Nintro}. Moreover, if $m\in\mathbb N\setminus\{0\}$ is the multiplicity of $\mu_{j_0} $,  $j_0\geq1$ is chosen as in \eqref{eq:multiplicity}, and  $\{\psi_j\}_{j=j_0,\dots, j_0+m-1}$ is an orthonormal basis of the eigenspace associated with $\mu_{j_0} $, then  
    \begin{equation}\label{risultatofinale-ext}
        \frac{U(\tau z)}{\tau^{\gamma}}\to |z|^\gamma\sum _{j=j_0}^{j_0+m-1}\beta_j \psi_j\left(\frac{z}{|z|}\right)\quad \text{in $H^1(B^+_1,t^{1-2s})$ as $\tau\to 0^+$},  
    \end{equation}
    where $(\beta_{j_0}, \dots, \beta_{j_0+m-1}) \in \mathbb{R}^m\setminus\{0\}$ and, 
    for every $j=j_0, \dots, j_0+m-1$,
    \begin{align}\label{betah}
     \beta_j=\frac{\varphi_j(R)}{R^\gamma} &+ \frac{N+\gamma-2s}{N+2\gamma-2s}\int_0^R t^{-N-1+2s-\gamma}\Upsilon_j(t)\, dt\\
    \notag &+\frac{\gamma R^{-N+2s-2\gamma}}{N+2\gamma-2s} \int_0^R t^{\gamma-1}\Upsilon_j(t)\, dt\quad \text{for all $R\in (0,1)$},
    \end{align}
    with $\varphi_j$ and   $\Upsilon_j$ defined  in \eqref{phi} and \eqref{U}, respectively. 
\end{theorem}

When proving the monotonicity formula for the extended problem, which underlies the proof of Theorem~\ref{asimestens} as described above, one immediately encounters the difficulty that the point, around which we construct the frequency function, is located  on the  ``boundary of the boundary'', namely on the boundary of the set $\Omega$, which in turn lies on the boundary of the half-space, once the extension is made. This difficulty had been addressed in \cite{DelFelVit22} through an approximation procedure: the $N$-dimensional region $\R^N\setminus \Omega$, where Dirichlet boundary conditions are imposed, is locally approximated  with smooth $(N +1)$-dimensional regions. Then, for  a sequence of approximating solutions, enough regularity is available to derive Pohozaev type identities and, consequently, to differentiate the  Almgren type frequency. Passing from $C^{1,1}$-domains,  as those considered in \cite{DelFelVit22}, to the conical sets we are treating in the present paper,  introduces an additional difficulty, due to  the non-regularity of the original domain. This requires the development in Section \ref{sec:approx} of a more sophisticated approximation and regularization procedure, which also smooths the corner of the cone.

Passing to traces,  Theorem \ref{asimestens} finally allows us to obtain the following result for problem \eqref{eq1}.

\begin{theorem}\label{asimu}
    Let $u\in \mathcal{D}^{s,2}(\mathbb{R}^N)$ be a non-trivial weak solution to \eqref{eq1}. Then there exist $j_0\in\mathbb N\setminus\{0\}$ and an eigenfunction $Y\in H^{1}(\mathbb S^N_+,\theta_{N+1}^{1-2s})\setminus\{0\}$ of problem \eqref{probautov} associated with the  eigenvalue $\mu_{j_0} $, such that  
    \begin{equation}\label{risultatofinale}
        \tau^{-\gamma_{j_0} }u(\tau x)\to |x|^{\gamma_{j_0} }\mathop{\rm Tr}Y\left(\frac{x}{|x|}\right)\quad \text{in $H^s(B'_1)$ as $\tau\to 0^+$},  
    \end{equation}
    where $\gamma_{j_0} $ is defined in \eqref{eq:gamma-h0}.
\end{theorem}
Theorem \ref{asimu} is proved by applying Theorem \ref{asimestens} to the Caffarelli-Silvestre extension $U=\mathcal H(u)$;  in particular, the eigenfunction $Y$ appearing in \eqref{risultatofinale} turns out to be  the term $\sum _{j=j_0}^{j_0+m-1}\beta_j \psi_j$ appearing in right hand side of \eqref{risultatofinale-ext}.

A relevant consequence of the previous asymptotics are the following strong unique continuation results from  $0\in \partial\Omega$, for solutions to \eqref{eq1EXT} and \eqref{eq1}, respectively. 
\begin{corollary}\label{corollario1}
    If $U\in H^1(B^+_1,t^{1-2s})$ is a solution to \eqref{eq1EXT} having infinite vanishing order at 0, i.e.  $U(z)=O(|z|^k)$ as $z\to 0$ for every $k\in\mathbb{N}$, then $U\equiv 0$ in $B_1^+$.
\end{corollary}
\begin{corollary}\label{corollario2}
    If $u\in \mathcal{D}^{s,2}(\mathbb{R}^N)$ is a solution to \eqref{eq1} having infinite vanishing order at 0, i.e. $u(x)=O(|x|^k)$ as $x\to 0$ for every $k\in\mathbb{N}$, then $u\equiv 0$ in $\R^N$.
\end{corollary}
    While the statement of Corollary \ref{corollario1} straightforwardly follows from Theorem \ref{asimestens} arguing by contradiction,  Corollary \ref{corollario2} can be deduced from  Theorem \ref{asimu}, once we made sure that the right-hand side in \eqref{risultatofinale} is not trivial. 
    This is true because, otherwise, the function $\Psi(z)=|z|^{\gamma_{j_0} }Y(z/|z|)$  would satisfy the equation $\mathrm{div}(t^{1-2s}\nabla \Psi)=0$ in $\mathcal C\times(0,+\infty)$, with both homogeneous Dirichlet and  Neumann conditions on $\mathcal C$. Thus, the trivial extension of $\Psi$ to $\mathcal C\times\R$  would be a solution to 
    $\mathrm{div}(|t|^{1-2s}\nabla\Psi)=0$,  violating classical unique continuation principles for elliptic operators with $A_2$ Muckenhoupt weights, see e.g.  \cites{GarLin86,TaoZhang2008}.

We observe that, for Theorem \ref{asimu}, and consequently Corollary \ref{corollario2}, to hold, it suffices to impose the condition that the solution $u$ of \eqref{eq1} vanishes in the complement of $\Omega$ in a small neighborhood of the vertex of the cone, rather than in the whole  of $\R^N\setminus \Omega$, while still requiring that
$u$ has finite energy in $\R^N$, i.e., belongs to $\mathcal{D}^{s,2}(\mathbb{R}^N)$. This is because, once the extension is made, 
only the extended problem localized in a half-ball centered at $0$ comes into play in the monotonicity and blow-up argument.

\subsection*{Structure of the paper}
In Section \ref{sec:2} we collect some preliminary results on the functional setting and useful inequalities. In particular, we provide some Hardy-type inequalities on cones 
and study the monotonicity  of  best constants with respect to inclusion of cones, providing results that we consider of independent interest. In Section \ref{sec:2} we also study the Laplace-Beltrami type spectral problem \eqref{probautov}, whose eigenvalues provide a quantization of the local behaviour of weak solutions to \eqref{eq1EXT}. In Section \ref{sec:3} we construct the regularized problems \eqref{eqappro} on suitably smoothed domains, whose solutions converge to a given weak solution of the original extended problem \eqref{eq1EXT}. This procedure allows us to prove a Pohozaev type inequality in Proposition \ref{lapoho}. In Section \ref{sezionemonotonia} we study the Almgren frequency associated to problem \eqref{eq1EXT} and develop a monotonicity argument,  which allows us to perform a blow-up analysis for scaled solutions. This finally leads to the proof of Theorems \ref{asimestens} and \ref{asimu}.

\section{Preliminaries}\label{sec:2}

In this section, we present some preliminaries regarding the weak formulation of the problem, fractional Hardy inequality on cones, and the weighted spherical eigenvalue problem with mixed boundary conditions which will play a role in the classification of blow-up profiles.

\subsection{Weak solutions}\label{sec:weak}
First, we give the notion of weak solutions to problem \eqref{eq1}. Let us recall that, for any $s\in (0,1)$, the fractional Laplacian $(-\Delta)^s $ is a non-local operator from the Schwartz space $ \mathcal{S}(\mathbb{R}^N)$ to $L^2(\mathbb{R}^N)$, defined as
\begin{equation}\label{eq:def-fract-lapl1}
  (-\Delta)^s u (x) = C_{N,s} \lim_{\varepsilon\to 0^+} \int_{\{y\in \R^N:|y-x|>\e\}}\frac{u(x)-u(y)
  }{|x-y|^{N+2s}}\, dy\quad \text{for every $x\in \mathbb{R}^N$},
\end{equation}
where $C_{N,s}$ is a positive constant depending on $N$ and $s$. Equivalently, $(-\Delta)^s$ can be defined by means of the unitary Fourier transform $\mathcal{F}$ on $\mathbb{R}^N$, as follows: 
\begin{equation}\label{eq:def-fract-lapl2}
    (-\Delta)^s u = \mathcal{F}^{-1}(|\xi|^{2s}\mathcal{F}u(\xi)).  
\end{equation} 
Let $\mathcal{D}^{s,2}(\mathbb{R}^N)$ be the completion of $C_c^\infty(\mathbb{R}^N)$ with respect to the norm 
\begin{equation*}
    \Vert u\Vert_{\mathcal{D}^{s,2}(\mathbb{R}^N)}= \left(\int_{\mathbb{R}^N} |\xi|^{2s} |\mathcal{F}u(\xi)|^2\, d\xi\right)^{1/2}.
\end{equation*}
We observe that $\mathcal{D}^{s,2}(\mathbb{R}^N)$ is a Hilbert space endowed with  the scalar product 
\begin{equation*}
    (u,v)_{\mathcal{D}^{s,2}(\mathbb{R}^N)}= \int_{\mathbb{R}^N} |\xi|^{2s} \mathcal{F}u(\xi) \overline{\mathcal{F}v(\xi)}, \quad \text{$u,v \in \mathcal{D}^{s,2}(\mathbb{R}^N)$}.
\end{equation*}
Moreover, $\mathcal{D}^{s,2}(\mathbb{R}^N)$ functions satisfy the following  Sobolev inequality: 
\begin{equation}\label{eq:sobolev}
    \Vert u \Vert_{L^{2^\ast(s)} (\mathbb{R}^N)}\leq S_{N,s} \Vert u \Vert_{\mathcal{D}^{s,2}(\mathbb{R}^N)}\quad \text{for every $u\in \mathcal{D}^{s,2}(\mathbb{R}^N)$}.
\end{equation}
Here $S_{N,s}$ is a positive constant depending only on $N$ and $s$.

The fractional Laplacian can be extended in a natural way as a linear and bounded operator from $\mathcal{D}^{s,2}(\mathbb{R}^N)$ to its dual $(\mathcal{D}^{s,2}(\mathbb{R}^N))^{\ast}$ as follows:
\begin{equation*}
   _{(\mathcal{D}^{s,2}(\mathbb{R}^N))^{\ast}}\langle(-\Delta)^s u, v\rangle_{\mathcal{D}^{s,2}(\mathbb{R}^N)}=(u,v)_{\mathcal{D}^{s,2}(\mathbb{R}^N)}  \quad \text{for every $u,v\in \mathcal{D}^{s,2}(\mathbb{R}^N)$}.
\end{equation*}
Thus, a weak solution to \eqref{eq1} is a function $u\in \mathcal{D}^{s,2}(\mathbb{R}^N)$ such that $u=0$ in $\R^N\setminus \Omega$ and
\begin{equation}\label{eq:weak-form}
     _{(\mathcal{D}^{s,2}(\mathbb{R}^N))^{\ast}}\langle(-\Delta)^s u, v\rangle_{\mathcal{D}^{s,2}(\mathbb{R}^N)} = \int_{\Omega} \left(\frac{\lambda}{|x|^{2s}}+ h\right) u v \, dx\quad \text{for every $v\in C^\infty_c(\Omega)$},
\end{equation}
i.e.  \eqref{weakformula} holds. We observe that the right hand side of \eqref{eq:weak-form} is well defined in view of the fractional Hardy type inequalities discussed in Subsection \ref{subHardy}, assumption \eqref{eq:ipo-h}, and \eqref{eq:sobolev}.

Let us define the functional space $\mathcal{D}^{1,2}(\mathbb{R}^{N+1}_+, t^{1-2s})$ as the completion of $C_c^\infty(\overline{\mathbb{R}^{N+1}_+})$ with respect to the norm
\begin{equation}\label{eq:D12}
    \Vert U\Vert_{\mathcal{D}^{1,2}(\mathbb{R}^{N+1}_+, t^{1-2s})}=\left(\int_{\mathbb{R}^{N+1}_+}t^{1-2s}|\nabla U|^2\, dz\right)^{1/2}.
\end{equation}
It is well known that there exists a linear and bounded trace operator
  \begin{equation}\label{eq:trace}
  \mathop{\rm Tr}:\mathcal D^{1,2}(\R^{N+1}_+,t^{1-2s})\to \mathcal
D^{s,2}(\R^N).
\end{equation}
The extension of some $u\in \mathcal{D}^{s,2}(\mathbb{R}^N)$, in the sense of  Caffarelli and Silvestre \cite{CafSil07}, is the unique solution $U=\mathcal H(u)$ to the minimization problem
\begin{equation}\label{problemadiminimo}
    \min_{\substack{U\in \mathcal{D}^{1,2}(\mathbb{R}^{N+1}_+, t^{1-2s})\\\mathrm{Tr}U=u}}\left\{\int_{\mathbb{R}^{N+1}_+}t^{1-2s}|\nabla U|^2\, dz\right\},
\end{equation}
which weakly solves 
\begin{equation*}
\begin{cases}
\mathrm{div}(t^{1-2s}\nabla U)=0 &\mathrm{in}  \ \mathbb{R}^{N+1}_+\\
-\lim_{t\to0^+}t^{1-2s}\partial_tU=\kappa_s(-\Delta )^s u &\mathrm{on} \ \mathbb{R}^N\times\{0\}\\
\mathop{\rm Tr}U=u &\mathrm{on} \ \mathbb{R}^N\times\{0\},
\end{cases}
\end{equation*}
i.e. 
\begin{equation}\label{eq:debo-est}
\int_{\R^{N+1}_+} t^{1-2s} \nabla U\cdot\nabla V\, dz= \kappa_s
\phantom{a}_{(\mathcal{D}^{s,2}(\mathbb{R}^N))^{\ast}}\langle(-\Delta)^s u, \mathop{\rm Tr}V\rangle_{\mathcal{D}^{s,2}(\mathbb{R}^N)}\quad \text{for all $V\in \mathcal{D}^{1,2}(\mathbb{R}^{N+1}_+, t^{1-2s})$}.
\end{equation}
Then \eqref{extensionweak-intro} is the weak formulation of \eqref{eqEXT}.

\subsection{A fractional Hardy inequality on cones}\label{subHardy}

In this section we prove some fractional Hardy inequalities on cones and discuss monotonicity properties of best constants with respect to the inclusion of cones.

In general, a subset $\mathcal C$ of $\R^N$ is a cone if $\alpha x\in \mathcal C$ for every $x\in \mathcal C$ and $\alpha>0$.

\begin{definition}[Cone spanned by $\omega$]
Given $\omega\subseteq \mathbb S^{N-1}$, a relatively open subset of $\mathbb S^{N-1}$, we define the open cone $\mathcal C(\omega)$ spanned by $\omega$ as
\begin{equation*}
\mathcal C(\omega):=\{r\theta'  : r>0, \ \theta'\in\omega\}.
\end{equation*}
\end{definition}
We observe that $\mathcal C=\mathcal C(\omega)$ with $\omega=\mathcal C\cap\mathbb S^{N-1}$.

Let  $\mathcal C\subseteq\R^N$ be  an open cone and let  $\omega=\mathcal C\cap\mathbb S^{N-1}$. Here, the case $\mathcal C=\R^N$ with $\omega=\mathbb S^{N-1}$ is included. Let us consider the minimization  problem
\begin{equation}\label{lambdaC}
\Lambda_{N,s}(\mathcal C):= \inf_{\varphi\in C^\infty_c(\mathcal C)\setminus\{0\}}\frac{\int_{\R^N} |\xi|^{2s}|\mathcal F\varphi(\xi)|^2\, d\xi}{\int_{\mathcal C}|x|^{-2s}|\varphi(x)|^2\, dx}= \inf_{\substack{
\phi\in H_{\mathcal C}\\\mathop{\rm Tr}\phi\not \equiv0}}\frac{\int_{\R^{N+1}_+} t^{1-2s}|\nabla\phi|^2\, dx \,dt}{\kappa_s\int_{\mathcal C}|x|^{-2s}|\mathop{\rm Tr}\phi|^2\, dx},
\end{equation}
where $H_{\mathcal C}=\{\phi\in\mathcal D^{1,2}(\R^{N+1}_+,t^{1-2s}) :  \mathop{\rm Tr}\phi\equiv0 \ \mathrm{in} \ \R^N\setminus \mathcal C\}$.
e observe that the equivalence of the two minimization problems above directly follows from \eqref{problemadiminimo} and \eqref{eq:debo-est}.
The infimum in \eqref{lambdaC} is the best constant in  the following Hardy-trace inequality:
\begin{equation}\label{HardyCone}
\kappa_s\Lambda_{N,s}(\mathcal C)\int_{\mathcal C}\frac{|\mathop{\rm Tr}\phi|^2}{|x|^{2s}}\, dx\leq \int_{\R^{N+1}_+} t^{1-2s}|\nabla\phi|^2\, dx \,dt\quad\text{for all }\phi\in H_{\mathcal C}.
\end{equation}
In the case $\mathcal C=\R^N$, the constant is explicitly given in \cite{Her77} and equals 
\begin{equation}\label{costantesututtoRn}
  \Lambda_{N,s}(\R^N)=2^{2s}\frac{\Gamma^2\left(\frac{N+2s}{4}\right)}{\Gamma^2\left(\frac{N-2s}{4}\right)}>0.  
\end{equation}
We observe that the best constant $\Lambda_{N,s}(\mathcal C)$ is decreasing with respect to inclusion of cones; indeed, by inclusion of spaces, for any two open cones $\mathcal C_1\subseteq \mathcal C_2\subseteq\R^N$, we have
\begin{equation}\label{eq:inclusioncones}
   \Lambda_{N,s}(\mathcal C_2)\leq \Lambda_{N,s}(\mathcal C_1). 
\end{equation}
e observe that \eqref{eq:inclusioncones} and \eqref{costantesututtoRn} imply that 
\begin{equation*}
    \Lambda_{N,s}(\mathcal C)>0\quad\text{for every open cone $\mathcal C\subseteq\R^N$.}
\end{equation*}
We will prove in Proposition \ref{strettainclu} that  the inequality in \eqref{eq:inclusioncones} is indeed strict whenever $\overline{\mathcal C_1}\setminus\{0\}\subset \mathcal C_2$. The proof of the latter fact strongly relies on an equivalent formulation of \eqref{lambdaC} as a minimization problem on the upper half sphere $\mathbb S^{N}_+$, as shown in the following lemma. We refer to \cite{Ter96} for the analogous result for classical Hardy's inequalities.

\begin{lemma}\label{lemmaminimo}
Let   $\mathcal C\subseteq\R^N$ be an open cone and $\omega=\mathcal C\cap\mathbb S^{N-1}$. Then
\begin{equation}\label{lambdaC2}
\Lambda_{N,s}(\mathcal C)= \min_{\psi\in V_\omega\setminus\{0\}}\frac{\int_{\mathbb S^N_+} \theta_{N+1}^{1-2s}(|\nabla_{\mathbb S^N}\psi|^2+\left(\frac{N-2s}{2}\right)^2\psi^2)\, dS}{\kappa_s\int_{\omega}|\mathop{\rm Tr}\psi|^2\, dS'},
\end{equation}
where $V_\omega$ is defined in \eqref{Vomega}.
\end{lemma}
\begin{proof}
Let us define 
\begin{equation}\label{eq:CNs}
m_{N,s}(\mathcal C):= \inf_{\psi\in V_\omega\setminus\{0\}}\frac{\int_{\mathbb S^N_+} \theta_{N+1}^{1-2s}(|\nabla_{\mathbb S^N}\psi|^2+\left(\frac{N-2s}{2}\right)^2\psi^2)\, dS}{\kappa_s\int_{\omega}|\mathop{\rm Tr}\psi|^2\, dS'}.
\end{equation}
We first prove that  $m_{N,s}(\mathcal C)$ is attained by a function $\overline\psi\in V_\omega$ (so that the above infimum is indeed  a minimum) and then that 
\begin{equation}\label{uguaglianza}
  m_{N,s}(\mathcal C)=\Lambda_{N,s}(\mathcal C).  
\end{equation}
The fact that  $m_{N,s}(\mathcal C)$ is attained follows by the compactness of the trace operator
\eqref{eq:traccia-sfera} (see Remark \ref{rem:compact-trace} below) and the weak closeness of $V_\omega$ in $H^1(\mathbb{S}^N_+, \theta_{N+1}^{1-2s})$. 
In order to prove \eqref{uguaglianza}, we first show that $\Lambda_{N,s}(\mathcal C)\leq m_{N,s}(\mathcal C)$. To this aim, we consider any function $\psi\in C^\infty(\overline{\mathbb{S}^N_+})$ such that $\mathop{\rm supp }\psi
\cap\partial\mathbb{S}^N_+\subseteq \omega$. For any 
$f\in C^\infty_c(0, +\infty)$ such that $f\not\equiv 0$, 
we define the function $v(r\theta)= f(r)\psi(\theta)$ for every $r>0$ and $\theta\in \mathbb{S}^N_+$. 
Rewriting \eqref{HardyCone} for such $v$, we have 
\begin{align}\label{dadivid}
    \kappa_s \Lambda_{N,s}(\mathcal C) &\left(\int_0^{+\infty}r^{N-1-2s}f^2(r)\, dr\right) \left(\int_\omega \psi^2(\theta',0)\, dS'\right)\\
   \notag \leq& \left(\int_0^{+\infty} r^{N+1-2s}|f'(r)|^2\, dr\right)\left(\int_{\mathbb{S}^N_+}
    \theta_{N+1}^{1-2s}\psi^2(\theta)\, dS\right)\\
    \notag &+ \left(\int_0^{+\infty}r^{N-1-2s}f^2(r)\, dr\right)\left(\int_{\mathbb{S}^N_+}
    \theta_{N+1}^{1-2s}|\nabla _{\mathbb{S}^N}\psi(\theta)|^2\, dS\right).
 \end{align}
Since 
\begin{equation}\label{HLP}
    \inf_{f\in C^\infty_c(0, +\infty)}\frac{\int_0^{+\infty} r^{N+1-2s}|f'(r)|^2\, dr}{\int_0^{+\infty}r^{N-1-2s}f^2(r)\, dr}=\left(\frac{N-2s}{2}\right)^2
\end{equation}
(see \cite{HarLitPol52}*{Theorem 330}), from \eqref{dadivid} and  a density argument it follows that 
\begin{equation*}
    \Lambda_{N,s}(\mathcal C) \leq \frac{\int_{\mathbb{S}^N_+}\theta_{N+1}^{1-2s}|\nabla _{\mathbb{S}^N}\psi(\theta)|^2\, dS+ \left(\frac{N-2s}{2}\right)^2\int_{\mathbb{S}^N_+}\theta_{N+1}^{1-2s}\psi^2(\theta)\, dS }{\kappa_s\int_\omega |\mathop{\rm Tr}\psi|^2\, dS'}
\end{equation*}
for every $\psi\in V_\omega$. The inequality  $\Lambda_{N,s}(\mathcal C)\leq m_{N,s}(\mathcal C)$ is thereby proved. 

It remains to prove the reverse inequality $\Lambda_{N,s}(\mathcal C)\geq m_{N,s}(\mathcal C)$. To this aim, we consider  any function $U\in C^\infty_c(\overline{\mathbb{R}^{N+1}_+})$ such that $\mathop{\rm supp }U\cap \partial\R^{N+1}_+\subseteq \mathcal C$. Passing to  polar coordinates, we have  
\begin{multline}\label{otherineq}
    \int_{\mathbb{R}^{N+1}_+}t^{1-2s} |\nabla U|^2\,dz\\= \int_{\mathbb{S}^N_+} \theta_{N+1}^{1-2s}\left(\int_0^{+\infty} r^{N+1-2s}|\partial _r U(r\theta)|^2\, dr + \int_0^{+\infty} r^{N-1-2s}|\nabla _{\mathbb{S}^N}U(r\theta)|^2\, dr\right) dS.
\end{multline}
Notice that, by \eqref{HLP}
\begin{multline}\label{HLP2}
  \int_{\mathbb{S}^N_+}  \theta_{N+1}^{1-2s}\left(\int_0^{+\infty}r^{N+1-2s}|\partial _r U(r\theta)|^2\, dr\right)dS\\
  \geq \left(\frac{N-2s}{2}\right)^2 \int_{\mathbb{S}^N_+}\theta_{N+1}^{1-2s}\left(\int_0^{+\infty}r^{N-1-2s}|U(r\theta)|^2\, dr\right) dS.
\end{multline}
Thus plugging \eqref{HLP2} into \eqref{otherineq}, from the definition of $m_{N,s}(\mathcal C)$ it follows that
\begin{equation*}
    \begin{split}
\int_{\mathbb{R}^{N+1}_+} t^{1-2s} |\nabla U|^2\,dz &\geq \int_0^{+\infty} r^{N-1-2s}\left(\int_{\mathbb{S}^N_+}\theta_{N+1}^{1-2s} \left(|\nabla _{\mathbb{S}^N}U(r\theta)|^2+\left(\frac{N-2s}{2}\right)^2 U^2(r\theta)\right)\, dS\right)\, dr\\
&\geq  m_{N,s}(\mathcal C) \int_0^{+\infty}r^{N-1-2s}\left(\kappa_s \int_\omega |\mathop{\rm Tr}U(r\theta')|^2\, dS'\right) dr
= m_{N,s}(\mathcal C) \kappa_s\int_{\mathcal C}\frac{|\mathrm{Tr}U|^2}{|x|^{2s}}\, dx.
    \end{split}
\end{equation*}
Notice that above we used that the function $\theta\mapsto U(r\theta)$ belongs to $C^\infty_c(\overline{\mathbb{S}^N_+})$ with
trace supported in $\omega$. A density argument and the definition of $\Lambda_{N,s}(\mathcal C)$ given in \eqref{lambdaC} lead us to $\Lambda_{N,s}(\mathcal C)\geq m_{N,s}(\mathcal C)$. 
\end{proof}

\begin{remark}\label{rem:compact-trace}
    We observe that the compactness of the trace map \eqref{eq:traccia-sfera}, which is crucial to guarantee the attainability  of the infimum \eqref{eq:CNs},  can be deduced by combining the continuity of the trace operator \eqref{eq:traccia-H1} and the compactness of the Sobolev embedding $H^s(B'_1)\hookrightarrow L^2(B'_1)$ (see \cite{DinPalVal12}*{Theorem 7.1}). More precisely, let us consider a sequence $\{\psi_n\}_n$ bounded  in $H^1(\mathbb{S}^N_+, \theta_{N+1}^{1-2s})$. Let $f\in C^\infty_c(0,+\infty)$,  $f \not \equiv 0$. Defining  $v_{n}(r\theta) = f(r) \psi_{n}(\theta)$ for every $r\in (0,1)$ and a.e. $\theta\in\mathbb{S}^N_+$, by direct computations  the sequence $\{v_n\}$ turns out to be  bounded in $H^1(B^+_1, t^{1-2s})$, so that the sequence $\{\mathop{\rm Tr}v_n\}$ is bounded in $H^s(B'_1)$ by the continuity of the operator  \eqref{eq:traccia-H1}. 
Hence, by compactness of the embedding $H^s(B'_1)\hookrightarrow L^2(B'_1)$,  there exists a subsequence $\{v_{n_k}\}$ such that  $\{\mathop{\rm Tr}v_{n_k}\}$ converges (hence being a Cauchy sequence) in  $L^2(B'_1)$. It follows that
\begin{equation*}
\int_{B_1'}|   \mathop{\rm Tr}v_{n_k}-\mathop{\rm Tr}v_{n_j}|^2\,dx=\left(\int_0^1 r^{N-1}f^2(r)\,dr\right)\left(\int_{\mathbb S^{N-1}}| \mathop{\rm Tr}\psi_{n_k}-\mathop{\rm Tr}\psi_{n_k}|^2\,dS'\right)\mathop{\longrightarrow}\limits_{k,j\to\infty}0,
\end{equation*}
hence $\{\mathop{\rm Tr}\psi_{n_k}\}$ is a Cauchy sequence in $L^2(\mathbb S^{N-1})$, thus converging in  
$L^2(\mathbb S^{N-1})$.
\end{remark}

As a  consequence of Lemma \ref{lemmaminimo}, we have 
\begin{equation}\label{fract}
\kappa_s\Lambda_{N,s}(\mathcal C)\int_{\omega}|\mathop{\rm Tr}\psi|^2\, dS'\leq \left(\frac{N-2s}{2}\right)^2 \int_{\mathbb{S}^N_+}\theta_{N+1}^{1-2s}\psi^2\, dS+ \int_{\mathbb{S}^N_+}\theta_{N+1}^{1-2s}|\nabla_{\mathbb{S}^N}\psi|^2\, dS \quad\text{for all }  \psi\in V_\omega.
\end{equation}

Another relevant consequence of Lemma \ref{lemmaminimo} is the following result.
\begin{proposition}\label{strettainclu}
Let us consider two open cones $\mathcal C_1,\mathcal C_2\subseteq\R^N$ such that $\overline{\mathcal C_1}\setminus\{0\}\subset \mathcal C_2$. Then $$\Lambda_{N,s}(\mathcal C_2)<\Lambda_{N,s}(\mathcal C_1).$$
\end{proposition}
\begin{proof}
Let us suppose by contradiction that $\Lambda_{N,s}(\mathcal C_2)=\Lambda_{N,s}(\mathcal C_1)$. Since $\overline{\mathcal C_1}\setminus\{0\}\subset \mathcal C_2$, then $\overline{\omega_1}\subset \omega_2$, where $\omega_i=\mathcal C_i\cap \mathbb S^{N-1}$ for $i=1,2$.
Thanks to Lemma \ref{lemmaminimo},  there exists $\psi\in V_{\omega_1}\setminus\{0\}\subset 
V_{\omega_2}\setminus\{0\}$ attaining the minimum in \eqref{lambdaC2} with $\mathcal C=\mathcal C_1$, which is equal to $\Lambda_{N,s}(\mathcal C_1)=\Lambda_{N,s}(\mathcal C_2)$; hence  $\psi$ also attains the minimum in \eqref{lambdaC2} with $\mathcal C=\mathcal C_2$. Thus, $\psi$  solves, for both $i=1,2$,
\begin{equation*}
\begin{cases}
-\mathrm{div}(\theta_{N+1}^{1-2s}\nabla_{\mathbb S^N}\psi)+\left(\frac{N-2s}{2}\right)^2\theta_{N+1}^{1-2s}\psi=0 & \mathrm{in} \ \mathbb S^N_+\\
-\lim_{\theta_{N+1}\to0^+}\theta_{N+1}^{1-2s}\nabla_{\mathbb S^N}\psi\cdot \mathbf{e}=\kappa_s\Lambda_{N,s}(C_i)\psi & \mathrm{on} \ \omega_i\\
\mathop{\rm Tr}\psi=0 & \mathrm{on} \ \mathbb S^{N-1}\setminus\omega_i.
\end{cases}
\end{equation*}
In order to get a contradiction we focus on what happens in $\omega_2\setminus \overline{\omega_1}$, which is a non empty open subset of $\mathbb S^{N-1}$: here the conditions $\mathop{\rm Tr}\psi=0$ and $\lim_{\theta_{N+1}\to0^+}\theta_{N+1}^{1-2s}\nabla_{\mathbb S^N}\psi\cdot \mathbf{e}=0$ are both satisfied, and hence the contradiction follows by unique continuation. 
Indeed, the function $\Psi(z)=|z|^{-(N-2s)/2}\psi(z/|z|)$  would satisfy the equation $\mathrm{div}(t^{1-2s}\nabla \Psi)=0$ in $\widetilde{\mathcal C}\times(0,+\infty)$, where $\widetilde{\mathcal C}=\{r\theta':r>0\text{ and }\theta'\in \omega_2\setminus\overline{\omega_1}\}$, with both homogeneous Dirichlet and  Neumann conditions on $\widetilde{\mathcal C}$. Thus, the trivial extension of $\Psi$ to $\widetilde{\mathcal C}\times\R$  would solve 
    $\mathrm{div}(|t|^{1-2s}\nabla\Psi)=0$ and   violate classical unique continuation principles for elliptic operators with $A_2$ Muckenhoupt weights, see e.g.  \cites{GarLin86,TaoZhang2008}. 
\end{proof}

\begin{lemma}
For every  $r>0$ and $\phi\in H^1(B^+_r, t^{1-2s})$ such that $\mathop{\rm Tr}\phi\equiv 0$ on $B'_r\setminus \mathcal C$,  there holds 
\begin{equation}\label{fract2}
\kappa_s\Lambda_{N,s}(\mathcal C)\int_{\mathcal C\cap B'_r}\frac{|\mathop{\rm Tr}\phi|^2}{|x|^{2s}}\, dx\leq \int_{B^+_r} t^{1-2s}|\nabla \phi|^2\, dz+\frac{N-2s}{2r}\int_{\partial^+B^+_r} t^{1-2s}\phi^2\, dS,
\end{equation}
where $\partial^+B^+_r=\R^{N+1}_+\cap B_r$.
\end{lemma}
\begin{proof}
For the proof it is sufficient to reason as in the proof of \cite{FalFel14}*{Lemma 2.5}: one can prove \eqref{fract2} first  for every $\phi\in C^\infty_c(\overline{B^+_r}\setminus (\R^N\setminus \mathcal C))$, exploiting \eqref{fract} and \cite{FalFel14}*{Lemma 2.4}, and then for every $\phi\in H^1(B^+_r, t^{1-2s})$ such that $\mathop{\rm Tr}\phi\equiv 0$ on $B'_r\setminus \mathcal C$  by a  density argument. 
\end{proof}

\subsection{Other inequalities}
For our purposes, it is first useful to remember the following Sobolev-type trace  inequality proved in \cite{FalFel14}*{Lemma 2.6}: there exists $\tilde S_{N,s}>0$ such that, for all $r>0$ and $V\in
  H^1(B_r^+,t^{1-2s})$,
  \begin{equation}\label{eq:lemma2.6FF}
    \bigg(\int_{B_r'}| \mathop{\rm Tr}(V)|^{2^*(s)}\,dx\bigg)^{\frac2{2^*(s)}}
\leq
\tilde S_{N,s}\bigg(
\frac{N-2s}{2r}\int_{\partial^+B^+_r}t^{1-2s}V^2dS+
\int_{B_r^+}t^{1-2s}|\nabla V|^2dz\bigg),
  \end{equation}
where $2^*(s)=\frac{2N}{N-2s}$ is the critical fractional Sobolev exponent.

The following lemma provides a local coercivity condition for the quadratic form associated with  equation \eqref{eq1EXT}.

\begin{lemma}\label{lemusefulineq}
For every $\alpha>0$, there exists $r_\alpha\in(0,1)$ such that, for any $r\in(0,r_\alpha]$, $f\in L^p(\mathcal C\cap B_1')$ such that $\|f\|_{L^p(\mathcal C\cap B_1')}\leq \alpha$ and $V\in H^1(B^+_r,t^{1-2s})$ such that $\mathop{\rm Tr}V\equiv 0$ on $B'_r\setminus \mathcal C$,
\begin{multline}\label{usefulineq}
\int_{B^+_r}t^{1-2s}|\nabla V|^2\, dz -\kappa_s\int_{\mathcal C\cap B'_r}\left( f+\frac{\lambda}{|x|^{2s}}\right) |\mathop{\rm Tr}V|^2\,dx+\frac{N-2s}{2r}\int_{\partial^+ B^+_r}t^{1-2s}V^2\, dS\\
\geq \tilde C\left(\int_{B^+_r}t^{1-2s}|\nabla V|^2\, dz + \frac{N-2s}{2r}\int_{\partial^+ B^+_r}t^{1-2s}V^2\, dS \right),
\end{multline}
where  $\tilde C=\tilde C(N,s,\lambda,\mathcal C)\in(0,1)$ is  a positive constant depending only on $N$, $s$, $\lambda$, and $\mathcal C$.
\end{lemma}
\begin{proof}
Let $r\in(0,1)$, $f\in L^p(\mathcal C\cap B_1')$,  and $V\in H^1(B^+_r, t^{1-2s})$ such that $\mathop{\rm Tr}V\equiv 0$ on $B'_r\setminus \mathcal C$.  By H\"{o}lder's inequality and \eqref{eq:lemma2.6FF}
we have
\begin{multline}\label{estim1}
\int_{\mathcal C\cap B'_r} |f||\mathop{\rm Tr} V|^2\,dx \leq V_N^{\frac{2sp-N}{Np}}   \Vert f \Vert _{L^p(\mathcal C\cap B'_1)} r^{\frac{2sp-N}{p}} \left(\int_{B'_r}|\mathop{\rm Tr}V|^{2^\ast(s)}\,dx\right)^{\frac{2}{2^\ast(s)}}\\
\leq V_N^{\frac{2sp-N}{Np}} \tilde{S}_{N,s}
\Vert f \Vert _{L^p(\mathcal C\cap B'_1)} r^{\frac{2sp-N}{p}} \left(\int_{B^+_r} t^{1-2s}|\nabla V|^2\, dz + \frac{N-2s}{2r}\int_{\partial^+B^+_r}t^{1-2s}V^2\,dS\right),
\end{multline}
where $V_N$ is the $N$-dimensional Lebesgue measure of $B_1'$. 
Moreover, the term involving the Hardy potential can be estimates using  \eqref{fract2}  as follows
\begin{equation}\label{estim2}
\kappa_s\lambda\int_{\mathcal C\cap B'_r} \frac{|\mathop{\rm Tr}V|^2}{|x|^{2s}} \, dx \leq \frac{\lambda}{\Lambda_{N,s}(\mathcal C)}  \left(\int_{B^+_r} t^{1-2s}|\nabla V|^2\, dz + \frac{N-2s}{2r}\int_{\partial^+B^+_r}t^{1-2s}V^2\,dS\right).
\end{equation}  
Thus, combining \eqref{estim1} and \eqref{estim2}, we obtain that, if $\alpha>0$ and $\|f\|_{L^p(B'_1\cap\mathcal C)}\leq \alpha$,
\begin{multline*}
\int_{B^+_r}t^{1-2s}|\nabla V|^2\, dz -\kappa_s\int_{\mathcal C\cap B'_r}\left( f+\frac{\lambda}{|x|^{2s}}\right) |\mathop{\rm Tr}V|^2\,dx+ \frac{N-2s}{2r}\int_{\partial^+B^+_r} t^{1-2s}V^2\, dS\\
\geq \left(1-\tfrac{\lambda}{\Lambda_{N,s}(\mathcal C)}- \kappa_s V_N^{\frac{2sp-N}{Np}} \tilde{S}_{N,s}\alpha \,r^{\frac{2sp-N}{p}}\right) \left(\int_{B^+_r} t^{1-2s}|\nabla V|^2\, dz + \tfrac{N-2s}{2r}\int_{\partial^+B^+_r}t^{1-2s}V^2\,dS\right).
\end{multline*}
Choosing  $r_\alpha\in(0,1)$ sufficiently small so that 
\begin{equation*}\label{mindi1mezzo}
\kappa_s V_N^{\frac{2sp-N}{Np}}  \tilde{S}_{N,s} \alpha \,r_\alpha^{\frac{2sp-N}{p}}<\frac12\left(1-\frac{\lambda}{\Lambda_{N,s}(\mathcal C)}\right),
\end{equation*} 
we then obtain that \eqref{usefulineq} is satisfied for all  $r\in(0,r_\alpha]$, $f\in L^p(\mathcal C\cap B_1')$ such that  $\|f\|_{L^p(\mathcal C\cap B_1')}\leq \alpha$ and $V\in H^1(B^+_r,t^{1-2s})$ such that $\mathop{\rm Tr}V\equiv 0$ on $B'_r\setminus\mathcal C$,  with $\tilde C=\frac12\big(1-\frac{\lambda}{\Lambda_{N,s}(\mathcal C)}\big)$.
\end{proof}

The following inequality will be crucial in the proof of a monotonicity formula in Section \ref{sezionemonotonia}. 
\begin{lemma}\label{maggioredi2ast}
Let $\alpha>0$ and $r_\alpha\in(0,1)$ be as in Lemma \ref{lemusefulineq}. Then, 
for every $r\in(0,r_\alpha]$, $f\in L^p(\mathcal C\cap B_1')$ such that $\|f\|_{L^p(\mathcal C\cap B_1')}\leq \alpha$, and $V\in H^1(B^+_r,t^{1-2s})$ such that $\mathop{\rm Tr}V\equiv 0$ on $B'_r\setminus \mathcal C$, we have
\begin{multline}\label{usefulineq2}
\int_{B^+_r}t^{1-2s}|\nabla V|^2\, dz -\kappa_s\int_{\mathcal C\cap B'_r}\left( f+\frac{\lambda}{|x|^{2s}}\right) |\mathop{\rm Tr}V|^2\,dx+\frac{N-2s}{2r}\int_{\partial^+ B^+_r}t^{1-2s}V^2\, dS\\
\geq \frac{\tilde C}{\tilde{S}_{N,s}}\left(\int_{\mathcal C\cap B'_r}|\mathop{\rm Tr}V|^{2^\ast(s)}\, dx\right)^{\frac{2}{2^\ast(s)}},
\end{multline}
being $\tilde C$ and  $\tilde{S}_{N,s}$ as in Lemma \ref{lemusefulineq} and  \eqref{eq:lemma2.6FF}, respectively.
\end{lemma}
\begin{proof} The statement follows as an immediate consequence of Lemma \ref{lemusefulineq} and  \eqref{eq:lemma2.6FF}. 
\end{proof}

\subsection{Spherical eigenvalues}\label{spherical}
The blow-up analysis carried out in Section \ref{sezionemonotonia}
will highlight that a precise description of the asymptotic behaviour of solutions to \eqref{eq1} at boundary conical points requires a classification of the homogeneous entire solutions to
\begin{equation}\label{eqCone}
\begin{cases}
\mathop{\rm div}(t^{1-2s}\nabla \Phi)=0 &\mathrm{in \ } \R^{N+1}_+\\
-\lim_{t\to0^+}t^{1-2s}\partial_t \Phi=\frac{\kappa_s\lambda}{|x|^{2s}}\Phi&\mathrm{on \ } \mathcal C\times\{0\}\\
\Phi=0  &\mathrm{on \ } (\R^N\setminus \mathcal C)\times\{0\},
\end{cases}
\end{equation}
for a given open cone $\mathcal C$, with $\lambda<\Lambda_{N,s}(\mathcal C)$. Homogeneous solutions are of the form 
\begin{equation*}
\Phi(z)=|z|^\gamma \Phi\left(\frac{z}{|z|}\right),
\end{equation*}
for some homogeneity degree $\gamma=\gamma(s,\lambda,\mathcal C)$. Considering the spherical cap $\omega=\mathcal C\cap \mathbb S^{N-1}$, 
the latter classification problem is equivalent to the classification of some spherical Laplace-Beltrami weighted eigenvalues, namely the eigenvalues of \eqref{probautov}.
More precisely, $\Phi$ is a $\gamma$-homogeneous solution to \eqref{eqCone} if and only if the value $\mu$, related to $\gamma$ by formulas
\begin{equation*}
\mu=\gamma(N-2s+\gamma),\quad \gamma=\sqrt{\left(\frac{N-2s}{2}\right)^2+\mu}-\frac{N-2s}{2},
\end{equation*} 
is an eigenvalue of \eqref{probautov}, with $v=\Phi\big|_{\mathbb S^{N}_+}$ as an associated eigenfunction; see \cite{FalFel14}*{Lemma 2.1}.

Let us also remark that, fixed an open cone $\mathcal C$,
the eigenfunction of \eqref{probautov} 
associated to the first eigenvalue $\mu_1=\mu_1(N,s,\lambda,\mathcal C)$
is unique up to multiplicative constants, nonnegative and related to nonnegative homogeneous solutions to \eqref{eqCone}. The classification of first eigenfunctions was studied in \cite{TerTorVit18} in case of homogeneous $s$-harmonic functions on cones; that is, when $\lambda=0$. In the latter case one has
\begin{equation*}
\mu_1=\mu_1(N,s,0,\mathcal C)\in[0,2sN]\quad\text{and}\quad 
\gamma_1(N,s,0,\mathcal C)=\sqrt{\left(\tfrac{N-2s}{2}\right)^2+\mu_1}-\tfrac{N-2s}{2}\in[0,2s],
\end{equation*}
being the extremal cases verified if and only if $\mathcal C=\R^N$ or $\mathcal C=\emptyset$, respectively. 
Moreover, $\gamma_1(N,s,0,\mathcal C)$ and $\mu_1(N,s,0,\mathcal C)$ are monotone decreasing with respect to domain inclusion and monotone increasing in $s\in(0,1)$. Homogeneous $s$-harmonic functions on cones are studied also in connection with symmetric stable processes and Martin kernels, see e.g. \cites{BanBog04,BogSiuSto15,Mic06}.

\begin{proposition}\label{propautovalori}
There exist an increasing sequence $\{\mu_j\}_{j\geq 1}\subset \big(-\big(\frac{N-2s}2\big)^2, +\infty\big)$ of real eigenvalues of 
problem \eqref{probautov}(repeated according to their finite multiplicities) such that $\mu_j\to +\infty$, and an orthonormal basis $\{\psi_j\}_{j\geq 1}$ of $L^2(\mathbb{S}^N_+, \theta_{N+1}^{1-2s})$, such that $\psi_j$ is an eigenfunction of \eqref{probautov} associated with $\mu_j$ for all $j$.
\end{proposition}
\begin{proof}
We first notice that the space $V_\omega$ defined in \eqref{Vomega} is compactly embedded in $L^2(\mathbb{S}^N_+, \theta_{N+1}^{1-2s})$ 
in view of \cite{opic-kufner}*{Theorem 19.7}.  Moreover,
the bilinear form $a\colon V_\omega\times V_\omega \to \mathbb{R}$ defined as  
\begin{equation*}
a(\varphi, \psi):= \int_{\mathbb{S}^N_+} \theta_{N+1}^{1-2s}\nabla_{\mathbb S^{N}}\varphi \cdot \nabla_{\mathbb S^{N}} \psi\, dS-\lambda \kappa_s \int_{\omega}\mathop{\rm Tr}\varphi \mathop{\rm Tr}\psi \, dS'\quad \text{for any $\varphi, \psi\in V_\omega$}
\end{equation*}
is continuous by continuity of the trace map \eqref{eq:traccia-sfera} and weakly coercive on $V_\omega$ by \eqref{fract}: indeed, for every $\psi\in V_\omega$,
\begin{align*}
a(\psi,\psi) &+ \left(\frac{N-2s}{2}\right)^2 \Vert \psi\Vert_{L^2(\mathbb{S}^N_+,\theta_{N+1}^{1-2s} )}^2\\
=&
\int_{\mathbb{S}^N_+}\theta_{N+1}^{1-2s}|\nabla_{\mathbb S^{N}} \psi|^2\, dS - \lambda \kappa_s\int_{\omega}|\mathop{\rm Tr}\psi|^2\, dS' + \left(\frac{N-2s}{2}\right)^2\int_{\mathbb{S}^N_+}\theta_{N+1}^{1-2s}|\psi|^2\, dS\\
\geq & \left(1-\frac{\lambda}{\Lambda_{N,s}(\mathcal C)}\right)\left(\int_{\mathbb{S}^N_+}\theta_{N+1}^{1-2s}|\nabla_{\mathbb S^{N}} \psi|^2\, dS+ \left(\frac{N-2s}{2}\right)^2\int_{\mathbb{S}^N_+}\theta_{N+1}^{1-2s}|\psi|^2\, dS\right). 
\end{align*}
The thesis follows by classical spectral theory, see e.g. \cite{Sal08}*{Theorem 6.16}.
\end{proof}
The first eigenvalue of the sequence $\{\mu_j\}_{j\geq 1}$ given in the previous proposition admits the following variational characterization
\begin{equation*}
\mu_1= \min_{\psi\in V_\omega\setminus \{0\}} \frac{\int_{\mathbb{S}^N_+} \theta_{N+1}^{1-2s}|\nabla_{\mathbb S^{N}}\psi|^2\, dS-\lambda \kappa_s \int_{\omega}|\mathop{\rm Tr}\psi|^2 \, dS'}{\int_{\mathbb{S}^N_+}\theta_{N+1}^{1-2s}\psi^2\, dS}.
\end{equation*}
\begin{remark}
In case $\lambda=0$ and $\mathcal{C}=\R^N$, from the blow-up analysis performed in \cite{FalFel14} and the regularity results proved in \cite{SirTerVit21a} for degenerate/singular problems arising from the Caffarelli–Silvestre extension, it easily follows that 
\begin{equation*}
\{\gamma_j(N,s,0,\R^N):j\geq 1\}={\mathbb N},\quad \{\mu_j(N,s,0,\R^N):j\geq1\}=\{k(k+N-2s):k\in{\mathbb N}\}.    
\end{equation*}
In case $\lambda=0$ and $\mathcal{C}=\R^N_+=\{x_N>0\}$, then 
\begin{equation*}
\{\gamma_j(N,s,0,\R^N_+):j\geq1\}={\mathbb N}+s,\quad 
\{\mu_j(N,s,0,\R^N_+):j\geq1\}=\{(k+s)(k+N-s):k\in{\mathbb N}\},
\end{equation*}
as pointed out in \cite{DelFelVit22}.
\end{remark}

\section{Approximation scheme and Pohozaev inequality}\label{sec:3}
In this section we develop an approximation procedure, needed to derive a Pohozaev-type inequality.

\subsection{Approximation scheme}\label{sec:approx}
Before constructing a sequence of regularized domains approximating the cone, we establish the following property of the distance function from the cone's boundary.
\begin{lemma}\label{1om}
Let $\mathcal C$ be an open cone. Then, the distance function $d(x):=\mathrm{dist}(x,\partial \mathcal C)$ is $1$-homogeneous, i.e. $d(\alpha x)=\alpha d(x)$ for all $x\in\mathcal C$ and $\alpha>0$. Moreover,
\begin{equation}\label{distance1hom}
\nabla d(x)\cdot x=d(x),\quad\text{for a.e. }x\in\mathcal C.
\end{equation}
\end{lemma}
\begin{proof}
If $x\in\mathcal C$, there exists $y\in \partial\mathcal C$ such that $|x-y|=d(x)$. Then, for every $\alpha>0$, $\alpha x\in\mathcal C$ and $\alpha y\in \partial\mathcal C$; therefore
\begin{equation*}
    d(\alpha x)=\min_{z\in \partial \mathcal C}|\alpha x-z|\leq|\alpha x-\alpha y|=\alpha d(x).
\end{equation*}
We conclude that $d(\alpha x)\leq \alpha d(x)$ for all $x\in\mathcal C$ and $\alpha>0$. The reverse inequality follows by renaming $\overline x=\alpha x\in \mathcal C$ and $\overline\alpha=1/\alpha>0$.

Since $d$ is a distance function, it is Lipschitz continuous, hence it is differentiable a.e. in $\mathcal C$. By direct computations, \eqref{distance1hom} holds in all $x$ where $d$ is differentiable.
\end{proof}

From now on, we will consider the cone $\mathcal C$ as defined in \eqref{ipocone}, for some $\varphi$ satisfying \eqref{cone_property}.
Let $n_0\in\mathbb N$ large enough to be appropriately fixed later (see Lemma \ref{lemstellaturasotto}). For every $n\geq n_0$,
we consider the following approximation of $\mathcal C$:
\begin{equation}\label{Cdelta}
\mathcal C_n:=
\{(x',x_N)\in\R^{N-1}\times\R:x_N<\psi_n(x')\},\quad\text{where}\quad  \psi_n(x'):=\frac1n+f_n(|x'|) \,\varphi\left(\frac{x'}{|x'|}\right),
\end{equation}
see Figure \ref{2figs}.
The function $f_n$ appearing in \eqref{Cdelta} is defined as
\begin{equation*}
    f_n:[0,+\infty)\to\R,\quad 
  f_n(t)= \int_{1/n^2}^t \zeta(n^2s)\,ds,
\end{equation*}
where $\zeta\in C^\infty([0,+\infty))$ is a fixed function satisfying
\begin{gather*}
    \zeta(t)=0\text{ if }t\leq 1,\quad 
    \zeta(t)=1\text{ if }t\geq 2,\quad 
0\leq\zeta(t)\leq1\text{ for all }t\in[0,+\infty),\\
\zeta'(t)\geq 0\text{ for all }t\in[0,+\infty),\quad\text{and }
\int_1^2\zeta(t)\,dt=\frac12.
\end{gather*}
We observe that $f_n\in C^\infty([0,+\infty))$ satisfies the following properties:
\begin{align}
  \label{eq:popfdelta1} 
  &f_n(t)= 0\quad\text{for every }0\leq t\leq \frac{1}{n^2}\quad \text{and}\quad f_n(t)=t-\frac3{2n^2}\quad \text{for every $t\geq \frac2{n^2}$},\\
  \label{fdeltaconv}
&-\frac 3{2n^2}\leq f_n(t)-t f'_n(t)\leq 0\quad\text{and}\quad 
|f_n(t)-t|\leq \frac 3{2n^2}
\quad \text{for every }t\geq 0.
\end{align}
The sequence of sets $\{\mathcal C_n\}_n$ approximate $\mathcal C$ in the sense that
\begin{equation}\label{eq:appro}
    \text{for every compact set $K\subset\R^N\setminus\overline{\mathcal C}$ there exists $n_K$ such that $K\subset\R^N\setminus\mathcal C_n$ for all $n\geq n_K$.}
\end{equation}
We observe that the function $\psi_n$ defined in \eqref{Cdelta} is of class $C^{1,1}$, hence $\mathcal C_n$ is a $C^{1,1}$-domain. In particular,
 $\mathcal C_n$ is a translation of the cone $\mathcal C$ with a smoothening of the vertex. Moreover, this transformation guarantees a starshapedness condition with respect to $0$ on $\partial \mathcal C_n$, as established in the following lemma.

\begin{lemma}\label{lemstellaturasotto}
Let $n_0=\lceil 6M\rceil$, where $M:=\max_{\mathbb S^{N-2}}|g|$ and  $\lceil \cdot\rceil$ denotes the ceil function. Then, for every $n\geq n_0$, 
\begin{itemize}
\item[\rm (i)] $x\cdot\nu(x)>0$ on $\partial \mathcal C_n$, where $\nu(x)$ stands for the outward unit normal vector to $\partial \mathcal C_n$ at $x$;
\item[\rm (ii)] $\mathcal C\subset\mathcal C_n$.
\end{itemize}
\end{lemma}
\begin{proof}
Let us define $F(x):=x_N-\psi_n(x')$ for every $x=(x',x_N)\in \mathbb{R}^{N-1}\times \mathbb{R}$.
By \eqref{Cdelta}, we have
\begin{equation}\label{0sulbordo}
F(x)=0 \quad \text{for every $x\in\partial \mathcal C_n$},
\end{equation}
so that $\partial \mathcal C_n$ is a level set for $F$. Therefore, at every $x\in \partial \mathcal C_n$, the outward normal vector to $\mathcal C_n$ is oriented upwards and parallel to the gradient of $F$, which is equal to
\begin{equation}\label{gradienteesplicito}
\nabla_x F(x',x_N)=\left(-f'_n(|x'|)\frac{x'}{|x'|}\varphi\left(\frac{x'}{|x'|}\right)-f_n(|x'|)\frac{1}{|x'|}\nabla _{\theta'}\varphi\left(\frac{x'}{|x'|}\right) , 1\right),
\end{equation}
in view of \eqref{Cdelta}, where $\theta'=x'/|x'|$.
By \eqref{Cdelta}, \eqref{0sulbordo} can be rewritten as
\begin{equation}\label{sostituisco}
    x_N= \frac1n+f_n(|x'|)\varphi\left(\frac{x'}{|x'|}\right) \quad \text{for every $x\in\partial \mathcal C_n$}.
\end{equation}
Using \eqref{gradienteesplicito}, \eqref{sostituisco} and the fact that $\nabla _{\theta'}\varphi$ and $x'$ are orthogonal,  we have
\begin{align}\label{stellaturasotto}
\nabla_x F(x',x_N) \cdot (x',x_N)=&\,x_N- \varphi\left(\frac{x'}{|x'|}\right) f'_n(|x'|)|x'|-f_n(|x'|) \nabla _{\theta'}\varphi\left(\frac{x'}{|x'|}\right) \cdot\frac{x'}{|x'|}\\
\notag= &\,\frac1n+ f_n(|x'|)\varphi\left(\frac{x'}{|x'|}\right)  - \varphi\left(\frac{x'}{|x'|}\right) f'_n(|x'|)|x'|\\
\notag=&\,\frac1n+ \varphi\left(\frac{x'}{|x'|}\right) \left(f_n(|x'|) -f'_n(|x'|)|x'|\right) \\
\notag\geq &\,\frac1n-\frac{3M}{2n^2} \geq\frac{3}{4n}
\end{align}
for every $n\geq n_0=\lceil 6M\rceil$,
where we have also taken advantage of \eqref{fdeltaconv}. (i) is thereby proved. 

To prove (ii) we observe that \eqref{cone_property} and \eqref{Cdelta}, together with the second estimate in \eqref{fdeltaconv}, yield
\begin{equation*}
    \psi_n(x')-\varphi(x')=\frac1n+(f_n(|x'|)-|x'|)\varphi\left(\frac{x'}{|x'|}\right)
    \geq \frac1n-\frac{3M}{2n^2}\geq \frac{3}{4n},
\end{equation*}
for every $x'\in\R^{N-1}$ and $n\geq n_0$. Since $\mathcal C$ and $\mathcal C_n$ are the subgraphs of $\varphi$ and $\psi_n$, respectively, the above estimate directly implies (ii). 
\end{proof}

\begin{figure}[ht]
\centering
\subfloat{\begin{tikzpicture}[scale=0.4]
\fill[line width=0.pt,color=white,fill=blue!30] (-7,8) -- (7,8) -- (7,-4) -- (-7,-4) -- cycle;
\fill[line width=0.pt,fill=blue!50] (-7,6.9) -- (0,-2) -- (7,6.9) -- (7,-4) -- (-7,-4) -- cycle;
\draw [shift={(-1.,2.)},line width=2.pt]  plot[domain=-2.35:-1.57,variable=\t]({1.*1.*cos(\t r)+0.*1.*sin(\t r)},{0.*1.*cos(\t r)+1.*1.*sin(\t r)});
\draw [line width=1.pt](-7,6.9) -- (0,-2) -- (7,6.9);
\draw [line width=2.pt](1.705,1.29)--(7,8);
\draw [line width=2.pt](-1.705,1.29)--(-7,8);
\draw [line width=2.pt](-1,1)--(1,1);
\draw [shift={(1.,2.)},,line width=2.pt]  plot[domain=-1.57:-0.785,variable=\t]({1.*1.*cos(\t r)+0.*1.*sin(\t r)},{0.*1.*cos(\t r)+1.*1.*sin(\t r)});
\fill[line width=0.pt,fill=white,fill opacity=1] (1.705,1.29) -- (7,8.03) -- (-7,8.03) -- (-1.705,1.29) -- cycle;
\fill[line width=0.pt,color=white,fill=white,fill opacity=1] (-1,1) -- (1,1) -- (1,2) -- (-1,2) -- cycle;
\draw [white,shift={(-1.,2.)},fill=white,fill opacity=1,line width=0.pt]  plot[domain=-3.14:3.14,variable=\t]({1.*0.99*cos(\t r)+0.*0.99*sin(\t r)},{0.*0.99*cos(\t r)+1.*0.99*sin(\t r)});
\draw [white,shift={(1.,2.)},fill=white,fill opacity=1,line width=0.pt]  plot[domain=-3.14:3.14,variable=\t]({1.*0.99*cos(\t r)+0.*0.99*sin(\t r)},{0.*0.99*cos(\t r)+1.*0.99*sin(\t r)});
\draw [line width=0.5pt,<->](0,1.5)--(1,1.5);
\node at (0.7,2.2) {\scriptsize $\frac3{2n^2}$};
\draw [line width=0.5pt,<->](1,-2)--(1,0.9);
\node at (1.5,-1) {\scriptsize $\frac1n$};
\node[blue,opacity=0.3] at (0,0) {\Large ${\mathcal C}_n$};
\node[blue] at (4,-3) {\Large $\mathcal C$};
\end{tikzpicture}}\quad
\subfloat{\begin{tikzpicture}[scale=0.4]
\draw [shift={(-1.,5.)},line width=2.pt]  plot[domain=1.57:2.355,variable=\t]({1.*1.*cos(\t r)+0.*1.*sin(\t r)},{0.*1.*cos(\t r)+1.*1.*sin(\t r)});
 \draw [line width=2.pt](1.705,5.71)--(7,-1);
 \draw [line width=2.pt](-1.705,5.71)--(-7,-1);
 \draw [line width=2.pt](-1,6)--(1,6);
\draw [shift={(1.,5.)},line width=2.pt]  plot[domain=0.785:1.57,variable=\t]({1.*1.*cos(\t r)+0.*1.*sin(\t r)},{0.*1.*cos(\t r)+1.*1.*sin(\t r)});
\fill[line width=2pt,blue!30,fill=blue!30,fill opacity=1] (1.705,5.71) -- (7,-1.02) -- (-7,-1.02) -- (-1.705,5.71) -- cycle;
\fill[line width=0.pt,color=blue!30,fill=blue!30,fill opacity=1] (-1,5) -- (1,5) -- (1,6) -- (-1,6) -- cycle;
\draw [blue!30,shift={(-1.,5.)},fill=blue!30,fill opacity=1,line width=0.pt]  plot[domain=-3.14:3.14,variable=\t]({1.*0.99*cos(\t r)+0.*0.99*sin(\t r)},{0.*0.99*cos(\t r)+1.*0.99*sin(\t r)});
\draw [blue!30,shift={(1.,5.)},fill=blue!30,fill opacity=1,line width=0.pt]  plot[domain=-3.14:3.14,variable=\t]({1.*0.99*cos(\t r)+0.*0.99*sin(\t r)},{0.*0.99*cos(\t r)+1.*0.99*sin(\t r)});
\draw [blue!30,fill=blue!30,fill opacity=1,line width=2.pt](-7,-1) -- (-7,-5.7) -- (7,-5.7) -- (7,-1);
\draw [line width=1.pt,fill=blue!50,fill opacity=1](-7,-5.8) -- (0,3) -- (7,-5.8);
\draw [line width=0.5pt,<->](0,5.6)--(1,5.6);
\node at (0.7,5) {\scriptsize $\frac3{2n^2}$};
\draw [line width=0.5pt,<->](-1,3)--(-1,6);
\node at (-1.5,4.5) {\scriptsize $\frac1n$};
\node[blue,opacity=0.3] at (2,3) {\Large ${\mathcal C}_n$};
\node[blue] at (0,-4.7) {\Large $\mathcal C$};
\end{tikzpicture}}
\caption{Sections of $\mathcal C$ (in dark blue) and of its smooth approximation $\mathcal C_n$ (in light blue).}
\label{2figs}
\end{figure}
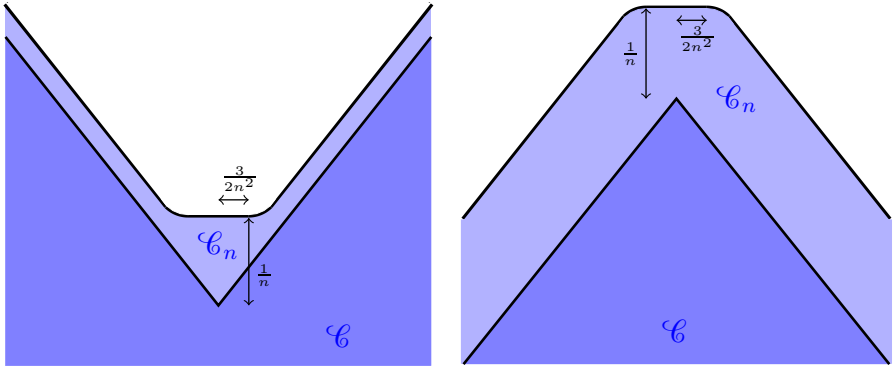
Let $h_n\in C^\infty(\overline {\mathcal C\cap B'_1})$ be such that 
\begin{equation}\label{lehnconverg}
    h_n\to h\quad \text{in $W^{1,p}(\mathcal C\cap B'_1)$.}
\end{equation}
  Let 
\begin{equation}\label{R0}
    R_0=r_{\alpha_0},
\end{equation}
being $r_{\alpha_0}$ as in Lemma \ref{lemusefulineq} with 
\begin{equation*}
    \alpha_0=\sup_n\|h_n\|_{L^{p}(\mathcal C\cap B'_1)}.
\end{equation*}
For every $n\geq n_0$, we consider the set 
\begin{equation}\label{Omegan}
\Omega_n:=\bigg\{z=(x',x_N,t)\in\R^{N+1} : \ x_N<\psi_{n}(x')+\frac n3f_n(t)\bigg\}\cap B_{R_0}^+.
\end{equation}
The topological boundary of $\Omega_n$  can be written as
\begin{equation*}
\partial \Omega_n=\overline{\sigma_n\cup\tau_n\cup\gamma_n},
\end{equation*}
where
\begin{align*}
&\sigma_n=\mathcal C_n\cap B'_{R_0}, \\
&\tau_n=\Big\{(x',x_N,t)\in\R^{N+1} : x_N<\psi_{n}(x')+\frac n3f_n(t)\Big\}\cap\partial^+B_{R_0}^+,\\
&\gamma_n=\Big\{(x',x_N,t)\in\R^{N+1} : x_N=\psi_{n}(x')+\frac n3f_n(t)\Big\}\cap B_{R_0}^+,
\end{align*}
see Figure \ref{f:Omegan}.
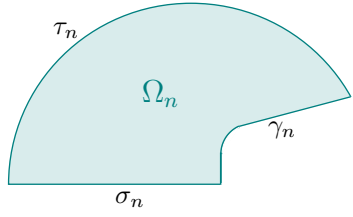
\begin{figure}[ht]
\centering
\begin{tikzpicture}[scale=0.4]
\draw[teal,fill=teal!15,line width=0.5pt] (-6,0) -- (6,0) arc(0:180:6) --cycle;
\draw [white,line width=1pt](1,0)--(1,1);
 \fill[line width=0.pt,color=white,fill=white,fill opacity=1] (1,0) -- (1,1) -- (1.58,1.9) -- (5.25,2.9) -- (6.25,2.9) -- (6.25,0)-- cycle;
\draw [shift={(2,1)},teal,line width=1.pt]  plot[domain=2:3.14,variable=\t]({1.*1.*cos(\t r)+0.*1.*sin(\t r)},{0.*1.*cos(\t r)+1.*1.*sin(\t r)});
\draw [white,shift={(2,1)},fill=white,fill opacity=1,line width=0.pt]   plot[domain=2:3.14,variable=\t]({1.*0.98*cos(\t r)+0.*1.*sin(\t r)},{0.*1.*cos(\t r)+1.*0.98*sin(\t r)});
\draw [teal,line width=0.5pt](1.58,1.92)--(5.25,2.9);
\draw [teal,line width=0.7pt](0.98,0)--(0.98,1.05);
\draw [white,line width=1pt](1,0)--(6,0);
 \node at (-2,-0.5) {$\sigma_n$};
 \node at (-4.1,5.1) {$\tau_n$};
 \node at (3,1.8) {$\gamma_n$};
 \node[teal,opacity=0.8] at (-1,3) {\Large $\Omega_n$};
\end{tikzpicture}
\caption{Section of $\Omega_n$.}
\label{f:Omegan}
\end{figure}
By construction, the approximating sets $\Omega_n$ inherit from the cones $\mathcal C_n$ the property of being star-shaped with respect to the origin, as stated in the following lemma. 

\begin{lemma}\label{starshape}
Let $n_0$ be as in Lemma \ref{lemstellaturasotto}. If $n\geq n_0$, $z\in \gamma_n$ and $\nu(z)$ denotes the outward unit normal vector at $z\in \partial\Omega_n$, then $z\cdot \nu(z)> 0$. 
\end{lemma}
\begin{proof}
For every $z=(x',x_N, t)\in \mathbb{R}^{N-1}\times \mathbb{R}\times \mathbb{R}$, we define
\begin{equation*}
    G(z):= x_N-\psi_n(x') -\frac n3f_n(t).
\end{equation*}
By the definition of $\gamma_n$,  it holds that 
\begin{equation}\label{azzeramento}
 G(z)=0 \quad \text{for every $z\in \gamma_n$}.   
\end{equation}
Thus the outward unit normal vector to $\partial\Omega_n$ at any $z=(x',x_N, t)\in\gamma_n$ is given by $$\frac{( -\nabla\psi_n(x'), 1, -\frac n3 f'_n(t))}{\sqrt{|\nabla\psi_n(x')|^2+1+\frac{n^2}9|f'_n(t)|^2}},$$ and consequently, by \eqref{azzeramento},  \eqref{Cdelta}, \eqref{stellaturasotto}, and \eqref{fdeltaconv},
\begin{equation*}
\begin{split}
\sqrt{|\nabla\psi_n(x')|^2+1+\tfrac{n^2}{9}|f'_n(t)|^2}\,(z\cdot \nu(z))&= (x',x_N,t)\cdot \Big( -\nabla\psi_n(x'),1,- \frac n3 f'_n(t)\Big) \\
&= x_N - \nabla\psi_n(x') \cdot x' -\frac n3 f'_n(t) t \\
&= \psi_n(x')- \nabla\psi_n(x') \cdot x'+\frac n3 f_n(t) -\frac n3 f'_n(t)t \\
&=\frac1n+ \varphi\left(\frac{x'}{|x'|}\right)\Big(f_{n}(|x'|)-f'_{n}(|x'|)|x'|\Big) +
\frac n3 \big(f_n(t) -f'_n(t) t\big) \\
&\geq \frac3{4n}-\frac{1}{2n}=\frac1{4n}> 0
\end{split}
\end{equation*}
for $n\geq n_0$.
\end{proof}

Once a non-trivial weak solution  $U\in H^1(B^+_1,t^{1-2s})$  to \eqref{eq1EXT} is fixed, the next step is 
 to construct a family of boundary value  problems on each $\Omega_n$ which approximate the extended problem \eqref{eq1EXT}, in the sense that the solutions $U_n$ to such approximating problems converge to $U$. To this aim,  let  $G_n\in C^\infty_c(\overline{B_1^+}\setminus (\R^N\setminus \mathcal C))$ be such that $G_n\to U$ in $H^1(B_1^+,t^{1-2s})$. It is not restrictive to assume that  $G_n$ vanishes on $\gamma_n$. Let $\eta\in C^\infty([0,+\infty))$  be  a monotone non-decreasing function such that 
 $$0\leq\eta\leq1, \quad \eta\equiv 0\ \text{in $[0,1]$},\quad \eta\equiv 1\ \text{in $[2,+\infty)$} .$$
Let
\begin{equation*}
\eta_n(x):=\eta(n d(x)), \quad x\in\mathcal C,
\end{equation*}
where  $d(x):=\mathrm{dist}(x,\partial \mathcal C)$.
Then, up to extend trivially to zero $\eta_n\big(h_n+\frac{\lambda}{|x|^{2s}}\big)$ in the whole of $\sigma_n$, we consider the following approximating boundary value problem
\begin{equation}\label{eqappro}
\begin{cases}
\mathrm{div}(t^{1-2s}\nabla U_n)=0 &\mathrm{in}\  \Omega_n\\
-\lim_{t\to0^+}t^{1-2s}\partial_tU_n=\kappa_s\eta_n \left(h_n+\frac{\lambda}{|x|^{2s}}\right) 
\mathop{\rm Tr}U _n&\mathrm{on} \ \sigma_n\\
U_n=G_n &\mathrm{on} \ \gamma_n\cup \tau_n.
\end{cases}
\end{equation}
In Lemma \ref{convdelleUn} below we prove that, for every fixed $n\in \mathbb{N}$, problem \eqref{eqappro} admits one and only one weak solution $U_n$, in the sense that 
$$
U_n \in H^1(\Omega_n,t^{1-2s}),\quad U_n= G_n \text{ on $\tau_n\cup\gamma_n$},
$$
and 
\begin{equation*}
\int_{\Omega_n} t^{1-2s}\nabla U_n \cdot \nabla \psi\,dz=\kappa_s\int_{\sigma_n}  \eta_n \left(h_n+\frac{\lambda}{|x|^{2s}}\right) \mathop{\rm Tr}U_n \mathop{\rm Tr}\psi \quad \text{for every \ } \psi\in C^\infty_c(\Omega_n\cup \sigma_n).
\end{equation*} 
The condition $U_n= G_n$ on $\tau_n\cup\gamma_n$ is understood in sense of traces; the trace on $\tau_n\cup\gamma_n$ of any $V\in H^1(\Omega_n,t^{1-2s})$ is denoted simply as $V$.

\begin{lemma}\label{convdelleUn}
For every $n\geq n_0$, problem \eqref{eqappro} admits a unique weak solution $U_n$. Moreover,   $U_n\to U$ in $H^1(B^+_{R_0},t^{1-2s})$ (\,$U_n$ being extended trivially to zero in $B^+_{R_0}\setminus \Omega_n$). 
\end{lemma}
\begin{proof}
 To prove the lemma, we turn to study the following \emph{auxiliary} problem obtained by setting $V_n=~\!\!U_n -~\!\! G_n$: $V_n\in H^1(\Omega_n,t^{1-2s})$, $V_n=0$ 
on $\tau_n\cup\gamma_n$, and 
\begin{multline}\label{senzanome}
\int_{\Omega_n}t^{1-2s} \nabla V_n\cdot \nabla \psi\, dz- \kappa_s \int_{\sigma_n}\eta_n \left(h_n +\frac{\lambda}{|x|^{2s}}\right) \mathop{\rm Tr}V_n \mathop{\rm Tr}\psi\, dx\\
= -\int_{\Omega_n}t^{1-2s} \nabla G_n\cdot \nabla \psi\, dz+ \kappa_s \int_{\sigma_n}\eta_n \left(h_n +\frac{\lambda}{|x|^{2s}}\right)\mathop{\rm Tr} G_n \mathop{\rm Tr}\psi\, dx
\end{multline}
for every $\psi\in C^\infty_c(\Omega_n\cup \sigma_n)$.

Let  $\mathcal H_n$ be the closure of $C^\infty_c(\Omega_n\cup \sigma_n)$ in $H^1(B^+_{R_0},t^{1-2s})$. We observe that the right-hand side of \eqref{senzanome}, interpreted as a linear functional on $\mathcal H_n$, is continuous and its operator norm is bounded uniformly with respect to $n$. Indeed, since $\eta_n\leq 1$, the H\"{o}lder inequality, \eqref{assumptonlambda}, and \eqref{fract2} yield 
\begin{multline}\label{funzionalecont}
\kappa_s \biggl|\int_{\sigma_n}  \frac{\eta_n\lambda}{|x|^{2s}}\mathop{\rm Tr}G_n \mathop{\rm Tr}\psi\, dx \biggr|\leq \bigg(\!\kappa_s\Lambda_{N,s}(\mathcal C)\int_{\mathcal{C}\cap B'_{R_0}} \!\!\!\frac{|\mathop{\rm Tr}G_n|^2}{|x|^{2s}}\, dx\!\bigg)^{\!\!\frac12}\bigg(\!\kappa_s\Lambda_{N,s}(\mathcal C)\int_{\mathcal{C}\cap B'_{R_0}} \!\!\!\frac{|\mathop{\rm Tr}\psi|^2}{|x|^{2s}}\, dx\!\bigg)^{\!\!\frac12}\\
\leq\bigg(\int_{B^+_{R_0}} t^{1-2s}|\nabla G_n|^2\, dz + \frac{N-2s}{2}\int_{\partial^+B^+_{R_0}}t^{1-2s} G_n^2\, dS\bigg)^{\!\!\frac12} \bigg(\int_{B^+_{R_0}} t^{1-2s}|\nabla \psi|^2\, dz \bigg)^{\!\!\frac12}
\end{multline}
and then we can exploit the boundedness of the sequence $\{G_n\}$ in $H^1(B^+_{R_0}, t^{1-2s})$ and the continuity of the trace map from $H^1(B^+_{R_0}, t^{1-2s})$ to $L^2(\partial^+B^+_{R_0}, t^{1-2s})$ to estimate from above the right-hand side of \eqref{funzionalecont} by $\mathrm{const}\Vert \nabla \psi \Vert_{L^2(B^+_1,t^{1-2s})}$, for some $\mathop{\rm const}>0$ independent of $n$.

Next, we  notice that the left-hand side of \eqref{senzanome}, seen as a bilinear form on $\mathcal H_n$, is coercive 
since $\eta_n\leq 1$ and as a consequence of Lemma \ref{lemusefulineq}. The continuity of the bilinear form can be easily derived combining the H\"{o}lder inequality and \eqref{fract2}. 

By the Lax-Milgram theorem,  problem \eqref{senzanome} thus admits a unique solution $V_n\in \mathcal H_n$ for every fixed $n\geq n_0$; consequently, also problem \eqref{eqappro} has a unique solution $U_n$. After extending each $V_n$ to $0$ in $B^+_{R_0}\setminus \Omega_n$, the stability estimates provided by the Lax-Milgram theorem ensure that  the sequence
$\{V_n\}$ is bounded in $H^1(B^+_{R_0}, t^{1-2s})$. 
Hence there exist a subsequence $\{V_{n_k}\}$ and $V\in H^1(B^+_{R_0}, t^{1-2s})$ such that 
\begin{equation}\label{Vnkdebole}
V_{n_k}\rightharpoonup V\quad \text{weakly in $H^1(B^+_{R_0}, t^{1-2s})$ as $k\to\infty$}. 
\end{equation}
Combining \eqref{senzanome}  with a density argument
and applying Lemma \ref{lemusefulineq} to $V_{n_k}$, we get the following inequality
\begin{align}\label{eQ1}
    -\int_{B^+_{R_0}}& t^{1-2s} \nabla G_{n_k}\cdot\nabla V_{n_k} \, dz + \kappa_s \int_{\mathcal{C}\cap B'_{R_0}}\eta_{n_k} \left(h_{n_k}+\frac{\lambda}{|x|^{2s}}\right)\mathop{\rm Tr}G_{n_k} \mathop{\rm Tr}V_{n_k}\, dx\\
    \notag&= \int_{B^+_{R_0}} t^{1-2s} |\nabla V_{n_k}|^2\, dz -\kappa_s \int_{\mathcal{C}\cap B'_{R_0}} \eta_{n_k} \left( h_{n_k}+\frac{\lambda}{|x|^{2s}}\right) |\mathop{\rm Tr}V_{n_k}|^2\, dx \\
    \notag &\geq \tilde C \int_{B^+_{R_0}} t^{1-2s} |\nabla V_{n_k}|^2\, dz.  
    \end{align}
On the other hand, $\mathop{\rm Tr}G_{n}\to \mathop{\rm Tr}U$ in $L^{2^\ast(s)}(\mathcal{C}\cap B'_{R_0})$ by \eqref{eq:lemma2.6FF}, $\eta_n$ to 1 a.e. in $\mathcal{C}\cap B'_{R_0}$, $h_n\to h$ in $L^{\frac{N}{2s}}(\mathcal{C}\cap B'_{R_0})$ thanks to \eqref{lehnconverg}, and $\mathop{\rm Tr}V_{n_k}\rightharpoonup \mathop{\rm Tr}V$ in $L^{2^\ast(s)}(B'_{R_0})$ as a consequence of \eqref{Vnkdebole} and \eqref{eq:lemma2.6FF}; moreover, since the map $V\mapsto |x|^{-s}\mathop{\rm Tr}V$ is continuous from $H^1(B^+_{R_0}, t^{1-2s})$ to $L^2(\mathcal{C}\cap B'_{R_0})$ by \eqref{fract2}, we have that $|x|^{-s}\mathop{\rm Tr}G_{n}\to |x|^{-s}\mathop{\rm Tr}U$ in $L^{2}(\mathcal{C}\cap B'_{R_0})$ and 
$|x|^{-s}\mathop{\rm Tr}V_{n_k}\rightharpoonup |x|^{-s}\mathop{\rm Tr}V $ weakly in $L^{2}(\mathcal{C}\cap B'_{R_0})$. Hence
\begin{multline}\label{eQ2}
        \lim_{k\to\infty}\int_{B^+_{R_0}} t^{1-2s} \nabla G_{n_k}\cdot\nabla V_{n_k} \, dz - \kappa_s \int_{\mathcal{C}\cap B'_{R_0}}\eta_{n_k} \left(h_{n_k}+\frac{\lambda}{|x|^{2s}}\right)\mathop{\rm Tr}G_{n_k} \mathop{\rm Tr}V_{n_k}\, dx\\
        = \int_{B^+_{R_0}} t^{1-2s}\nabla U\cdot\nabla V\, dz - \kappa_s \int_{\mathcal{C}\cap B'_{R_0}}\left(h+\frac{\lambda}{|x|^{2s}}\right) \mathop{\rm Tr}U \mathop{\rm Tr}V\, dx=0,
   \end{multline}
where in the last equality we have used \eqref{extensionweak}, taking into account that the limit $V$ has null trace on both $\partial ^+B^+_{R_0}$ and $B'_{R_0}\setminus\mathcal{C}$; this is a consequence of the fact that $V_{n_k}\in\mathcal H_{n_k}$ for all $k$, $\mathop{\rm Tr}V_{n_k}\rightharpoonup \mathop{\rm Tr}V$ in $L^{2^\ast(s)}(B'_{R_0})$, and \eqref{eq:appro}.
From \eqref{eQ1} and \eqref{eQ2} it follows that $\Vert \nabla V_{n_k}\Vert_{L^2(B^+_{R_0}, t^{1-2s})}\to 0$ as $k\to\infty$, which in turn combined with 
\cite{FalFel14}*{Lemma 2.4} and \eqref{Vnkdebole} allows us to conclude that $V_{n_k}\to 0$ strongly in $H^1(B^+_{R_0}, t^{1-2s})$ as $k\to\infty$. Since 
the limit actually does not depend on the subsequence, by the Urysohn's subsequence principle, we have that $V_n\to 0$ in $H^1(B^+_{R_0}, t^{1-2s})$ as $n\to\infty$. Therefore $U_n= V_n+G_n \to U$ in $H^1(B^+_{R_0}, t^{1-2s})$ as $n\to\infty$. The proof is thereby complete. 
\end{proof}
\subsection{A Pohozaev-type inequality}
In order to provide a Pohozaev-type inequality for solutions to \eqref{eq1EXT}, we fix $r\in (0,R_0)$ and $n>\frac1r$.
Let $0<\delta<\rho<\frac{1}{n^2}$. In this way
\begin{equation*}
    B_\rho^+\subset\Omega_n\cap B_r^+.
\end{equation*}
In order to proceed with the approximation argument, we consider the following domains
\begin{equation}\label{Omeganrdelta}
\Omega_{r,n,\rho,\delta} :=\Omega_n\cap 
\{(x,t)\in B_r\setminus \overline{B_\rho}:t>\delta\},
\end{equation}
where $\Omega_n$ is defined in \eqref{Omegan}. Letting
\begin{equation*}
\begin{split}
\sigma_{r,n,\rho,\delta}&:= 
\Omega_n\cap 
\{(x,t)\in B_r\setminus \overline{B_\rho}:t=\delta\},\\
\tau_{\rho,n,\delta}&:= \Omega_n\cap \{(x,t)\in \partial B_\rho:t>\delta\},\\
\tau_{r,n,\delta}&:= \Omega_n\cap \{(x,t)\in \partial B_r:t>\delta\},\\
\gamma_{r,n,\delta}&:= \partial\Omega_n\cap \{(x,t)\in B_r:t>\delta\},\\
\end{split}
\end{equation*}
we thus have 
\begin{equation*}
\partial\Omega_{r,n,\rho,\delta}= \overline{\sigma_{r,n,\rho,\delta}\cup \tau_{r,n,\delta}\cup \tau_{\rho,n,\delta}\cup \gamma_{r,n,\delta}},
\end{equation*}
see Figure \ref{f:Omeganedelta}.
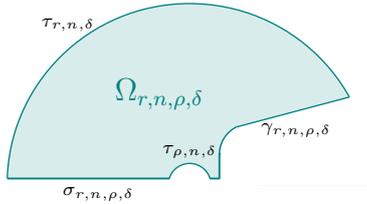
\begin{figure}[ht]
\centering
\begin{tikzpicture}[scale=0.4]
\draw[teal,fill=teal!15,line width=0.5pt] (-6,0) -- (6,0) arc(0:180:6) --cycle;
\draw[teal,fill=white,fill opacity=1,line width=0.5pt] (-0.7,0) -- (0.7,0) arc(0:180:0.7) --cycle;
\draw [white,line width=1pt](1,0)--(1,1);
 \fill[line width=1.pt,color=white,fill=white,fill opacity=1] (1,0) -- (1,1) -- (1.58,1.9) -- (5.25,2.9) -- (6.25,2.9) -- (6.25,0)-- cycle;
\draw [shift={(2,1)},teal,line width=1.pt]  plot[domain=2:3.14,variable=\t]({1.*1.*cos(\t r)+0.*1.*sin(\t r)},{0.*1.*cos(\t r)+1.*1.*sin(\t r)});
\draw [white,shift={(2,1)},fill=white,fill opacity=1,line width=0.pt]   plot[domain=2:3.14,variable=\t]({1.*0.98*cos(\t r)+0.*1.*sin(\t r)},{0.*1.*cos(\t r)+1.*0.98*sin(\t r)});
\draw [teal,line width=0.5pt](1.58,1.92)--(5.25,2.9);
\draw [teal,line width=0.7pt](0.98,0)--(0.98,1.05);
\draw [white,line width=1pt](1,0)--(6,0);
 \fill[line width=1.pt,color=white,fill=white,fill opacity=1] (-6.2,-0.2) -- (2.2,-0.2) -- (2.2,0.2) -- (-6.2,0.2)-- cycle; 
\draw [teal,line width=0.7pt](-6,0.2)--(-0.65,0.2);
\draw [teal,line width=0.7pt](0.65,0.2)--(1,0.2);
 \node at (-3,-0.3) {\scriptsize$\sigma_{r,n,\rho,\delta}$};
 \node at (-4,5.3) {\scriptsize$\tau_{r,n,\delta}$};
 \node at (3.5,1.8) {\scriptsize$\gamma_{r,n,\rho,\delta}$};
 \node at (0,1.1) {\scriptsize$\tau_{\rho,n,\delta}$};
 \node[teal,opacity=0.8] at (-1,3) {\Large $\Omega_{r,n,\rho,\delta}$};
\end{tikzpicture}
\caption{Section of $\Omega_{r,n,\rho,\delta}$.}
\label{f:Omeganedelta}
\end{figure}

\begin{proposition}[Pohozaev inequality]\label{lapoho}
Let $U\in H^1(B^+_1,t^{1-2s})$ be a weak solution to \eqref{eq1EXT}. Then, for a.e. $r\in (0,R_0)$,
\begin{align}\label{poho}
\frac{r}{2}\bigg(\int_{\partial ^+ B^+_r}& t^{1-2s} |\nabla U|^2\,dS - \kappa_s\lambda\int_{\mathcal{C}\cap\partial B'_r} \frac{\mathop{\rm Tr}U^2}{|x|^{2s}}\,dS'\bigg)- r \int_{\partial ^+ B^+_r} t^{1-2s} \left|\frac{\partial U}{\partial\nu}\right|^2\,dS\\
\notag&\quad+\frac{\kappa_s}{2} \int_{\mathcal{C}\cap B'_r} (\nabla h \cdot x+ N h )|\mathop{\rm Tr}U|^2\,dx - \frac{r\kappa_s}{2} \int_{\mathcal{C}\cap \partial B'_r} h |\mathop{\rm Tr}U|^2\, dS'\\
\notag&\geq \, \frac{N-2s}{2}\bigg(\int_{B^+_r} t^{1-2s}|\nabla U|^2\, dz-\kappa_s\lambda\int_{\mathcal{C}\cap B'_r}\frac{|\mathop{\rm Tr}U|^2}{|x|^{2s}}\,dx\bigg),
\end{align}
and 
\begin{equation}\label{poho2}
\int_{B^+_r}t^{1-2s}|\nabla U|^2\, dz - \kappa_s \int_{\mathcal{C}\cap B'_r} \left(h+\frac{\lambda}{|x|^{2s}} \right)|\mathop{\rm Tr}U|^2\, dx= \int_{\partial^+B^+_r} t^{1-2s} U\frac{\partial U}{\partial\nu}\,dS
\end{equation}
where $\nu=z/|z|$ is the outward unit normal vector to $\partial^+B^+_r$.
\end{proposition}
\begin{proof}
The idea of the proof is to derive a Pohozaev-type inequality for 
problem \eqref{eqappro} and consequently for problem \eqref{eq1EXT}, passing to the limit as $n\to \infty$ and taking advantage of Lemma \ref{convdelleUn}. 
For this, we 
start by observing that, for every fixed 
$r\in (0,R_0)$, $n>\frac1r$, and $0<\delta<\rho<\frac{1}{n^2}$,
 the solution $U_n$ to problem \eqref{eqappro}  belongs to $H^2(\Omega_{r,n,\rho,\delta})$  by classical elliptic regularity theory, see \cite{grisvard}*{Theorem 2.2.2.3}, since 
the equation holds true in a $C^{1,1}$-domain containing the set $\Omega_{r,n,\rho,\delta}$.
Therefore we are in the position to integrate over $\Omega_{r,n,\rho,\delta}$ the following Rellich-Ne\u{c}as identity
\begin{align*}
\mathop{\rm div}\big(t^{1-2s}|\nabla U_n| ^2 z-2t^{1-2s} (z\cdot \nabla U_n)\nabla U_n\big)
&=(N-2s)t^{1-2s}|\nabla U_n| ^2 - 2 (z\cdot \nabla U_n) \mathop{\rm div} (t^{1-2s}\nabla U_n)\\
&=(N-2s)t^{1-2s}|\nabla U_n| ^2 
\end{align*}
which holds true in a classical sense in $\Omega_n\cap B_r$.
Since $t^{1-2s}|\nabla U_n| ^2 z-2t^{1-2s} (z\cdot \nabla U_n)\nabla U_n\in W^{1,1}(\Omega_{r,n,\rho,\delta})$, we can apply  the divergence theorem on the Lipschitz domain $\Omega_{r,n,\rho,\delta}$  defined in \eqref{Omeganrdelta}; thus
\begin{align}\label{relnecint}
-&\int_{\sigma_{r,n,\rho,\delta}}\delta^{2-2s} |\nabla U_n|^2\, dx + 2\int_{\sigma_{r,n,\rho,\delta}} \delta^{1-2s} (x\cdot \nabla_x U_n + \delta \partial_t U_n)\partial_t U_n\,dx \\
\notag& \quad+r \int_{\tau_{r,n,\delta}} t^{1-2s} |\nabla U_n|^2\,dS - 2r\int_{\tau_{r,n,\delta}}t^{1-2s}\left|\frac{\partial U_n}{\partial\nu}\right|^2\, dS-\int_{\gamma_{r,n,\delta}}t^{1-2s}\left|\frac{\partial U_n}{\partial\nu}\right|^2(z\cdot\nu)\, dS\\
\notag&\quad -\rho \int_{\tau_{\rho,n,\delta}} t^{1-2s} |\nabla U_n|^2\,dS + 2\rho\int_{\tau_{\rho,n,\delta}}t^{1-2s}\left|\frac{\partial U_n}{\partial\nu}\right|^2\, dS\\
\notag&=(N-2s)\int_{\Omega_{r,n,\rho,\delta}} t^{1-2s}|\nabla U_n| ^2\,dz,
\end{align}
where we have used that the outward unit normal vector on $\sigma_{r,n,\rho,\delta}$  is given by $\nu=(0,\dots,0,-1)$, while for every $z\in\tau_{r,n,\delta}$ and for every $z\in\tau_{\rho,n,\delta}$ it is given  by $\nu(z)=\pm z/|z|$, respectively; moreover, along $\gamma_{r,n,\delta}$ the tangential component of the gradient of $U_n$ is zero and thus $|\nabla U_n|^2=\left|\frac{\partial U_n}{\partial\nu}\right|^2$.

Now we are going to consider the limit as $\delta\to 0^+$ in \eqref{relnecint}. First, we observe that  the term on $\gamma_{r,n,\delta}$ has a sign by Lemma \ref{starshape}. Thus, we obtain the following inequality 
\begin{multline}\label{relnecineq}
-\int _{\sigma_{r,n,\rho,\delta}}\delta^{2-2s} |\nabla U_n|^2\, dx + 2\int_{\sigma_{r,n,\rho,\delta}} \delta^{1-2s} (x\cdot \nabla_x U_n + \delta \partial_t U_n)\partial_t U_n\,dx +r \int_{\tau_{r,n,\delta}} t^{1-2s} |\nabla U_n|^2\,dS \\
- 2r\int_{\tau_{r,n,\delta}}t^{1-2s}\left|\frac{\partial U_n}{\partial\nu}\right|^2\, dS-\rho \int_{\tau_{\rho,n,\delta}} t^{1-2s} |\nabla U_n|^2\,dS + 2\rho\int_{\tau_{\rho,n,\delta}}t^{1-2s}\left|\frac{\partial U_n}{\partial\nu}\right|^2\, dS\\
\geq(N-2s)\int_{\Omega_{r,n,\rho,\delta}} t^{1-2s}|\nabla U_n| ^2\,dz.
\end{multline}
From the fact that $\int_{\Omega_n}t^{1-2s}|\nabla U_n|^2\,dz<+\infty$, it follows that,
 along a sequence $\delta_k\to 0^+$,
\begin{equation*}
-\int _{\sigma_{r,n,\rho,\delta_k}}\delta_k^{2-2s} |\nabla U_n|^2\, dx+2\int_{\sigma_{r,n,\rho,\delta_k}}\delta_k^{2-2s}|\partial_t U_n|^2\, dx\to 0\quad \text{as $k\to \infty$}.
\end{equation*}
Furthermore, we observe that
\begin{equation}\label{eq:continuity}
t^{1-2s} \partial_t U_n\quad\text{and}\quad \nabla_x U_n\quad\text{are continuous in $\overline{\Omega_n\cap(B_r\setminus B_\rho)}$}. 
\end{equation}
Indeed, the continuity of $t^{1-2s} \partial_t U_n$ and $\nabla_x U_n$ away from $\partial \R^{N+1}_+$ is ensured by classical elliptic regularity theory, whereas the continuity away from  $\gamma_n$ up to 
$\partial \R^{N+1}_+$ follows from \cite{FalFel14}*{Lemma 3.3}; finally, 
since $\eta_n$ vanishes in a neighborhood of $\partial\mathcal C_n$ and the initial section of the vertical boundary 
starting from $\partial\mathcal C_n$ is straight, in view of  the definition of $\gamma_n$ and \eqref{eq:popfdelta1}, we can apply \cite{DelFelVit22}*{Lemma A.1} to deduce the 
continuity of $t^{1-2s} \partial_t U_n$ and $\nabla_x U_n$ in a neighbourhood of the edge $\overline\gamma_n\cap \partial \R^{N+1}_+$. From \eqref{eq:continuity} and the dominated convergence theorem it follows that
\begin{align*}
\int_{\sigma_{r,n,\rho,\delta}}\delta^{1-2s}(x\cdot \nabla _x U_n)\partial_t U_n \, dx \to &-\kappa_s \int_{\mathcal{C}_n\cap(B'_r\setminus \overline{B'_\rho})}\eta_n\left(h_n+\frac{\lambda}{|x|^{2s}}\right)\mathop{\rm Tr}U_n (x\cdot\nabla _x U_n) \, dx\\
&=
-\kappa_s \int_{\mathcal{C}\cap(B'_r\setminus \overline{B'_\rho})}\eta_n\left(h_n+\frac{\lambda}{|x|^{2s}}\right)\mathop{\rm Tr}U_n (x\cdot\nabla _x U_n) \, dx
\quad \text{as $\delta\to 0$}. 
\end{align*}
So, by the absolute continuity of the Lebesgue integral, taking into account \eqref{eq:continuity} and passing to the limit as $k\to \infty$ in \eqref{relnecineq} along the sequence $\delta=\delta_k\to 0^+$, we infer that, for all $r\in (0,R_0)$, $n>1/r$ and $\rho<\frac1{n^2}$,
\begin{align}\label{relnecineq2}
-2&\kappa_s\int _{\mathcal C  \cap(B'_r\setminus \overline{B'_\rho})} \eta_n \left(h_n+\frac{\lambda}{|x|^{2s}}\right) \mathop{\rm Tr}U_n(x\cdot \nabla_x U_n)\,dx+r \int_{\Omega_{n}\cap \partial B_r} t^{1-2s} |\nabla U_n|^2\,dS\\
\notag& \quad- 2r\int_{\Omega_{n}\cap \partial B_r}t^{1-2s}\left|\frac{\partial U_n}{\partial\nu}\right|^2\, dS-\rho\int_{\Omega_{n}\cap \partial B_\rho} t^{1-2s} |\nabla U_n|^2\,dS + 2\rho\int_{\Omega_{n}\cap \partial B_\rho}t^{1-2s}\left|\frac{\partial U_n}{\partial\nu}\right|^2\, dS\\
\notag&\geq(N-2s)\int_{\Omega_{n}\cap (B_r\setminus \overline{B_\rho})} t^{1-2s}|\nabla U_n| ^2\,dz.
\end{align}
The next step is to  take the limit as $\rho\to 0^+$ in \eqref{relnecineq2}: to this purpose, we apply the divergence theorem to rewrite the integral over the set $\mathcal{C}\cap (B'_r\setminus \overline{B'_\rho})$ as follows
\begin{align}\label{pezzodatratt}
&\int_{\mathcal C \cap (B'_r\setminus \overline{B'_\rho})}\! \eta_n\bigg(\! h_n\!+\!\frac{\lambda}{|x|^{2s}} \!\bigg)\!\mathop{\rm Tr}U_n(x\!\cdot\! \nabla_x U_n)\,dx 
=\frac{1}{2} \int_{\mathcal{C}\cap (B'_r\setminus \overline{B'_\rho})}\!\! \mathop{\rm div} \bigg(\!\eta_n\bigg(\! h_n\!+\!\frac{\lambda}{|x|^{2s}} \!\bigg)|\mathop{\rm Tr}U_n|^2 x\!\bigg)dx\\
\notag&\quad-\frac{1}{2} \int _{\mathcal{C}\cap (B'_r\setminus \overline{B'_\rho})}\mathop{\rm div}\bigg(\eta_n\bigg( h_n+\frac{\lambda}{|x|^{2s}} \bigg)x\bigg)|\mathop{\rm Tr}U_n|^2\, dx\\
\notag=&\,\frac{r}{2}\int_{\mathcal{C}\cap \partial B'_r}\eta_n\left( h_n+\frac{\lambda}{|x|^{2s}}\right) |\mathop{\rm Tr}U_n|^2\,dS'-\frac{\rho}{2}\int_{\mathcal C\cap \partial B'_\rho}\eta_n\left( h_n+\frac{\lambda}{|x|^{2s}}\right) |\mathop{\rm Tr}U_n|^2\,dS'\\
\notag& \quad -\frac{1}{2} \int_{\mathcal{C}\cap (B'_r\setminus \overline{B'_\rho})}(\eta_n\nabla  h_n\cdot x + N\eta_n h_n ) |\mathop{\rm Tr}U_n|^2\, dx- \frac{N-2s}{2}\int_{\mathcal{C}\cap (B'_r\setminus \overline{B'_\rho})}\eta_n \frac{\lambda}{|x|^{2s}} |\mathop{\rm Tr}U_n|^2\, dx\\
\notag& \quad -\frac{1}{2} \int_{\mathcal{C}\cap (B'_r\setminus \overline{B'_\rho})}\left(h_n+\frac{\lambda}{|x|^{2s}}\right)|\mathop{\rm Tr}U_n|^2(\nabla \eta_n\cdot x) \, dx,
\end{align}
where we have used that $\nabla \left(\lambda |x|^{-2s}\right)\cdot x= -2s  \lambda |x|^{-2s}$.
Since $U_n\in H^1(\Omega_n, t^{1-2s})$ and $\eta_n(h_n+\lambda|x|^{-2s}) |\mathop{\rm Tr}U_n|^2 \in L^1(\mathcal C_n\cap B_r')$ in view of \cite{FalFel14}*{Lemma 2.5}, there exists a sequence $\rho_j\to 0^+$ such that  
\begin{equation}\label{rhoa0}
   \rho_j \left(\int_{\Omega_n\cap \partial B_{\rho_j}} t^{1-2s} |\nabla U_n|^2\, dS
   +\int_{\partial B'_{\rho_j}}\eta_n \left|h_n+\frac{\lambda}{|x|^{2s}}\right| |\mathop{\rm Tr}U_n|^2\, dS'   \right)\to 0\qquad \text{as $j\to \infty$}.
\end{equation}
Since $U_n\in H^1(\Omega_n,t^{1-2s})$, $\frac\lambda{|x|^{2s}} |\mathop{\rm Tr}U_n|^2 \in L^1(\mathcal C\cap B_r')$ by \cite{FalFel14}*{Lemma 2.5}, and $h_n,\eta_n\in C^\infty(\overline{\mathcal{C}\cap B'_1})$, the Lebesgue dominated convergence theorem yields, as $\rho\to 0^+$,
\begin{align*}
&\int_{\mathcal{C}\cap (B'_r\setminus \overline{B'_\rho})} 
\left(h_n+\frac{\lambda}{|x|^{2s}}\right) |\mathop{\rm Tr}U_n|^2 (\nabla \eta_n\cdot x)\, dx\to 
\int_{\mathcal{C}\cap B'_r} 
\left(h_n+\frac{\lambda}{|x|^{2s}}\right) |\mathop{\rm Tr}U_n|^2 (\nabla \eta_n\cdot x)\, dx,\\
& \int_{\Omega_n\cap (B_r\setminus \overline{B_\rho})} t^{1-2s} |\nabla U_n|^2\, dz \to \int_{\Omega_n\cap B_r} t^{1-2s}|\nabla U_n|^2\, dz,\\
 &   \int_{\mathcal{C}\cap (B'_r\setminus \overline{B'_\rho})} (\eta_n\nabla h_n\cdot x + N\eta_n h_n)|\mathop{\rm Tr}U_n|^2\, dx \to \int_{\mathcal{C}\cap B'_r} (\eta_n\nabla h_n\cdot x + N\eta_n h_n)|\mathop{\rm Tr}U_n|^2\, dx,\\
    &\int_{\mathcal{C}\cap (B'_r\setminus \overline{B'_\rho})} \eta_n \frac{\lambda}{|x|^{2s}} |\mathop{\rm Tr}U_n|^2\,dx\to \int_{\mathcal{C}\cap B'_r} \eta_n \frac{\lambda}{|x|^{2s}} |\mathop{\rm Tr}U_n|^2\,dx.
  \end{align*}
Taking into account the above convergences, \eqref{pezzodatratt}, and \eqref{rhoa0}, we can pass to the limit as $j\to \infty$ in \eqref{relnecineq2} with $\rho=\rho_j$, thus obtaining  that, for all $r\in (0,R_0)$ and $n>1/r$,
\begin{multline}\label{ultimaeq}
r\int_{\Omega_n\cap \partial B_r}\!t^{1-2s}|\nabla U_n|^2 dS -r\kappa_s\int_{\mathcal{C}\cap \partial B'_r}\!\eta_n\bigg(\!h_n+\frac{\lambda}{|x|^{2s}}\!\bigg)|\mathop{\rm Tr}U_n|^2dS'-2r\int_{\Omega_n\cap \partial B_r}\!t^{1-2s}\left|\frac{\partial U_n}{\partial\nu}\right|^2 dS\\
+\kappa_s\int_{\mathcal{C}\cap B'_r}(\eta_n \nabla  h_n\cdot x+N\eta_n h_n)|\mathop{\rm Tr}U_n|^2\, dx + \kappa_s \int_{\mathcal{C}\cap B'_r} \left(h_n+\frac{\lambda}{|x|^{2s}}\right)|\mathop{\rm Tr}U_n|^2(\nabla \eta
_n\cdot x)\,dx\\
\geq \,(N-2s) \bigg(\int_{\Omega_n\cap B_r}t^{1-2s}|\nabla U_n|^2\,dz- \lambda\kappa_s\int_{\mathcal{C}\cap B'_r}\eta
_n\frac{|\mathop{\rm Tr}U_n|^2}{|x|^{2s}}\,dx \bigg).
\end{multline}
Finally, we aim at passing to the limit as $n\to \infty$ in \eqref{ultimaeq}. Extending $U_n$ trivially to zero in $B^+_r\setminus \Omega_n$, from Lemma \ref{convdelleUn} it follows that 
\begin{equation}\label{c1}
   \int_{\Omega_n\cap B_r} t^{1-2s} |\nabla U_n|^2\,dz\to \int_{B^+_r}t^{1-2s} |\nabla U|^2\,dz\qquad \text{as $n\to \infty$. }  
\end{equation}
From this, by using the coarea formula, we infer that, up to a subsequence still denoted with $U_n$,
\begin{equation}\label{c2}
\int_{\Omega_{n}\cap \partial B_r} \!\!t^{1-2s}|\nabla U_n| ^2dS\!\to\!  \int_{\partial^+ B^+_r}\! t^{1-2s}|\nabla U| ^2dS, \quad
\int_{\Omega_{n}\cap \partial B_r}\!\! t^{1-2s}\left|\tfrac{\partial U_n}{\partial\nu}\right|^2dS\!\to\!  \int_{\partial^+ B^+_r} \!t^{1-2s}\left|\tfrac{\partial U}{\partial\nu}\right|^2dS  
\end{equation}
as $n\to \infty$, for a.e. $r\in (0,R_0)$. 
Furthermore, using \eqref{lehnconverg}, the fact that $\mathop{\rm Tr}U_n\to \mathop{\rm Tr}U$ in $L^{2^\ast(s)}(B'_{R_0})$ as a consequence of Lemma \ref{convdelleUn} and \eqref{eq:lemma2.6FF}, and  exploiting the convergence of $\eta_n$ to 1 a.e. in $\mathcal{C}\cap B'_{R_0}$, we deduce that  
\begin{equation}\label{c3}
\int_{\mathcal{C}\cap B'_r} (\eta_n \nabla h_n \cdot x + N\eta_n h_n)\,|\mathop{\rm Tr}U_n|^2\, dx \to 
\int_{\mathcal{C}\cap B'_r} (\nabla h \cdot x + Nh)  |\mathop{\rm Tr}U|^2\, dx \quad\text{as $n\to \infty$}, 
\end{equation}
and that, up to a subsequence still denoted with $U_n$, 
\begin{equation}\label{c4}
\int_{\mathcal C\cap \partial B'_r} \eta_n h_n |\mathop{\rm Tr}U_n|^2\, dS'\to \int_{\mathcal{C}\cap \partial B'_r} h  |\mathop{\rm Tr}U|^2\, dS' \quad\text{as $n\to \infty$,\quad for a.e. $r\in (0,R_0)$.}
\end{equation}
Going further, we claim that
\begin{equation}\label{convergenzadelgraddieta}
 \int_{ \mathcal{C}\cap B'_r} \left(h_n + \frac{\lambda}{|x|^{2s}}\right) |\mathop{\rm Tr}U_n|^2(\nabla \eta_n\cdot x) \, dx \to 0\quad \text{as $n\to \infty$}.
\end{equation}
Indeed, since $\eta'(\tau)\neq 0$ if and only if $1\leq \tau\leq 2$, we have 
\begin{equation*}
\nabla \eta_n \cdot x = \eta'(n d(x))n \nabla d(x)\cdot x = 
\begin{cases}
\eta'(n d(x))n d(x) &\text{if $1\leq n d(x)\leq 2$},\\
0 &\text{otherwise},
\end{cases}
\end{equation*}
by definition of $\eta_n$ and  Lemma \ref{1om}. Hence  $\{\nabla \eta_n \cdot x\}$ is  bounded in $ L^\infty(\mathcal{C}\cap B'_r)$ uniformly with respect to $n$. 
Furthermore, since $h_n\to h$ in $L^p(\mathcal C\cap B_{R_0}')$ by \eqref{lehnconverg} and $\mathop{\rm Tr}U_n\to \mathop{\rm Tr}U$ in $L^{2^\ast(s)}(B'_{R_0})$ by \eqref{eq:lemma2.6FF} and Lemma \ref{convdelleUn}, we have 
\begin{equation}\label{xalla-sconv2}
    h_n|\mathop{\rm Tr}U_n|^2\to   h|\mathop{\rm Tr}U|^2 \quad\text{in }L^1(\mathcal C\cap B_{R_0}'),
\end{equation}
whereas  from \cite{FalFel14}*{Lemma 2.5} and Lemma \ref{convdelleUn} it follows that \begin{equation}\label{xalla-sconv}
\frac{|\mathop{\rm Tr}U_n|^2}{|x|^{2s}}\to\frac{|\mathop{\rm Tr}U|^2}{|x|^{2s}}\quad \text{in $L^{1}(\mathcal{C}\cap B'_{R_0})$}.
\end{equation} 
 Along a subsequence, the convergences \eqref{xalla-sconv2} and \eqref{xalla-sconv} also hold a.e. in $(0,R_0)$ and $\big|h_n + \frac{\lambda}{|x|^{2s}}\big||\mathop{\rm Tr}U_n|^2$  is almost everywhere dominated by a $L^1(\mathcal{C}\cap B'_{R_0})$-function uniformly with respect to $n$.
Therefore, by the Lebesgue dominated convergence theorem we conclude that 
\begin{multline*}\label{tia}
\lim_{n\to\infty}
\int_{ \mathcal{C}\cap B'_r} \left(h_n + \frac{\lambda}{|x|^{2s}}\right)|\mathop{\rm Tr}U_n|^2(\nabla\eta_n\cdot x) \, dx\\
=\lim_{n\to\infty} \int_{\mathcal{C}\cap B'_r} \chi_{\{\frac{1}{n}\leq d(x)\leq \frac{2}{n}\}}\left(h_n + \frac{\lambda}{|x|^{2s}}\right)|\mathop{\rm Tr}U_n|^2(\nabla \eta_n\cdot x) \, dx=0,
\end{multline*}
where $\chi$ stands for the characteristic function.
Claim \eqref{convergenzadelgraddieta} is thereby proved.  

To complete the proof, it remains to notice that 
\begin{equation}\label{eq4}
\int_{\mathcal{C}\cap B'_r}\eta_n \frac{|\mathop{\rm Tr}U_n|^2}{|x|^{2s}}\, dx\to \int_{\mathcal{C}\cap B'_r} \frac{|\mathop{\rm Tr}U|^2}{|x|^{2s}}\, dx \quad \text{as $n\to\infty$}
\end{equation}
as a consequence of \eqref{xalla-sconv}. 
From this and the coarea formula, up to a subsequence still denoted with $U_n$, we have 
\begin{equation}\label{eq5}
\int_{\mathcal{C}\cap \partial B'_r}\eta_n \frac{|\mathop{\rm Tr}U_n|^2}{|x|^{2s}}\, dS'\to  \int_{\mathcal{C}\cap\partial B'_r} \frac{|\mathop{\rm Tr}U|^2}{|x|^{2s}}\, dS'\quad\text{as $n\to\infty$, for a.e. $r\in (0,R_0)$}.
\end{equation}
Assembling \eqref{c1}, \eqref{c2}, \eqref{c3}, \eqref{c4}, \eqref{convergenzadelgraddieta}, \eqref{eq4} and \eqref{eq5}, we arrive at \eqref{poho} for a e. $r\in (0,R_0)$. 

In order to prove \eqref{poho2}, we test equation \eqref{eqappro} with $U_n$ itself and then we pass to the limit proceeding as above.  
\end{proof}

\section{Almgren-type monotonicity formula and blow-up analysis}\label{sezionemonotonia}

\subsection{Almgren-type monotonicity formula}
In this section we prove a monotonicity formula for the Almgren frequency function.
First of all, for  a fixed nontrivial weak solution $U\in H^1(B^+_1, t^{1-2s})$ to problem \eqref{eq1EXT},
we introduce  the following two functions:
\begin{align}\label{D}
&D:(0,1)\to\R,\quad D(r)= r^{2s-N}\left(\int_{B^+_r}t^{1-2s}|\nabla U|^2\,dz-\kappa_s\int_{\mathcal{C}\cap B'_r} \left(h+\frac{\lambda}{|x|^{2s}}\right)|\mathop{\rm Tr}U|^2\,dx\right),\\
\label{H}
&H:(0,1)\to\R,\quad H(r)= r^{2s-N-1}\int_{\partial^+B^+_r} t^{1-2s} U^2\, dS.
\end{align}
Then we define the Almgren frequency function as 
\begin{equation}\label{N}
\mathcal N(r)=\frac{D(r)}{H(r)} \quad \text{for every $r\in (0,R_0]$},
\end{equation} 
being $R_0$ as in \eqref{R0}; we note that the above function is well-defined,
since $H$ turns out to be strictly positive in $(0,R_0]$, as stated in the following lemma. 
\begin{lemma}\label{Hpositiva}
The function $H$ defined in \eqref{H} is strictly positive in $(0,R_0]$, with $R_0$ being as in \eqref{R0}. 
\end{lemma}
\begin{proof}
We argue by contradiction, assuming that $H$ is null at some $\bar{r}\in(0,R_0]$. This implies that $U\equiv 0$ on $\partial^+ B^+_{\bar{r}}$. Hence, by testing \eqref{extensionweak} with $U$ itself,
\begin{equation*}
\int_{B^+_{\bar{r}}}t^{1-2s} |\nabla U|^2\, dz -\kappa_s \int_{\mathcal{C}\cap B'_{\bar{r}}}\left(h+\frac{\lambda}{|x|^{2s}}\right)|\mathop{\rm Tr}U|^2\, dx =0,
\end{equation*}
which in turn, together with Lemma \ref{lemusefulineq}, leads us to $\int_{B^+_{\bar{r}}}t^{1-2s} |\nabla U|^2\, dz=0$. Thus, 
$U$ is constant on $B^+_{\bar{r}}$. Since $\mathop{\rm Tr}U=0$ on $ B'_1\setminus \Omega$, we necessarily have $U\equiv0$ in $B^+_{\bar{r}}$. This is enough to conclude that $U$ is identically zero in $B^+_1$ by classical unique continuation principles for second order elliptic equations  with bounded coefficients, see e.g. \cite{GarLin86}, thus producing a contradiction.
\end{proof}

As a direct consequence of Lemma \ref{lemusefulineq}, the Almgren 
frequency function $\mathcal N$ can be estimated from below as follows. 
\begin{lemma}
For every $r\in (0,R_0]$
\begin{equation}\label{Nmagg}
\mathcal N(r)> -\frac{N-2s}{2}.
\end{equation}
\end{lemma}
\begin{proof}
The claim easily follows from  \eqref{usefulineq} and  \eqref{D}--\eqref{N}. 
\end{proof}

In order to provide also an estimate from above for the Almgren frequency function $\mathcal N$, and then to prove the existence of $\lim_{r\to 0^+}\mathcal N(r)\in\mathbb{R}$, some information on the derivative of $\mathcal N$ is needed. For this, we prove below some results involving the derivatives of $D$ and $H$: the Pohozaev inequality \eqref{poho} will be a relevant tool for this purpose.
\begin{lemma}\label{properties}
Let $D$ and $H$ be defined  in \eqref{D} and \eqref{H}, respectively. Then
\begin{itemize}
\item[\emph{(i)}] $H\in W^{1,1}_{\mathrm{loc}}(0,1)$ and 
\begin{equation}\label{H'}
H'(r)=2 r^{2s-N-1}\int_{\partial^+B^+_r} t^{1-2s} U \frac{\partial U}{\partial\nu}\, dS 
\end{equation}
in a distributional sense and a.e in $(0,1)$;
\item[\emph{(ii)}] $D\in W^{1,1}_{\mathrm{loc}}(0,1)$ and 
\begin{equation}\label{D'}
D'(r)\geq \ 2 r^{2s-N}\int_{\partial^+B^+_r} t^{1-2s}\left|\frac{\partial U}{\partial\nu}\right|^2\,dS 
-\kappa_s r^{2s-N-1}\int_{\mathcal C\cap B'_r} (2s h+\nabla h \cdot x) |\mathop{\rm Tr}U|^2\, dx
\end{equation}
in a distributional sense and a.e. in $(0,R_0)$;
\item[\emph{(iii)}] for a.e. $r\in (0,R_0)$
\begin{equation}\label{H'=D}
H'(r)=\frac{2}{r}D(r).
\end{equation}
\end{itemize}
\end{lemma}
\begin{proof}
We omit the proof of  (i) because it is similar to that of \cite{DelFelVit22}*{Lemma 3.1} (see identity  (3.6) in \cite{DelFelVit22}).
With regard to (ii), the regularity of $D$ follows from the coarea formula, which ensures that
\begin{equation*}
r\mapsto \int_{B^+_r}t^{1-2s}|\nabla U|^2\,dz-\kappa_s\int_{\mathcal C\cap B'_r} \left(h+\frac{\lambda}{|x|^{2s}}\right)|\mathop{\rm Tr}U|^2\,dx \in W^{1,1}(0,1),
\end{equation*}
in view of \eqref{fract2}, \eqref{eq:lemma2.6FF},  \eqref{eq:ipo-h}, and the fact that  $U\in H^1(B^+_1, t^{1-2s})$. Moreover,  \eqref{D'} can be derived by using first the coarea formula to write the distributional derivative of $D$,  and then  \eqref{poho} to estimate from below the term $\int_{\partial ^+B^+_r} t^{1-2s}|\nabla U|^2\, dS
-\kappa_s\lambda \int_{\mathcal C\cap \partial B'_r} |x|^{-2s}|\mathop{\rm Tr}U|^2\,dS'$. Finally, \eqref{H'=D} follows from  \eqref{poho2}, \eqref{D}, and \eqref{H'}. 
\end{proof}
Now we have all the ingredients to provide an estimate from below of the derivative of $\mathcal N$.

\begin{lemma}\label{lemmaN'}
Let $\mathcal N$ be defined in \eqref{N}. Then $\mathcal N\in W^{1,1}_{\mathrm{loc}} (0,R_0)$.
Furthermore there exists a positive constant $C_1>0$, depending only on $N$, $s$, $h$, $\lambda$, and $\mathcal C$, such that
\begin{equation}\label{N'}
\mathcal N'(r)\geq -C_1\,r^{-1+\frac{2sp-N}{p}} \left(\mathcal N(r)+\frac{N-2s}{2}\right)\quad \text{for almost every $r\in (0,R_0)$.}
\end{equation}
\end{lemma}
\begin{proof}
We immediately infer that
$\mathcal N\in W^{1,1}_{\mathrm{loc}} (0,R_0)$, by observing that it is the product of $D$ and $1/H$, which belong to $W^{1,1}_{\mathrm{loc}} (0,1)$ by Lemma \ref{properties} (i-ii), with $H>0$ in $(0,R_0)$ in view of  Lemma \ref{Hpositiva}. By direct calculations, using first \eqref{H'=D}, and then \eqref{H'}--\eqref{D'}, we obtain  
\begin{align}\label{disugdiN'}
\mathcal N'(r)&\geq \ 2r\frac{\left(\int_{\partial^+ B^+_r} t^{1-2s}\left|\frac{\partial U}{\partial\nu}\right|^2\, dS\right)\left(\int_{\partial^+ B^+_r} t^{1-2s}U^2\, dS\right)-\left(\int_{\partial^+ B^+_r} t^{1-2s}U \frac{\partial U}{\partial\nu}\, dS\right)^2}{\left(\int_{\partial^+ B^+_r}t^{1-2s}U^2\, dS\right)^2}\\
&\notag \quad -\kappa_s\frac{\int_{\mathcal{C}\cap B'_r } (2sh+\nabla h\cdot x) |\mathop{\rm Tr}U|^2\, dx}{\int_{\partial^+ B^+_r} t^{1-2s}U^2\, dS }
\end{align}
for a.e. $r\in (0,R_0)$.
In order to to estimate the second term on the right-hand side of \eqref{disugdiN'},  we observe that, by the H\"{o}lder inequality and Lemma \ref{maggioredi2ast},
\begin{align*}
\biggl|\int_{\mathcal{C}\cap B'_r} &(2sh+\nabla h\cdot x)|\mathop{\rm Tr}U|^2\, dx\biggr| \leq \,V_N^{\frac{2sp-N}{Np}} \Vert 2sh+\nabla h\cdot x \Vert _{L^p(\mathcal C \cap B'_1)} r^{\frac{2sp-N}{p}}\bigg(\int_{\mathcal C\cap B'_r} |\mathop{\rm Tr}U|^{2^\ast(s)}\, dx\bigg)^{\!\!\frac{2}{2^\ast(s)}}\\
\leq &\,V_N^{\frac{2sp-N}{Np}}\frac{\tilde{S}_{N,s}}{\tilde C}\Vert 2sh+\nabla h\cdot x \Vert _{L^p(\mathcal{C}\cap B'_1)} r^{\frac{2sp-N}{p}} \biggl(\int_{B^+_r}t^{1-2s}|\nabla U|^2\, dz\\ 
&\hskip3cm -\kappa_s\int_{\mathcal{C}\cap B'_r}\bigg( h+\frac{\lambda}{|x|^{2s}}\bigg) |\mathop{\rm Tr}U|^2\,dx+\frac{N-2s}{2r}\int_{\partial^+ B^+_r}t^{1-2s}U^2\, dS\biggr)
\end{align*}
for all $r\in(0,R_0)$.
From this and exploiting \eqref{D}, \eqref{H} and \eqref{N}, we conclude that
\begin{equation}\label{modulo}
\left|
\kappa_s\frac{\int_{\mathcal{C}\cap B'_r } (2sh+\nabla h\cdot x) |\mathop{\rm Tr}U|^2\, dx}{\int_{\partial^+ B^+_r} t^{1-2s}U^2\, dS }
\right| \leq C_1 \,r^{-1+\frac{2sp-N}{p}} \left(\mathcal N(r)+\frac{N-2s}{2}\right)
\end{equation}
for some constant $C_1>0$ which depends only on $N$, $s$, $h$, $\lambda$, and $\mathcal C$. 

Observing that the first term on the right-hand side of \eqref{disugdiN'} is positive by the Cauchy-Schwarz inequality, applied to $U$ and $\frac{\partial U}{\partial\nu}$ as vectors in $L^2(\partial^+ B^+_r, t^{1-2s})$, and  combining \eqref{disugdiN'} and \eqref{modulo},  we finally arrive at \eqref{N'}, as desired. 
\end{proof}

Estimate \eqref{N'} allows us to prove the boundedness from above of $\mathcal N$ and thus the existence of $\lim_{r\to 0^+}\mathcal N(r)$. 
\begin{lemma}
There exists a constant $C_2>0$ such that
\begin{equation}\label{Nlimitdallalto}
\mathcal N(r)\leq C_2\quad\text{for every $r\in (0,R_0)$}.
\end{equation}
\end{lemma}
\begin{proof}
By \eqref{N'}, for almost every $\tau\in (0,R_0)$ we have 
\begin{equation*}
\left(\mathcal N(\cdot)+\frac{N-2s}{2}\right)'(\tau)\geq -C_1 \,\tau^{-1+\frac{2sp-N}{p}} \left(\mathcal N(\tau)+\frac{N-2s}{2}\right).
\end{equation*}
By \eqref{Nmagg},  we can divide  by $\mathcal N(\tau)+ \frac{N-2s}{2} $, thus obtaining, after integration  over $(r,R_0)$ for any  $r\in (0,R_0)$,
\begin{equation*}
\mathcal{N}(r)+\frac{N-2s}{2}\leq\left(\mathcal N(R_0)+\frac{N-2s}{2} \right) \exp\Big(\tfrac{C_1 p}{2sp-N} R_0^{\frac{2sp-N}{p}}\Big)\quad \text{for every $r\in (0,R_0)$}.
\end{equation*}
The proof of \eqref{Nlimitdallalto} is thereby complete. 
\end{proof}
\begin{lemma}\label{esistelimite}
The limit $\lim_{r\to 0^+}\mathcal N(r)$ does exist and is finite.
\end{lemma}
\begin{proof}
If we rewrite $\mathcal N'(\tau)$, for a.e. $\tau\in (0,R_0)$, as 
\begin{equation*}
\mathcal N'(\tau)= \left[\mathcal N'(\tau)+C_1 \,\tau^{-1+\frac{2sp-N}{p}} \left(\mathcal N(\tau)+\frac{N-2s}{2}\right)\right] -C_1 \,\tau^{-1+\frac{2sp-N}{p}} \left(\mathcal N(\tau)+\frac{N-2s}{2}\right) ,
\end{equation*}
we observe that it is the sum of a nonnegative function, by Lemma \ref{lemmaN'}, and of a $L^1(0,R_0)$-function, by \eqref{Nlimitdallalto} and the fact that $p>\frac{N}{2s}$. Therefore, from the absolute continuity of $\mathcal N$ it follows that,  for any $r\in (0,R_0)$,
\begin{equation*}
\begin{split}
\mathcal N(r)=& \ \mathcal N(R_0) - \int_r^{R_0} \mathcal N'(\tau
)\, d\tau \\
=&\  \mathcal N(R_0) - \int_0^{R_0} \chi_{(r,R_0)}(\tau)\left[\mathcal N'(\tau)+C_1 \,\tau^{-1+\frac{2sp-N}{p}} \left(\mathcal N(\tau)+\frac{N-2s}{2}\right)\right]\, d\tau \\
&+ \int_0^{R_0} \chi_{(r,R_0)}(\tau) C_1 \,\tau^{-1+\frac{2sp-N}{p}} \left(\mathcal N(\tau)+\frac{N-2s}{2}\right) \, d\tau.
\end{split}
\end{equation*}
Hence,  the existence of the limit  $\lim_{r\to0^+}\mathcal N(r)$ follows from the monotone convergence theorem and the Lebesgue dominated convergence theorem. 
Furthermore, such a limit is finite as a consequence of \eqref{Nmagg} and \eqref{Nlimitdallalto}. 
\end{proof}
From now on, we will denote
\begin{equation}\label{illimitegamma}
\gamma:=\lim_{r\to 0^+}\mathcal N(r).
\end{equation}
In the following lemma we exhibit two inequalities, providing an upper and lower bound estimate for the order of magnitude of $H(r)$ as $r\to0^+$.
\begin{lemma}\label{lemmaHr2gamma}
There exists a constant $C_3>0$ such that 
\begin{equation}\label{prima}
\frac{H(r)}{r^{2\gamma}}\leq C_3\quad \text{for all $r\in (0,R_0)$}.
\end{equation}
Furthermore, for any $\varepsilon>0$ there exists a constant $C_4=C_4(\e)>0$ such that 
\begin{equation}\label{seconda}
\frac{H(r)}{r^{2\gamma+\varepsilon}}\geq C_4\quad \text{for all $r\in (0,R_0)$}.
\end{equation}
\end{lemma}
\begin{proof}
In view of Lemma \ref{esistelimite} and letting $\gamma$ be as in \eqref{illimitegamma},
for all $\rho\in (0,R_0)$ we have
\begin{equation}\label{servedopo}
\mathcal N(\rho)-\gamma = \int_0^\rho \mathcal N'(\tau)\, d\tau.
\end{equation}
From this, \eqref{N'}, and \eqref{Nlimitdallalto}, we  deduce that, for all $\rho\in (0,R_0)$,
\begin{equation*}
\mathcal N(\rho)\geq \gamma - C_1 \frac{p}{2sp-N}\left(C_2+\frac{N-2s}{2}\right)
\rho^{\frac{2sp-N}{p}}.
\end{equation*}
Hence, in view of  \eqref{H'=D} and \eqref{N}, for a.e. $\rho\in (0,R_0)$ we have
\begin{equation}\label{aboveeq}
\frac{H'(\rho)}{H(\rho)} = \frac{2}{\rho}\mathcal N(\rho)\geq \frac{2}{\rho}\gamma- C_1'\left(C_2+\frac{N-2s}{2}\right)\rho^{-1+\frac{2sp-N}{p}},
\end{equation}
with $C_1'= \frac{2C_1p}{2sp-N}$. Integrating \eqref{aboveeq} with respect to $\rho$ over $(r,R_0)$ for any fixed $r\in (0,R_0)$, we obtain~\eqref{prima}. 

As for \eqref{seconda}, the proof is also based on the integration of the identity $\frac{H'}{H} = \frac{2}{\rho}\mathcal N$, after estimating $\mathcal N$ from above with $\gamma+\frac\e2$ is a neighbourhood of $0$, as made possible by Lemma \ref{esistelimite} and  \eqref{illimitegamma}.
\end{proof}
\begin{lemma}
The limit
\begin{equation}\label{limdaprovmagdizero}
\lim_{r\to 0^+} \frac{H(r)}{r^{2\gamma}}
\end{equation}
does exist and is finite.
\end{lemma}
\begin{proof}
In light of \eqref{prima}, we only have to show that the limit does exist. By \eqref{H'=D}, \eqref{N} and \eqref{servedopo}, for a.e. $\rho\in (0,R_0)$  we have 
\begin{equation*}
\begin{split}
\frac{d}{d\rho}\frac{H(\rho)}{\rho^{2\gamma}}=& \frac{H'(\rho)}{\rho^{2\gamma}}-2\gamma\frac{H(\rho)}{\rho^{2\gamma+1}}= \frac{2(D(\rho)-\gamma H(\rho))}{\rho^{2\gamma+1}}= \frac{2H(\rho)}{\rho^{2\gamma+1}}\int_0^\rho \mathcal N'(\tau)\, d\tau\\
=& \frac{2H(\rho)}{\rho^{2\gamma+1}} \int_0^\rho 
\left(\mathcal N'(\tau)+C_2'\tau^{-1+\frac{2sp-N}{p}} \right)\,d\tau- \frac{2 C_2'H(\rho)}{\rho^{2\gamma+1}} \int_0^\rho\tau^{-1+\frac{2sp-N}{p}} \, d\tau.
\end{split}
\end{equation*}
where $C_2'=C_1\left(C_2+\frac{N-2s}{2}\right)$.
Integrating the above identity with respect to $\rho\in (r,R_0) $ for any fixed $r\in (0,R_0)$, we obtain 
\begin{align*}
\frac{H(R_0)}{R_0^{2\gamma}}- \frac{H(r)}{r^{2\gamma}}=2\int_r^{R_0}& \frac{H(\rho)}{\rho^{2\gamma+1}} \int_0^\rho 
\left(\mathcal N'(\tau)+C_2'\tau^{-1+\frac{2sp-N}{p}} \right)d\tau\, d\rho- \frac{2 C_2' p}{2sp-N} \int_r^{R_0} \frac{H(\rho)}{\rho^{2\gamma}} \rho^{-1+\frac{2sp-N}{p}}\, d\rho.  
\end{align*}
The first term on the right hand side of the above identity  has a limit as $r\to 0^+$ by the monotone convergence theorem, since the integrand is nonnegative as a consequence of Lemma \ref{Hpositiva}, \eqref{N'} and \eqref{Nlimitdallalto}. The second term admits a finite limit as $r\to 0^+$ by the Lebesgue dominated convergence theorem and \eqref{prima}. 
\end{proof}

\subsection{Blow-up analysis}

In this section, we perform a blow-up analysis at a given conical boundary point (fixed at the origin, without loss of generality). 
Given a non-trivial weak solution $U\in H^1(B^+_1, t^{1-2s})$ to \eqref{eq1EXT}, we are interested in the limiting behaviour, as $\tau \to 0^+$, of the following blow-up family 
\begin{equation}\label{wlambda}
w^\tau(z):= \frac{U(\tau z)}{\sqrt{H(\tau)}},\qquad \tau\in (0,1).
\end{equation}
Applying the change of variable $z\mapsto \tau z$, one may immediately observe that 
\begin{equation}\label{normalizzaz}
\int_{\partial ^+ B^+_1}t^{1-2s}|w^\tau|^2\, dS=1.
\end{equation}
Moreover, for any  $\tau\in (0,1)$, $w^\tau\in H^1(B^+_{1/\tau}, t^{1-2s})$ is a weak solution to 
\begin{equation}\label{eqwlambda}
\begin{cases}
\mathrm{div}(t^{1-2s}\nabla w^\tau)=0 &\mathrm{in} \ B_{1/\tau}^+\\
-\lim_{t\to0^+}t^{1-2s}\partial_t w^\tau=\kappa_s\left( \tau ^{2s}h(\tau x) + \frac{\lambda}{|x|^{2s}}\right) \mathop{\rm Tr}w^\tau &\mathrm{on} \ \mathcal C\cap B'_{1/\tau}\\
w^\tau=0 &\mathrm{on} \ B'_{1/\tau}\setminus \mathcal C,
\end{cases}
\end{equation}
i.e.
\begin{equation*}
\int_{B^+_{1/\tau}} t^{1-2s} \nabla w^\tau\cdot\nabla \psi\, dz = \kappa_s\int_{\mathcal C\cap B'_{1/\tau}} \left(\tau^{2s}h(\tau x)+ \frac{\lambda}{|x|^{2s}}\right) \mathop{\rm Tr}w^\tau \mathop{\rm Tr}\psi\, dx
\end{equation*}
for every $\psi\in C^\infty_c(B^+_{1/\tau}\cup (\mathcal C\cap B'_{1/\tau}))$.

\begin{lemma}
Let $R_0$ be as in \eqref{R0} and 
$w^\tau$ be defined in \eqref{wlambda}, for every $\tau\in (0,R_0)$.
 Then there exists a positive constant $M>0$ such that 
\begin{equation}\label{wtaulimitata}
\Vert w^\tau\Vert_{H^1(B^+_1, t^{1-2s})}\leq M\quad \text{for every $\tau\in (0,R_0)$}.
\end{equation}
\end{lemma}
\begin{proof}
By Lemma \ref{lemusefulineq} and the change of variable $z\mapsto\tau z $, we have 
\begin{equation*}
\begin{split}
\int_{B^+_\tau} t^{1-2s} |\nabla U|^2\, dz -\kappa_s\int_{\mathcal{C}\cap B'_\tau}&\left(h+\frac{\lambda}{|x|^{2s}}\right) |\mathop{\rm Tr}U|^2\, dx\geq 
  -\frac{N-2s}{2}\tau^{N-2s}H(\tau) \int_{\partial ^+ B^+_1} t^{1-2s} |w^\tau|^2\, dS \\
 &+ \tau^{N-2s}H(\tau)\tilde{C}\left(\int_{B^+_1}|\nabla w^\tau|^2 \, dz + \frac{N-2s}{2} \int_{\partial^+ B^+_1} t^{1-2s} |w^\tau|^2\, dS\right)
\end{split}
\end{equation*}
for every $\tau\in (0,R_0)$.
Dividing by $\tau^{N-2s}H(\tau)$ and taking into account \eqref{normalizzaz}, we deduce that 
\begin{equation*}
\mathcal N(\tau)\geq \tilde{C}\int_{B^+_1}|\nabla w^\tau|^2 \, dz + \frac{N-2s}{2}(\tilde{C}-1).
\end{equation*}
The latter, combined with \eqref{Nlimitdallalto}, allows us to conclude that $\Vert \nabla w^\tau\Vert_{L^2(B^+_1,t^{1-2s})}$ is bounded uniformly with respect to $\tau\in (0,R_0)$. From this and \cite{FalFel14}*{Lemma 2.4}, we get \eqref{wtaulimitata}. 
\end{proof}
The next lemma provides some crucial doubling-type properties.   
\begin{lemma}\label{l:doub}
There exists a constant $C_5>0$ such that, for every $\tau\in (0,R_0/2)$ and every $R\in [1,2]$,
\begin{equation}\label{doubH}
\frac{1}{C_5}\leq \frac{H(R\tau)}{H(\tau)}\leq C_5,
\end{equation}
\begin{equation}\label{stimawtau}
\int_{B^+_R } t^{1-2s}|w^\tau|^2\,dz\leq C_5 2^{N+2-2s} \int_{B^+_1} t^{1-2s}|w^{R\tau}| ^2\,dz
\end{equation}
and
\begin{equation}\label{stimanablawtau}
\int_{B^+_R}t^{1-2s} |\nabla w^\tau| ^2\,dz\leq C_5 2^{N-2s} \int_{B^+_1}t^{1-2s}|\nabla w^{R\tau}| ^2\,dz.
\end{equation}

\end{lemma}
\begin{proof}
Putting together  \eqref{N}, \eqref{Nmagg}, \eqref{H'=D}, and \eqref{Nlimitdallalto}, 
we deduce that, for a.e. $r\in (0,R_0)$,
\begin{equation*}
-\frac{N-2s}{r}\leq \frac{H'(r)}{H(r)}\leq \frac{2C_2}{r}. 
\end{equation*}
Integrating the above inequality over $(\tau, R\tau)$, with $1<R\leq 2$ and $\tau \in (0,R_0/R)$, we obtain 
\begin{equation*}\label{ok}
2^{2s-N}\leq \frac{H(R\tau)}{H(\tau)}\leq 4^{C_2}\quad \text{for any $1<R\leq 2$ and $\tau\in (0, R_0/R)$},
\end{equation*}
that implies \eqref{doubH}, which is trivially satisfied if $R=1$. As for the proof of \eqref{stimawtau}, we notice that, applying first the change of variable $z\mapsto \tau z$ and then $z\mapsto (R\tau)^{-1} z$, one has
\begin{equation*}
\int_{B^+_R} t^{1-2s} |w^\tau|^2\, dz = \frac{\tau^{2s-2-N}}{H(\tau)}\int_{B^+_{R\tau}} t^{1-2s}U^2\,dz= R^{N-2s+2}\frac{H(R\tau)}{H(\tau)}\int_{B^+_1} t^{1-2s} |w^{R\tau}|^2\, dz
\end{equation*}
for every $\tau\in (0,R_0/2)$ and  $R\in [1,2]$.
This, combined with \eqref{doubH}, leads to \eqref{stimawtau}. 
The proof of  \eqref{stimanablawtau} is analogous.
\end{proof}
The following result provides an estimate of the boundary energy for a selection of radii.
\begin{lemma}\label{lemmastimadivol}
There exist $\overline{M}>0$ and $\tau_0\in (0,R_0/2)$ such that, for every $\tau\in (0,\tau_0)$, there exists $R_\tau\in [1,2]$ such that 
\begin{equation}\label{stimadivolume}
\int_{\partial^+ B^+_{R_\tau}} t^{1-2s}|\nabla w^\tau|^2\,dS\leq \overline{M} \int_{B^+_{R_\tau}} t^{1-2s}|\nabla w^\tau|^2\,dz.
\end{equation}
\end{lemma}
\begin{proof}
From  \eqref{wtaulimitata} and \eqref{stimawtau}--\eqref{stimanablawtau} with $R=2$, it follows that 
\begin{equation}\label{limitatainB2}
\{w^\tau\}_{\tau\in (0,R_0/2)} \quad \text{is bounded in $H^1(B^+_2, t^{1-2s})$}.
\end{equation}
We also observe that, for every fixed $\tau \in (0,R_0/2)$, the function $$f_\tau: r\mapsto \int_{B^+_r}t^{1-2s} |\nabla w^\tau|^2\, dz$$ is absolutely continuous in $[0,2]$ with distributional derivative given by 
\begin{equation*}
f'_\tau(r)=\int_{\partial^+B^+_r} t^{1-2s} |\nabla w^\tau|^2\, dS\quad\text{for a.e. }r\in(0,2).
\end{equation*} 
We argue by contradiction and  assume that, for any $\overline{M}>0$, there exists a sequence $\tau_n\to 0^+$ such that 
\begin{equation*}
f'_{\tau_n}(r) > \overline{M}f_{\tau_n}(r) 
\end{equation*}
for a.e. $r\in (1,2)$ and for every $n\in \mathbb{N}$. An integration of the last inequality leads to 
\begin{equation*}
f_{\tau_n}(2)>e^{\overline{M}} f_{\tau_n}(1) \quad \text{ for every $n\in \mathbb{N}$}.
\end{equation*}
Therefore 
\begin{equation}\label{in}
\liminf _ {\tau\to 0^+}f_\tau(1)\leq \liminf _{n\to \infty}f_{\tau_n}(1)\leq \limsup _{n\to \infty}f_{\tau_n}(1) \leq e^{-\overline{M}}  \limsup _{n\to \infty}f_{\tau_n}(2)\leq e^{-\overline{M}} \limsup _ {\tau\to 0^+}f_\tau(2). 
\end{equation}
We observe that $\limsup _ {\tau\to 0^+}f_\tau(2)<+\infty$ by \eqref{limitatainB2}; hence we can pass to the limit as  $\overline{M}\to +\infty$  in \eqref{in}, thus obtaining
$\liminf _ {\tau\to 0^+}f_\tau(1)=0$.  Hence there exists a sequence $\tilde{\tau}_n\to 0^+$ such that 
\begin{equation}\label{vanish}
\int_{B^+_1}t^{1-2s}|\nabla w^{\tilde{\tau}_n}|^2\, dz \to 0 \quad \text{as $n\to \infty$}.
\end{equation}
By \eqref{wtaulimitata}, up to consider a subsequence still denoted with $\tilde{\tau}_n$, there exists $w\in H^1(B^+_1,t^{1-2s})$ such that
\begin{equation}\label{deboleconv}
w^{\tilde{\tau}_n} \rightharpoonup w \quad \text{weakly in $H^1(B^+_1,t^{1-2s})$}.
\end{equation}
By \eqref{vanish},  \eqref{deboleconv}, and  weak lower semicontinuity of the norm, we  conclude that
\begin{equation}\label{cost}
\int_{B^+_1} t^{1-2s} |\nabla w|^2\, dz =0.
\end{equation} 
From \eqref{deboleconv}, \eqref{normalizzaz}, and 
the compactness of the trace map
\begin{equation}\label{compact}
\mathop{\rm tr}: H^1(B^+_1, t^{1-2s}) \to L^2(\partial^+B^+_1, t^{1-2s}),
\end{equation} it  follows that 
\begin{equation*}
\int_{\partial ^+ B^+_1}t^{1-2s}|w|^2\, dS=1
\end{equation*}
 This, together with \eqref{cost}, implies that 
\begin{equation}\label{dacontraddire}
w\equiv \mathrm{const}\neq 0 \quad \text{in $B^+_1$}.
\end{equation}
On the other hand, since the space of all functions in $ H^1(B^+_1,t^{1-2s})$ with null trace on $B'_1\setminus\Omega$ is weakly closed in $ H^1(B^+_1,t^{1-2s})$, from \eqref{eqwlambda} and \eqref{deboleconv} we deduce that $\mathop{\rm Tr}w=0$ on $B'_1\setminus\Omega$. 
This contradicts \eqref{dacontraddire}, thus completing the proof.
\end{proof}
Combining Lemmas \ref{l:doub} and \ref{lemmastimadivol}, we obtain the following uniform bound of energies on $\partial^+B_1^+$, for blow-up functions with scaling parameter of the form $\tau R_\tau$.
\begin{lemma}\label{l:boundbord}
There exists $\widetilde{M}>0$ such that, for every $\tau\in (0,\tau_0)$,
\begin{equation*}
\int_{\partial ^+ B^+_1} t^{1-2s}|\nabla w^{\tau R_\tau}|^2\,dS\leq \widetilde{M}. 
\end{equation*}
\end{lemma}
\begin{proof}
By the change of variable $z\mapsto R_\tau z$, from \eqref{doubH}, \eqref{stimadivolume}, and \eqref{stimanablawtau} it follows that 
\begin{equation*}
\begin{split}
\int_{\partial^+B^+_1}t^{1-2s}|\nabla w^{\tau R_\tau}|^2\, dS&\leq R_\tau^{1+2s-N}\frac{H(\tau)}{H(\tau R_\tau)} \int_{\partial ^+ B^+_{R_\tau}} t^{1-2s}|\nabla w^\tau|^2\, dS\\
&\leq 2C_5\overline{M} \int_{B^+_{R_\tau}}t^{1-2s} |\nabla w^\tau|^2\, dz\leq 2^{1+N-2s}C_5^2 \overline{M} \int_{B^+_1} t^{1-2s} |\nabla w^{\tau R_\tau}|^2\, dz.
\end{split}
\end{equation*}
The thesis then follows from \eqref{wtaulimitata}, since $\tau R_\tau <R_0$. 
\end{proof}
\begin{remark}\label{remark}
As a consequence of Lemma \ref{l:boundbord}, we have that, for any sequence $\tau_n\to 0^+$, there exist a subsequence $\{\tau_{n_k}\}_{k\in\mathbb{N}}$ and a function $\xi \in L^2(\partial ^+ B^+_1, t^{1-2s})$ such that 
\begin{equation*}
\frac{\partial w^{\tau_{n_k}R_{\tau_{n_k}}}}{\partial\nu } \rightharpoonup \xi \quad \text{weakly in $L^2(\partial ^+ B^+_1, t^{1-2s})$ as $k\to\infty$}. 
\end{equation*}   
\end{remark}
\begin{lemma}\label{lemmaconvblow}
Let $\gamma$ be defined in \eqref{illimitegamma}. Then, there exists an eigenvalue $\mu_{   j_0}$ of problem \eqref{probautov} such that
\begin{equation}\label{relazione}
\gamma= \sqrt{\left(\frac{N-2s}{2}\right)^2+\mu_{j_0} }-\frac{N-2s}{2}.
\end{equation} 
Furthermore, for every sequence $\tau_n\to 0^+$, there exist a subsequence $\{\tau_{n_k}\}_{k\in\mathbb{N}}$ and a $L^2(\mathbb{S}^N_+, \theta_{N+1}^{1-2s})$-normalized eigenfunction $\psi$ associated with $\mu_{j_0} $ such that 
\begin{equation*}
w^{\tau_{n_k}}(z) \to |z|^\gamma\psi\left(\frac{z}{|z|}\right)\quad \text{as $k\to\infty$ in $H^1(B^+_1, t^{1-2s})$}.
\end{equation*}
\end{lemma}
\begin{proof}
We observe  that, in view of \eqref{wtaulimitata},
the family $\{w^{\tau R_\tau}\}_{\tau\in (0, \tau_0)}$, with $R_\tau\in [1,2]$ being as in Lemma \ref{lemmastimadivol},
  is bounded in $H^1(B^+_1, t^{1-2s})$.
  Therefore, for any sequence $\tau_n\to 0^+$, there exist a subsequence 
  $\{\tau_{n_k}\}_{k\in\mathbb{N}}$  and $w\in H^1(B^+_1,t^{1-2s})$ such that
\begin{equation}\label{debtaun}
w^{\tau_{n_k} R_{\tau_{n_k}}} \rightharpoonup w\quad \text{as $k\to\infty$ 
weakly in $H^1(B^+_1, t^{1-2s})$}.
\end{equation}
We observe that, by  \eqref{normalizzaz}, \eqref{debtaun}, and the compactness of the trace map \eqref{compact}
we have 
\begin{equation}\label{eq:normalizzazione-w}
\int_{\partial^+ B^+_1}t^{1-2s}|w|^2\, dS=1,
\end{equation}
so that, in particular, $w\not\equiv 0$.

We split the rest of the proof into several steps.
 
\smallskip\noindent
\emph{Step 1.}
Our first goal is to show that the convergence in \eqref{debtaun} is actually a strong convergence, i.e. 
\begin{equation}\label{convfortedaprovare}
 w^{\tau_{n_k} R_{\tau_{n_k}}} \to w  \quad \text{as $k\to\infty$ strongly in $H^1(B^+_1, t^{1-2s})$}.
\end{equation}
To this aim, we notice that, since $B^+_1\subset B^+_{1/\tau_{n_k} R_{\tau_{n_k}}}$ if $k$ is sufficiently large,  by \eqref{eqwlambda} we have
\begin{multline}\label{passallimite}
\int_{B^+_1}t^{1-2s} \nabla w^{\tau_{n_k} R_{\tau_{n_k}}} \cdot \nabla \psi \, dz=\, \kappa_s\int_{\mathcal C\cap B'_1} (\tau_{n_k} R_{\tau_{n_k}} )^{2s}h(\tau_{n_k} R_{\tau_{n_k}} x)\mathop{\rm Tr}w^{\tau_{n_k} R_{\tau_{n_k}}} \mathop{\rm Tr}\psi\, dx \\
+\kappa_s \int_{\mathcal C\cap B'_1} \frac{\lambda}{|x|^{2s}} \mathop{\rm Tr}w^{\tau_{n_k} R_{\tau_{n_k}}} \mathop{\rm Tr}\psi\, dx + \int_{\partial ^+ B^+_1} t^{1-2s}  \frac{\partial w^{\tau_{n_k} R_{\tau_{n_k}}}}{\partial\nu} \, \psi
\,dS,
\end{multline}
 for every $\psi\in H^1(B^+_1,t^{1-2s})$ with $\mathop{\rm Tr}\psi=0$ on $ B'_1\setminus\mathcal C$ and 
for sufficiently large $k$. By the continuity of the trace operator \eqref{eq:traccia-H1} from $H^1(B^+_1,t^{1-2s})$ to $ L^{2^\ast(s)}(B'_1)$, the boundedness of $\{w^{\tau_{n_k} R_{\tau_{n_k}}}\}_{k\in\mathbb{N}}$ in $H^1(B^+_1,t^{1-2s})$, and  H\"{o}lder's inequality, we have
\begin{multline}\label{prima'}
\int_{\mathcal C\cap B'_1} (\tau_{n_k}R_{\tau_{n_k}} )^{2s} h(\tau_{n_k}R_{\tau_{n_k}} x)\mathop{\rm Tr}w^{\tau_{n_k}R_{\tau_{n_k}}} \mathop{\rm Tr}\psi \, dx \to 0\\
\text{and}\quad 
\int_{\mathcal C\cap B'_1} (\tau_{n_k}R_{\tau_{n_k}} )^{2s} h(\tau_{n_k}R_{\tau_{n_k}} x)|\mathop{\rm Tr}w^{\tau_{n_k}R_{\tau_{n_k}}}|^2\, dx \to 0
\quad\text{as $k\to \infty$}.
\end{multline}
In addition, using \eqref{debtaun} and the fact that 
\begin{equation*}\label{traceopercont}
\mathop{\rm Tr}:H^1(B^+_1, t^{1-2s})  \to L^2(B'_1, |x|^{-2s})\quad\text{is continuous},
\end{equation*}
which in turn follows from \cite{FalFel14}*{Lemma 2.5}, we deduce that 
\begin{equation}\label{seconda'}
\frac{ \mathop{\rm Tr}w^{\tau_{n_k} R_{\tau_{n_k}}} }{|x|^{s}} \rightharpoonup \frac{\mathop{\rm Tr}w}{|x|^{s}} \quad \text{weakly in $L^2(B'_1)$}.
\end{equation}
In light of \eqref{debtaun}, \eqref{prima'}, \eqref{seconda'}, and Remark \ref{remark}, up to extracting a further subsequence,
we can pass to the limit in \eqref{passallimite} as $k\to \infty$, thus obtaining 
\begin{equation}\label{eqdiw}
\int_{B^+_1} t^{1-2s} \nabla w\cdot \nabla \psi \, dz=\kappa_s \int_{\mathcal C\cap B'_1} \frac{\lambda}{|x|^{2s}} \mathop{\rm Tr}w\mathop{\rm Tr}\psi\,dx + \int _{\partial^+B^+_1}t^{1-2s} \xi \psi\, dS
\end{equation}
 for every $\psi\in H^1(B^+_1,t^{1-2s})$ with $\mathop{\rm Tr}\psi=0$ on $B'_1\setminus \mathcal C$.

On the other hand, testing  \eqref{passallimite} with $\psi=w^{\tau_{n_k} R_{\tau_{n_k}}}$ yields
\begin{multline*}
\int_{B^+_1}t^{1-2s} |\nabla w^{\tau_{n_k} R_{\tau_{n_k}}}|^2\, dz
-\kappa_s \int_{\mathcal C\cap B'_1} \frac{\lambda}{|x|^{2s}} |\mathop{\rm Tr}w^{\tau_{n_k} R_{\tau_{n_k}}}  |^2\, dx \\
=  \kappa_ s \int_{\mathcal C\cap B'_1} (\tau_{n_k} R_{\tau_{n_k}})^{2s} h(\tau_{n_k} R_{\tau_{n_k}} x)|\mathop{\rm Tr}w^{\tau_{n_k} R_{\tau_{n_k}}}|^2\, dx
+  \int_{\partial ^+B^+_1} t^{1-2s} \frac{\partial w^{\tau_{n_k} R_{\tau_{n_k}}}}{\partial\nu}w^{\tau_{n_k} R_{\tau_{n_k}}}\, dS.
\end{multline*}
Using \eqref{prima'},  Remark \ref{remark}, and the fact that
\begin{equation}\label{convdelletracce}
w^{\tau_{n_k} R_{\tau_{n_k}}}\to w \quad \text{in $L^2(\partial^+B^+_1, t^{1-2s})$}
\end{equation} 
as a consequence of \eqref{debtaun} and  the compactness of the trace map \eqref{compact},
we can pass to the limit in the above identity as $k\to\infty$, thus obtaining 
\begin{multline}\label{nablaconverge}
\lim_{k\to \infty} 
\left(\int_{B^+_1}t^{1-2s} |\nabla w^{\tau_{n_k} R_{\tau_{n_k}}}|^2\, dz
-\kappa_s \int_{\mathcal C\cap B'_1} \frac{\lambda}{|x|^{2s}} |\mathop{\rm Tr}w^{\tau_{n_k} R_{\tau_{n_k}}}  |^2\, dx \right)\\= \int_{\partial^+B^+_1} t^{1-2s} \xi w\, dS=
\int_{B^+_1}t^{1-2s}|\nabla w|^2\, dz-\kappa_s \int_{\mathcal C\cap B'_1} \frac{\lambda}{|x|^{2s} } |\mathop{\rm Tr}w|^2\, dx,
\end{multline}
in  view of \eqref{eqdiw} with $\psi=w$.
We observe that  the $H^1(B^+_1, t^{1-2s})$-norm is equivalent to the norm 
\begin{equation*}
\bigg(\int_{B^+_1} t^{1-2s}|\nabla v|^2\,dz
-\kappa_s \int_{B'_1} \frac{\lambda}{|x|^{2s}} |\mathop{\rm Tr}v|^2\, dx
+ \frac{N-2s}{2}\int_{\partial^+B^+_1} t^{1-2s}v^2\,dz\bigg)^{1/2},
\end{equation*}
by \eqref{fract2}, the continuity of the trace operator \eqref{compact}, and \cite{FalFel14}*{Lemma 2.4}.
Hence \eqref{nablaconverge}, together with \eqref{convdelletracce} and \eqref{debtaun}, implies \eqref{convfortedaprovare}. 

\smallskip\noindent
\emph{Step 2.} In this step we give an explicit description of the limit profile $w$. To this aim, we consider its Almgren frequency function, defined as
\begin{equation*}\label{Nw}
\mathcal N_w(r):= \frac{D_w(r)}{H_w(r)}\quad\text{for every $r\in (0,1]$},
\end{equation*}
where 
\begin{equation*}\label{Dw}
D_w(r):= r^{2s-N}\left(\int_{B^+_r}t^{1-2s}|\nabla w|^2\,dz-\kappa_s\int_{\mathcal C\cap B'_r} \frac{\lambda}{|x|^{2s}} |\mathop{\rm Tr}w|^2\,dx\right)
\end{equation*}
and 
\begin{equation}\label{Hw}
H_w(r):=r^{2s-N-1}\int_{\partial^+B^+_r} t^{1-2s} w^2\, dS. 
\end{equation}
We notice that \eqref{eqdiw} is the weak formulation of the problem
\begin{equation}\label{formdebw}
\begin{cases}
\mathrm{div}\left(t^{1-2s}\nabla w\right)=0 &\mathrm{in \ } B^+_1\\
-\lim_{t\to0^+} t^{1-2s} \partial_t w =\kappa_s \frac{\lambda}{|x|^{2s}}\mathop{\rm Tr}w &\mathrm{in \ } \mathcal C\cap B'_1\\
\mathop{\rm Tr}w=0 &\mathrm{in \ }  B'_1\setminus \mathcal C.\\
\end{cases}
\end{equation}
We remark that the definition of $\mathcal N_w$ is well-posed: indeed, since $w\not\equiv 0$, one can argue  as in the proof of Lemma \ref{Hpositiva}, with $h\equiv 0$, to verify that $H_w(r)>0$ for every $r\in (0,1]$.  

Next, for every $k\in \mathbb N $ and  $r\in (0,1]$, we define 
\begin{equation*}
D_k(r)
:= r^{2s-N}\bigg(\int_{B^+_r}\!\!t^{1-2s}|\nabla w^{\tau_{n_k} R_{\tau_{n_k}}}|^2 dz - \kappa_s\int_{\Omega\cap B'_r}\!\!\Big( (\tau_{n_k} R_{\tau_{n_k}})^{2s} h(\tau_{n_k} R_{\tau_{n_k}}x)+\tfrac{\lambda}{|x|^{2s}}\Big)|\mathop{\rm Tr}w^{\tau_{n_k} R_{\tau_{n_k}}}|^2 dx\bigg)
\end{equation*}
and
\begin{equation*}
H_k(r):= r^{2s-N-1} \int_{\partial^+ B^+_r} t^{1-2s}|w^{\tau_{n_k} R_{\tau_{n_k}}}|^2\, dS.
\end{equation*}
By direct computations we have 
\begin{equation*}\label{conversioni}
D_k(r)= \frac{D(\tau_{n_k} R_{\tau_{n_k}}r)}{H(\tau_{n_k} R_{\tau_{n_k}})}\quad \text{and} \quad H_k(r)= \frac{H(\tau_{n_k} R_{\tau_{n_k}}r)}{H(\tau_{n_k} R_{\tau_{n_k}})},
\end{equation*}
and hence, by definition of $\mathcal{N}$ (see \eqref{N}), 
\begin{equation}\label{dovepassallim}
\mathcal N(\tau_{n_k} R_{\tau_{n_k}}r)= \frac{D(\tau_{n_k} R_{\tau_{n_k}} r)}{H(\tau_{n_k} R_{\tau_{n_k}} r)} = \frac{D_k(r)}{H_k(r)}\quad\text{for all $r\in(0,1]$ and $k\in {\mathbb N}$}. 
\end{equation}
Moreover, from  \eqref{convfortedaprovare}, \eqref{prima'}, \eqref{fract2}, and \eqref{convdelletracce} it follows that, for every fixed $r\in (0,1]$
\begin{equation}\label{combining}
D_k(r)\to D_w(r) \quad \text{and} \quad H_k(r)\to H_w(r) \quad\text{as $k\to \infty$}.
\end{equation}
Combining \eqref{combining} with the fact that
\begin{equation*}
\lim_{k\to\infty}\mathcal N(\tau_{n_k} R_{\tau_{n_k}}r)=\gamma \quad \text{for every fixed $r\in (0,1]$}
\end{equation*}
as a consequence of \eqref{illimitegamma}, and letting $k\to \infty$ in \eqref{dovepassallim}, we conclude that $\mathcal N_w(r) = \gamma$ for all $r\in(0,1]$. Hence $\mathcal N'_w(r)=0$ for every $r\in (0,1)$. Combining the Cauchy-Schwarz inequality with  estimate \eqref{disugdiN'} for $h\equiv 0$, we infer that 
\begin{equation*}
0=\mathcal N'_w(r)\geq \ 2r\frac{\left(\int_{\partial^+ B^+_r} t^{1-2s}\left|\frac{\partial w}{\partial\nu}\right|^2\, dS\right)\left(\int_{\partial^+ B^+_r} t^{1-2s}w^2\, dS\right)-\left(\int_{\partial^+ B^+_r} t^{1-2s}w \frac{\partial w}{\partial\nu}\, dS\right)^2}{\left(\int_{\partial^+ B^+_r}t^{1-2s}w^2\, dS\right)^2}\geq 0,
\end{equation*}
so that
\begin{equation*}
\left(\int_{\partial^+ B^+_r} t^{1-2s}\left|\frac{\partial w}{\partial\nu}\right|^2\, dS\right)\left(\int_{\partial^+ B^+_r} t^{1-2s}w^2\, dS\right)-\left(\int_{\partial^+ B^+_r} t^{1-2s}w \frac{\partial w}{\partial\nu}\, dS\right)^2=0
\end{equation*}
for a.e. $r\in (0,1)$. In particular, $w$ and $\frac{\partial w}{\partial\nu}$ are parallel vectors in $L^2(\partial^+B^+_r, t^{1-2s})$ for a.e. $r\in (0,1)$, i.e.  there exists $\eta:(0,1)\to\R$ such that
\begin{equation}\label{daintegrare}
\frac{d}{dr}w(r\theta)=\frac{\partial w}{\partial\nu} (r\theta)= \eta(r)w(r\theta)\quad \text{for a.e. $r\in (0,1)$ and for every $\theta\in \mathbb{S}^N_+$}.
\end{equation}  
From \eqref{daintegrare}, \eqref{Hw}, and \eqref{H'} we deduce that $\eta=\frac12 H'_w/H_w$  a.e. in $(0,1)$, so that, in view of  \eqref{H'=D},
\begin{equation*}
\eta(r)= \frac{\mathcal N_w(r)}{\gamma}=\frac{\gamma}{r}\quad \text{for a.e. $r\in (0,1)$}.
\end{equation*}
Therefore, integrating \eqref{daintegrare} over $(r,1)$ for some fixed $r\in (0,1)$, we obtain
\begin{equation}\label{wesplicita}
w(r\theta)= r^\gamma \psi(\theta)\quad \text{for every $r\in (0,1)$ and $\theta\in \mathbb{S}^N_+$},
\end{equation}
where $\psi=w\big|_{\mathbb{S}^N_+}$. Furthermore, \eqref{eq:normalizzazione-w} implies that
\begin{equation}\label{lapsihanorma1}
\int_{\mathbb{S}^N_+} \theta_{N+1}^{1-2s}|\psi(\theta)|^2\, dS=1,
\end{equation}
so that  $\psi\not\equiv0$ on $\mathbb{S}^N_+$. Moreover, the condition $\mathop{\rm Tr}w=$ on $B_1'\setminus\Omega$ implies that $\psi\in V_\omega$, being $V_\omega$ defined in \eqref{Vomega}. By plugging \eqref{wesplicita} into  \eqref{formdebw}, in view of 
 \cite{FalFel14}*{Lemma 2.1}  we infer that 
 \begin{equation*}
\gamma(\gamma+N-2s)r^{\gamma-1-2s}\theta_{N+1}^{1-2s} \psi(\theta)+ r^{\gamma-1-2s} \mathrm{div}_{\mathbb{S}^N} (\theta_{N+1}^{1-2s}\nabla _{\mathbb{S}^N}\psi(\theta))=0,
\end{equation*} 
that is, $\psi$ is an eigenfunction of problem \eqref{probautov} associated with the eigenvalue $\gamma(\gamma+N-2s) $. Therefore, by Proposition \ref{propautovalori}, there exists $j_0\geq 1$ such that $ \mu_{j_0}=\gamma(\gamma+N-2s)$. 
It follows that either $\gamma=y^+$ or $\gamma=y^-$, where
\begin{equation*}
y_{\pm}=-\frac{N-2s}{2}\pm\sqrt{\left(\frac{N-2s}{2}\right)^2 +\mu_{j_0}}.
\end{equation*}
Then \eqref{relazione} must be satisfied because, otherwise, if $\gamma=y_-$, \eqref{wesplicita} would imply 
\begin{equation*}
\mathop{\rm Tr}w(x)=|x|^{y_-}\psi\left(\frac{x}{|x|},0\right)\notin L^{2}(B'_1,|x|^{-2s}),
\end{equation*}
which contradicts the fact that $w\in H^1(B^+_1, t^{1-2s})$ and \eqref{fract2}.

\emph{Step 3.} 
To complete the proof, we prove that, up to passing to a further subsequence,  $w^{\tau_{n_k}}\to w$ as $k\to \infty$ in $H^1(B^+_1, t^{1-2s})$. 
To this aim, we observe that, by \eqref{doubH},  up to a further subsequence,  the limits
\begin{equation}\label{eq:limits}
\ell:= \lim_{k\rightarrow +\infty}\frac{H(\tau_{n_k}R_{\tau_{n_k}})}{H(\tau_{n_k})}\quad\text{and}\quad 
\overline{R}:=\lim_{k\rightarrow +\infty}R_{\tau_{n_k}}
\end{equation}
exist, with $\ell\in(0,+\infty)$ and $\overline{R}\in[1,2]$. 
By \eqref{convfortedaprovare}, a scaling argument and \eqref{eq:limits} we can verify that  
\begin{equation*}
w^{\tau_{n_k}}\rightarrow
\sqrt{\ell}w(\cdot/\overline{R})\quad\text{strongly in $H^1(B^+_1,t^{1-2s}dz)$},
\end{equation*}
see \cite{DelFelVit22}*{Proposition 4.5} for details. On the other hand, \eqref{wesplicita} implies that $\sqrt{\ell}w(\cdot/\overline{R})=
\overline{R}^{-\gamma}\sqrt{\ell}w$, so that the normalization conditions \eqref{normalizzaz} and \eqref{eq:normalizzazione-w}, together with the continuity of the trace map \eqref{compact}, imply that necessarily $\overline{R}^{-\gamma}\sqrt{\ell}=1$, thus completing the proof.
\end{proof}
The next step consists in proving that the limit in \eqref{limdaprovmagdizero} is strictly positive. To this aim, letting
$\{\psi_j\}_{j\geq 1}$ being the orthonormal basis  of $L^2(\mathbb{S}^N_+, \theta_{N+1}^{1-2s})$ constructed in  Proposition \ref{propautovalori}, we consider the corresponding Fourier coefficients associated with $U$, i.e. the functions $\varphi_j(\tau)$ defined in \eqref{phi}
for every $j\geq 1 $ and $\tau\in (0,1)$.
Therefore, for every  $\tau\in (0,1)$, 
\begin{equation}\label{espansionediU}
    U(\tau\theta)= \sum_{j\geq 1} \varphi_j(\tau)\psi_j(\theta)\quad\text{in $L^2(\mathbb{S}^N_+, \theta_{N+1}^{1-2s})$}.
\end{equation}
\begin{lemma} 
Let $\mu_{j_0} $ be as in Lemma \ref{lemmaconvblow}. Let $m\in \mathbb{N}\setminus \{0\}$ be the multiplicity of $\mu_{j_0} $, where $j_0\geq1$ is chosen as in \eqref{eq:multiplicity}. Then, for every $R\in (0,1)$
and  $j=j_0, \dots , j_0+m-1$,
\begin{align}\label{sviluppocoeff}
\varphi_{j}(\tau) =& \,\tau^\gamma\left(\frac{\varphi_j(R)}{R^\gamma} + \frac{N+\gamma-2s}{N+2\gamma-2s}\int_\tau^R t^{-N-1+2s-\gamma}\Upsilon_j(t)\, dt+\frac{\gamma R^{-N+2s-2\gamma}}{N+2\gamma-2s} \int_0^R t^{\gamma-1}\Upsilon_j(t)\, dt\right)\\
&\notag + O(\tau^{\gamma+2s-\frac{N}{p}})\quad \text{as $\tau \to 0^+$},
\end{align}
where, for every $\tau\in (0,1)$ and $j=j_0, \dots , j_0+m-1$, $\Upsilon_j(\tau)$ is defined in \eqref{U}.
\end{lemma}
\begin{proof}
Let  $j\in\{j_0, \dots , j_0+m-1\}$.  From \eqref{eq1EXT} and \eqref{probautov} it follows that $\varphi_j$ solves 
    \begin{equation}\label{ode}
        - \varphi''_j(\tau)-\frac{N+1-2s}{\tau}\varphi'_j(\tau)+ \frac{\mu_j }{\tau^2}\varphi_j(\tau)= \zeta_j(\tau)\quad \text{in the sense of distributions in $(0,1)$},
    \end{equation}
    where, for every $\tau\in (0,1)$, 
    \begin{equation*}\label{zeta}
       \zeta_j(\tau):= \kappa_s  \tau^{2s-2}\int_\omega h (\tau\theta') \mathop{\rm Tr}U(\tau\theta') \mathop{\rm Tr}\psi_j(\theta')\, dS' .
    \end{equation*}
    Being $\Upsilon_j(\tau)$ as in \eqref{U},  notice  that
    \begin{equation*}\label{relazzeta}
        \zeta_j(\tau) = \tau^{2s-N-1} \Upsilon'_j(\tau)\quad \text{in the sense of distributions in $(0,1)$},
    \end{equation*}
which, in view of \eqref{relazione}, allows us to rewrite \eqref{ode} as follows
\begin{equation*}
   -\big( \tau^{N+1+2\gamma-2s}(\tau^{-\gamma}\varphi_j)'\big)'=\tau^\gamma \Upsilon'_j(\tau)\quad \text{in the sense of distributions in $(0,1)$}. 
\end{equation*}
    A double integration of the above equation over $(\tau, R)$, for any $R\in (0,1)$  and $\tau\in(0,R)$, leads us to conclude that, for any fixed $R\in (0,1)$, 
    there exists $c_j(R)\in\mathbb{R}$ such that, for all $\tau \in (0,R)$,
\begin{align}\label{primaversione}
        \varphi_j(\tau)= \tau^\gamma&\left( \frac{\varphi_j(R)}{R^{\gamma}}-\frac{\gamma c_j (R)R^{-N-2\gamma+2s}}{N+2\gamma-2s} + \frac{N+\gamma-2s}{N+2\gamma-2s}\int_\tau^R t ^{-N+2s-\gamma-1}\Upsilon_j(t)\, dt\right)\\
        \notag& + \frac{\gamma\tau^{-N+2s-\gamma}}{N+2\gamma-2s}\left(c_j(R)+ \int_\tau^R t^{\gamma-1}\Upsilon_j(t)\, dt\right).
    \end{align}
    We claim that, for every $R\in(0,1)$, 
\begin{equation}\label{L1}
       t\mapsto  t^{-N+2s-\gamma-1} \Upsilon_j(t)\in L^1(0,R).
    \end{equation} 
     To this aim, applying H\"{o}lder's inequality we estimate $|\Upsilon_j(t)|$ as follows 
\begin{equation}\label{stimaupsilon}
    |\Upsilon_j(t)|\leq \kappa_s \sqrt{\int_{\mathcal C\cap B'_t} |h(x)||\mathop{\rm Tr}U(x)|^2\, dx}\cdot \sqrt{\int_{\mathcal C\cap B'_t} |h(x)|\left|\mathop{\rm Tr}\psi_j\big(\tfrac{x}{|x|}\big)\right|^2\, dx}.
\end{equation}  
Using  \eqref{usefulineq2} and recalling the definition of the functions $D$, $H$ and $\mathcal N$ given in \eqref{D}, \eqref{H} and \eqref{N} respectively, for every $t\in (0,R_0)$ we have 
\begin{align}\label{stimaupsilon1}
  \int_{\mathcal C\cap B'_t} |h(x)||\mathrm{Tr}U(x)|^2\, dx &\leq V_N^{\frac{2sp-N}{Np}}\Vert h\Vert _{L^p(\mathcal C \cap B'_1)} t^{\frac{2sp-N}{p}} \left(\int_{\mathcal C\cap B'_t }|\mathop{\rm Tr}U(x)|^{2^\ast(s)}\, dx\right)^{\frac{2}{2^\ast(s)}} \\
  \notag& \leq V_N^{\frac{2sp-N}{Np}}\Vert h\Vert _{L^p(\mathcal C \cap B'_1)} \frac{\tilde S_{N,s}}{\tilde C} t^{\frac{2sp-N}{p}} t^{N-2s}\left(D(t)+\frac{N-2s}{2}H(t)\right)\\
  \notag &= V_N^{\frac{2sp-N}{Np}}\Vert h\Vert _{L^p(\mathcal C \cap B'_1)} \frac{\tilde S_{N,s}}{\tilde C} t^ {\frac{ N(p-1)}{p}}H(t)\left(\mathcal N(t)+\frac{N-2s}{2}\right).
\end{align}
In a similar way, we can estimate the second term in \eqref{stimaupsilon} as follows
\begin{align}\label{stimaupsilon2}
         \int_{\mathcal C\cap B'_t} |h(x)|\left|\mathop{\rm Tr}\psi_j\big(\tfrac{x}{|x|}\big)\right|^2\, dx &\leq   V_N^{\frac{2sp-N}{Np}}\Vert h\Vert _{L^p(\mathcal C \cap B'_1)} t^{\frac{2sp-N}{p}} \left(\int_{B'_t }
     \left|\mathop{\rm Tr}\psi_j\big(\tfrac{x}{|x|}\big)\right|^{2^\ast(s)}\, dx\right)^{\frac{2}{2^\ast(s)}} \\
    \notag &= V_N^{\frac{2sp-N}{Np}}\Vert h\Vert _{L^p(\mathcal C \cap B'_1)}  t^{\frac{2sp-N}{p}} t^{\frac{2N}{2^\ast(s)}}  \left(\int_{B'_1} 
     \left|\mathop{\rm Tr}\psi_j\big(\tfrac{x}{|x|}\big)\right|^{2^\ast(s)} \, dx\right)^{\frac{2}{2^\ast(s)}}\\
     \notag &= V_N^{\frac{2sp-N}{Np}}\Vert h\Vert _{L^p(\mathcal C \cap B'_1)}   t^{\frac{N(p-1)}{p}} \left(\int_{B'_1} 
     \left|\mathop{\rm Tr}\psi_j\big(\tfrac{x}{|x|}\big)\right|^{2^\ast(s)} \, dx\right)^{\frac{2}{2^\ast(s)}}.
    \end{align}
From \eqref{stimaupsilon}, \eqref{stimaupsilon1}, and \eqref{stimaupsilon2}, taking into account  \eqref{Nlimitdallalto} and \eqref{prima}, we infer that 
\begin{equation*}
   |\Upsilon_j(t)|\leq  \mathop{\rm const}t^{\gamma+\frac{N(p-1)}{p}}\quad\text{for all }t\in (0,R_0),
\end{equation*}
and then, since $\Upsilon_j$ is bounded in $(0,1)$, 
\begin{equation*}
   |\Upsilon_j(t)|\leq  \mathop{\rm const}t^{\gamma+\frac{N(p-1)}{p}}\quad\text{for all }t\in (0,1),
\end{equation*}
for some $\mathrm{const}>0$ independent of $t$. It follows that, for every $R\in (0,1)$,
\begin{equation}\label{stimaupsilon3}
    \int_0^R t ^{-N+2s-\gamma-1}|\Upsilon_j(t)|\, dt\leq \mathrm{const}\, R^{2s-\frac{N}{p}},
\end{equation}
for some $\mathrm{const}>0$ independent of $R$. \eqref{L1} is thereby proved. 

We observe that, for every $R\in (0,1)$, 
\begin{equation}\label{ugualeazero}
    c_j(R) + \int_0^R t^{\gamma-1} \Upsilon_j(\tau)\, dt =0.
\end{equation}
Indeed, 
the function  $t\mapsto  t ^{\gamma-1}\Upsilon_j(t)$ is integrable in a neighbourhood of $0$, due to \eqref{L1} and the fact that 
\begin{equation}\label{esponenti}
 \gamma>-\frac{N-2s}2;
\end{equation}
the latter inequality is a consequence of \eqref{relazione} and Proposition \ref{propautovalori}, which ensures that $\mu_{j_0}>-(\frac{N-2s}2)^2$. Then, if \eqref{ugualeazero} were not true, in view of \eqref{primaversione} we would have $\varphi_j(\tau)\sim \mathop{\rm const}\tau^{-N+2s-\gamma}$ as $\tau\to0^+$ for some $\mathop{\rm const}\neq 0$. This would contradict the fact that, by the Parseval identity,
\begin{equation*}
|\varphi_j(\tau)|^2\leq \int_{\mathbb{S}^N_+} \theta_{N+1}^{1-2s}|U(\tau\theta)|^2\, dS,
\end{equation*}
and hence, in view of  \cite{FalFel14}*{Lemma 2.4},
\begin{equation*}
    \int_0^{1}\tau^{N-1-2s}|\varphi_j(\tau)|^2\,d\tau\leq \int_{B_{1}^+}t^{1-2s}\frac{U^2(z)}{|z|^2}\,dz<+\infty.
\end{equation*}
Then \eqref{ugualeazero} is satisfied.

Once \eqref{ugualeazero} is proved, we can conclude that 
\begin{equation}\label{opiccolo}
    \tau^{-N+2s-\gamma}\left(c_j(R)+\int_\tau^R t^{\gamma-1}\Upsilon_j(t)\, dt\right)=O(\tau^{\gamma+2s-\frac{N}{p}})\quad \text{as $\tau\to 0^+$}.
\end{equation}
Indeed, by \eqref{stimaupsilon3} and \eqref{esponenti}, we have that, for all $\tau\in (0,1)$,
\begin{equation*}
\begin{split}
    \biggl|\tau^{-N+2s-\gamma}&\left(c_j(R)+\int_\tau^R t^{\gamma-1}\Upsilon_j(t)\, dt\right)\biggr| = \biggl| \tau^{-N+2s-\gamma} \int_0^\tau t^{\gamma-1}\Upsilon_j(t)\, dt\biggr|\\
    &\leq  \tau^{-N+2s-\gamma} \int_0^\tau t^{2\gamma+N-2s}\cdot t^{-N+2s-\gamma-1}|\Upsilon_j(t)|\, dt \leq \mathrm{const}\,\tau^{\gamma+2s-\frac{N}{p}}.  
    \end{split}
\end{equation*}
Hence \eqref{opiccolo} is proved. Substituting \eqref{ugualeazero} and \eqref{opiccolo} into \eqref{primaversione}, we 
finally obtain \eqref{sviluppocoeff}.  
\end{proof}
\begin{lemma}
We have 
    \begin{equation}\label{illimiteepositivo}
        \lim_{r\to 0^+}\frac{H(r)}{r^{2\gamma}}>0,
    \end{equation}
    being $\gamma$ defined in \eqref{illimitegamma}. 
\end{lemma}
\begin{proof}
    Recalling the definition of $H$ given in \eqref{H}, using \eqref{espansionediU} and the Parseval identity, we have that 
    \begin{equation}\label{Hcomesomma}
        H(\tau)= \sum_{j\geq 1} |\varphi_j(\tau)|^2.
    \end{equation}
    We suppose by contradiction that \eqref{illimiteepositivo} does not hold true: by  \eqref{Hcomesomma}, this implies that  
    \begin{equation*}
        \lim_{\tau\to 0^+} \frac{\varphi_j(\tau)}{\tau^\gamma}=0\quad \text{for every $j=j_0, \dots, j_0+m-1$},  
    \end{equation*}
    which in turn, combined with \eqref{sviluppocoeff}, implies that, for every $j=j_0, \dots, j_0+m-1$ and $R\in (0,1)$,
    \begin{equation*}
        \frac{\varphi_j(R)}{R^\gamma} + \frac{N+\gamma-2s}{N+2\gamma-2s}\int_0^R t^{-N-1+2s-\gamma}\Upsilon_j(t)\, dt+\frac{\gamma R^{-N+2s-2\gamma}}{N+2\gamma-2s} \int_0^R t^{\gamma-1}\Upsilon_j(t)\, dt =0.
    \end{equation*}
    Plugging the above identity  into \eqref{sviluppocoeff}, we obtain that,  for every $j=j_0, \dots, j_0+m-1$,  
    \begin{equation*}
        \varphi_j(\tau)= -\tau^\gamma\frac{N+\gamma-2s}{N+2\gamma-2s} \int_0^\tau t^{-N-1+2s-\gamma}\Upsilon_j(t)\, dt + O(\tau^{\gamma+2s-\frac{N}{p}})\quad \text{as $\tau\to 0^+$}. 
    \end{equation*}
    From this, \eqref{phi}, and \eqref{stimaupsilon3} it follows that, for every $j=j_0, \dots, j_0+m-1$, 
    \begin{equation*}
        (U(\tau\cdot), \psi_j)_{L^2(\mathbb{S}^N_+, \theta_{N+1}^{1-2s})}= O(\tau^{\gamma+2s-\frac{N}{p}}) \quad \text{as $\tau\to 0^+$}. 
    \end{equation*}
 Using \eqref{seconda} with $\varepsilon=2s-\frac Np$ to estimate from below $\sqrt{H(\tau)}$, from  \eqref{wlambda} it follows that, for every $j=j_0, \dots, j_0+m-1$, 
    \begin{equation}\label{lassurdo}
     (w^\tau, \psi_j)_{L^2(\mathbb{S}^N_+, \theta_{N+1}^{1-2s})}= O(\tau^{s-\frac{N}{2p}})
     =o(1)
     \quad \text{as $\tau\to 0^+$}.    
    \end{equation}
    On the other hand,  by Lemma \ref{lemmaconvblow} and continuity of the trace map \eqref{compact}, for every  sequence $\tau_n\to 0^+$, there exist a subsequence $\{\tau_{n_{k}}\}_{k\in\mathbb{N}}$  and a  function $\psi\in \mathrm{span }\{\psi_j: j=j_0, \dots, j_0+m-1\}$ satisfying \eqref{lapsihanorma1} such that 
    $w^{\tau_{n_k}}\to \psi$ in $L^2(\mathbb{S}^N_+, \theta_{N+1}^{1-2s})$ as $k\to \infty$. This would imply the convergence
    \begin{equation*}
      (w^{\tau_{n_k}}, \psi)_{L^2(\mathbb{S}^N_+, \theta_{N+1}^{1-2s})} \to \Vert \psi\Vert ^2_{L^2(\mathbb{S}^N_+, \theta_{N+1}^{1-2s})} =1\quad \text{as $k\to\infty$},
    \end{equation*}
    which contradicts \eqref{lassurdo}. 
\end{proof}  
We are now ready to  prove the asymptotics at the vertex of the cone, stated in Theorems \ref{asimestens} and \ref{asimu}.

\begin{proof}[Proof of Theorem \ref{asimestens}]
Let $\mu_{j_0}$ be as in Lemma \ref{lemmaconvblow}, with $j_0\geq 1$ chosen as in \eqref{eq:multiplicity} and $m\in\mathbb{N}\setminus \{0\}$ being the multiplicity of $\mu_{j_0}$. Let $\{\psi_j\}_{j=j_0,\dots, j_0+m-1}$ be an orthonormal basis of the eigenspace associated with the eigenvalue $\mu_{j_0} $. By Lemma \ref{lemmaconvblow}, for any sequence $\tau_n\to 0^+$ there exist a subsequence $\{\tau_{n_k}\}_{k\geq 1}$ and a vector $(\beta_{j_0}, \dots, \beta_{j_0+m-1})\in \mathbb{R}^m\setminus\{0\}$  such that 

 \begin{equation}\label{eq:convUsubs}
     \frac{U(\tau_{n_k} z)}{\tau_{n_k}^{\gamma}}\to |z|^\gamma\sum _{j=j_0}^{j_0+m-1}\beta_j \psi_j\left(\frac{z}{|z|}\right)\quad \text{in $H^1(B^+_1,t^{1-2s})$ as $k\to \infty$},  
 \end{equation}
where we have taken into account also \eqref{limdaprovmagdizero} and \eqref{illimiteepositivo}. To prove \eqref{risultatofinale-ext}, it is enough to prove that the coefficients $\beta_j$ above satisfy \eqref{betah}: indeed, showing that the limit depends nor on the sequence $\{\tau_n\}$ neither on its subsequence, the thesis follows from Urysohn's subsequence principle. To this purpose, it is sufficient to notice that, by \eqref{phi} and \eqref{eq:convUsubs}, for every $\ell=j_0,\dots,j_0+m-1$, 
\begin{equation*}
    \lim_{k\to \infty}\frac{\varphi_{\ell}(\tau_{n_{k}})}{\tau_{n_k}^{\gamma}}=\sum_{j=j_0}^{j_0+m-1}\beta_j\int_{\mathbb{S}^N_+} \theta_{N+1}^{1-2s} \psi_j(\theta)\psi_{\ell}(\theta)\, dS=\beta_{\ell},
\end{equation*}
whereas  \eqref{sviluppocoeff} yields
\begin{equation*}
    \lim_{k\to \infty}\frac{\varphi_\ell(\tau_{n_k})}{\tau_{n_k}^{\gamma}}= \frac{\varphi_\ell(R)}{R^\gamma} + \frac{N+\gamma-2s}{N+2\gamma-2s}\int_0^R t^{-N-1+2s-\gamma}\Upsilon_\ell(t)\, dt+\frac{\gamma R^{-N+2s-2\gamma}}{N+2\gamma-2s} \int_0^R t^{\gamma-1}\Upsilon_\ell(t)\, dt
\end{equation*}
for every $R\in (0,1)$. This completes the proof.  
\end{proof}

\begin{proof}[Proof of Theorem \ref{asimu}]
  The proof follows from Theorem \ref{asimestens}. Indeed, if $u\in \mathcal{D}^{s,2}(\mathbb{R}^N)$ is a non-trivial weak  solution to \eqref{eq1}, we  can consider its extension $U=\mathcal H(u)\in \mathcal{D}^{1,2}(\mathbb{R}^{N+1}_+, t^{1-2s})$. 
  We observe that, since $\mathop{\rm Tr}U\not\equiv0$ in $\R^N$,   $U\not\equiv 0$ in $\R^{N+1}_+$; then $U\not\equiv 0$ in $B_1^+$ by classical unique continuation principles for elliptic operators, see e.g. \cite{GarLin86}.  Therefore, we can apply
  to $U$  Theorem \ref{asimestens}. The conclusion finally follows exploiting  the continuity of the trace map \eqref{eq:traccia-H1}.
\end{proof}

\section*{Acknowledgements}
The authors are members of the GNAMPA research group of INdAM - Istituto Nazionale di Alta Matematica.
A. De Luca and S. Vita are partially supported by the INDAM-GNAMPA
    2023 grant ``Regolarit\`{a} e singolarit\`{a} in problemi con frontiere libere'' CUP E53C22001930001 and the INDAM-GNAMPA
    2024 grant ``Nuove frontiere nella capillarit\`{a} non locale'' CUP E53C23001670001.
    S. Vita is partially supported by the MUR funding for Young Researchers - Seal of Excellence (ID: SOE\_0000194. Acronym: ADE. Project Title: Anomalous diffusion equations: regularity and geometric properties of solutions and free boundaries). 
    V. Felli is partially supported by the MUR-PRIN project no. 20227HX33Z 
``Pattern formation in nonlinear phenomena'' granted by the European Union -- Next Generation EU.

\end{document}